%% file: CDII.tex
\documentclass[11pt,twoside,a4paper]{amsart}
\input{settings}

\usepackage{threeparttable}

\newcommand{\subG}{\mathsf{subG}}
\newcommand{\bU}{\bar{U}}
\DeclareMathOperator{\tr}{tr}
\DeclareMathOperator{\sta}{sta}
\DeclareMathOperator{\app}{app}
\DeclareMathOperator{\err}{err}

\begin{document}

\title[Current density impedance imaging with PINNs]{Current density impedance imaging with PINNs}
\author{Chenguang Duan$^1$, Yuling Jiao$^{1}$, Xiliang Lu$^{1*}$ and Jerry Zhijian Yang$^1$}
\address{1 School of mathematics and statistics, Wuhan University
}
\email{$^*$ Corresponding author: xllv.math@whu.edu.cn}

\begin{abstract}
In this paper, we introduce CDII-PINNs, a computationally efficient method for solving CDII using PINNs in the framework of Tikhonov regularization. This method constructs a physics-informed loss function by merging the regularized least-squares output functional with an underlying differential equation, which describes the relationship between the conductivity and voltage. A pair of neural networks representing the conductivity and voltage, respectively, are coupled by this loss function. Then, minimizing the loss function provides a reconstruction. A rigorous theoretical guarantee is provided. We give an error analysis for CDII-PINNs and establish a convergence rate, based on prior selected neural network parameters in terms of the number of samples. The numerical simulations demonstrate that CDII-PINNs are efficient, accurate and robust to noise levels ranging from $1\%$ to $20\%$.
\end{abstract}

\maketitle

\section{Introduction}

Electrical Impedance Tomography (EIT) is a medical imaging technique to recover the conductivity value of soft tissue from pairs of flux/voltage on its boundary. However, the problem of reconstructing the electrical conductivity distribution is severely ill-posed, which is very sensitive to noise and measurement errors. To avoid the drawback of EIT, people consider other conductivity imaging techniques that combine electrical impedance tomography (EIT) and magnetic resonance imaging (MRI). In this work, we consider the current density impedance imaging (CDII), one of the medical imaging techniques, which aims to image the current density distribution within a conductive medium \cite{Nachman2009Recovering}. CDII has potential applications in a variety of fields, including non-destructive testing, material science, and biomedical imaging. For example, it can be used in non-destructive testing of concrete to detect cracks and voids, in material science to study the microstructure of metals, and in biomedical imaging to detect tumors and other abnormalities in tissues.

Theoretical results on the uniqueness of the recovery and conditional stability of CDII have been studied in \cite{Nachman2009Recovering, hoell2014current} and \cite{montalto2013stability, lopez2020stability}, respectively. Numerical reconstruction algorithms have been extensively studied, including the level set method \cite{nachman2007conductivity, tamasan2015stable}, the weighted least-gradient method \cite{Nachman2009Recovering}, and methods based on linearized reconstruction \cite{yazdanian2021numerical}. These algorithms have been shown to be effective in reconstructing the current density distribution from measured data. As a classical regularization method, Tikhonov regularization is one of the efficient models to solve  CDII, but the numerical solver to minimize the Tikhonov functional may be completed due to solving forward partial differential equations repeatedly in the iterative methods. To overcome the drawback of iterative methods for minimizing the Tikhonov functional, we propose a deep solver based on Physics-Informed Neural Networks (PINNs) to the Tikhonov regularization problem. Compared to traditional iterative methods, our approach requires only roughly two forward problems in total. This is in contrast to these traditional methods, which require solving two forward problems in each iteration.

In recent years, Physics-Informed Neural Networks (PINNs) \cite{Raissi2019Physics} have gained a lot of attention as a residual-based deep learning method for solving both forward and inverse problems in science and engineering. In \cite{Raissi2019Physics}, the authors explored numerically the recovery of constant functions in PDEs using PINNs. Building on this work, we propose to use PINNs to identify the parameters in CDII by constructing a residual-based loss function that involves two neural networks. These networks are used to simultaneously approximate the unknown conductivity coefficient function and the underlying solution. For solving inverse problems in PDEs, other deep learning methods have also been proposed, such as those in \cite{Bao2020Numerical, Jin2022imaging,zhang2023stability,Mishra2021Estimates}.
Here, we highlight the main differences between these methods and ours.  In \cite{Bao2020Numerical}, a min-max loss function is induced by the weak form of the elliptic equations. This method requires an additional network to approximate the dual variable, which can make the training process more challenging than ours. On the other hand, in \cite{Jin2022imaging}, the authors propose to handle the CDII problem by solving a  forward problem using deep learning under the framework of weighted least-gradient \cite{Nachman2009Recovering}.  In the paper \cite{zhang2023stability}, the authors investigate the inverse source problem within the framework of PINNs. However, their analysis differs significantly from ours, which focuses on consistency from the perspective of statistical learning theory. A related work, \cite{Mishra2021Estimates}, considers the error analysis of PINNs for certain inverse problems. Nevertheless, this study does not provide guidance on how to choose the neural network parameters to ensure consistency.
\par Our main contributions are as follows.  
\begin{itemize}
\item  We introduce CDII-PINNs, a computationally efficient method for solving CDII using PINNs in the framework of nonparametric statistical learning \cite{Gine2015Mathematical}. Our numerical simulations demonstrate that CDII-PINNs achieve high accuracy and robustness to noise levels ranging from $1\%$ to $20\%$.
\item
We present a comprehensive error analysis for CDII-PINNs, along with a prior rule for selecting appropriate neural network parameters, such as width and depth, based on the number of samples. This approach ensures that the estimation error is consistent and can be controlled to achieve the desired level of precision. To the best of our knowledge, this is the first theoretical result of its kind for PINNs in the context of inverse problems.
\end{itemize}
\par The remainder of the paper is organized as follows. In \cref{notations}, we introduce the notations used throughout this paper. Our CDII-PINNs and their convergence analysis are presented in \cref{method} and \cref{errana}, respectively. In \cref{simu}, we provide numerical simulations. The conclusions and main proofs are presented in \cref{conc} and \cref{append}, respectively.

\section{Preliminary}\label{notations}
\subsection{Problem formulation}
Let $U\subseteq\mathbb{R}^{d}$ ($d>1$) denote a bounded non-empty $C^{\infty}$-domain, referring to the space occupied by the object. Without loss of generality, we assume $U\subseteq(0,1)^{d}$. The electric potential or voltage $u$ in the interior of the domain is governed by the following second-order elliptic equation
\begin{equation}\label{eq:elliptic:pde}
\nabla\cdot(\gamma\nabla u)=0\quad\text{in }U, 
\end{equation}
corresponding to the unknown conductivity coefficient $\gamma\in L^{\infty}(U)$. In the idealistic situation, one assumes that the potential is measured everywhere on the boundary, that is, the Dirichlet trace
\begin{equation}\label{eq:elliptic:bc}
u=f\quad\text{on }\partial U.
\end{equation}
The govern equation \cref{eq:elliptic:pde} together with the observations on the boundary \cref{eq:elliptic:bc} is a well-used mathematical model for EIT, called current density impedance imaging (CDII) \cite{Widlak2012Hybrid}. Denote $J=-\gamma\nabla u$ the internal current density vector. In CDII, the magnitude of the current density field $a(x)=|J(x)|$ is observed on the whole of $U$. The goal is to determinate the conductivity distribution from the noisy measurement data.
\par Moreover, we assume the conductivity $\gamma$ belongs to the admissible set $\calK$ defined as
\begin{equation}\label{eq:box}
\calK=\{\gamma\in L^{\infty}(U):\gamma_{0}\leq\gamma(x)\leq\gamma_{1}\text{ a.e. in } U\},
\end{equation}
with $0<\gamma_{0}<\gamma_{1}<\infty$.
\par Given a conductivity function $\gamma$, denote by $u_{\gamma}$ the  solution of \cref{eq:elliptic:pde,eq:elliptic:bc}. The CDII problems can be considered as a nonlinear operator equation
\begin{equation}\label{eq:CDII:operator}
\calG[\gamma]=a^{\dagger},
\end{equation}
where $\calG:\gamma\mapsto\gamma|\nabla u_{\gamma}|$ is a nonlinear operator. In practical applications, one has only access to the noisy data $a^{\delta}$ satisfying $\|a^{\dagger}-a^{\delta}\|\leq\delta$, where $a^{\dagger}=\calG[\gamma^{\dagger}]$.
\subsection{Neural networks}
A neural network $\phi:\mathbb{R}^{N_{0}}\rightarrow\mathbb{R}^{N_{\calD+1}}$ is a function defined by 
\begin{equation*}
\phi(x)=T_{\calD}(\varrho(T_{\calD-1}(\cdots\varrho(T_{0}(x))\cdots))),
\end{equation*}
where the activation function $\varrho$ is applied component-wisely and $T_{\ell}(x):=A_{\ell}x+b_{\ell}$ is an affine transformation with $A_{\ell}\in\bbR^{N_{\ell+1}\times N_{\ell}}$ and $b_{\ell}\in\bbR^{N_{\ell}}$ for $\ell=0,\ldots\calD$. In this paper, we consider the case $N_{0}=d$ and $N_{\calD+1}=1$. The number $\calD$ the depth of the neural network. Denote $\calS_{i}=\sum_{\ell=1}^{i}(\|A_{\ell}\|_{0}+\|b_{\ell}\|_{0})$ for $i=1,\cdots,\calD$ as the number of nonzero weights on the first $i$-layers and then $\calS=\calS_{\calD}$ is the total number of nonzero weights.
\par We define the neural network class $\mathcal{N}_{\varrho}(\calD,\calS,\calB)$ as the collection of $\varrho$-neural networks with at most $\calD$ layers and at most $\calS$ nonzero weights and each weight are bounded by $\calB$.

\section{Method}\label{method}
\par Denote by $Y=a^{\delta}(X)$ for $X\in U$ the measurement data. Consider an additive $\delta^{2}$-sub-Gaussian noise model
\begin{equation}\label{eq:noise:model}
Y=\calG[\gamma^{\dagger}](X)+\xi, \quad\xi\sim^{i.i.d.}\subG(0,\delta^{2}),
\end{equation}
where the measurement noise $\xi$ is independent of $X$ and $\gamma^{\dagger}$.
By Tikhonov regularization, we can obtain a reconstruction by solving the optimization problem
\begin{equation}\label{eq:J:gamma}
\min_{\gamma}\Big\{\calJ(\gamma):=\bbE_{(X,Y)}\big[(Y-\calG[\gamma](X))^{2}\big]+\alpha\psi(\gamma)\Big\},
\end{equation}
where $\psi$ is the regularization functional. Observe that \cref{eq:J:gamma} is equivalent to a optimal control problem
\begin{equation*}
\min_{\gamma,u}J(\gamma,u)=\bbE_{(X,Y)}\big[(Y-\calG[\gamma](X))^{2}\big]+\alpha\psi(\gamma),\quad\text{subject to \cref{eq:elliptic:pde,eq:elliptic:bc}},
\end{equation*}
which can be converted to an unconstrained minimization problem by the penalty method
\begin{equation}\label{eq:population:risk}
\min_{\gamma,u}L(\gamma,u):=J(\gamma,u)+G(\gamma,u).
\end{equation}
For the sake of simplicity, we choose the penalty term as 
\begin{equation*}
G(\gamma,u)=\|\nabla\cdot(\gamma\nabla u)\|_{L^{2}(U)}^{2}+\|Tu-f\|_{L^{2}(\partial U)}^{2},
\end{equation*}
which is the PINNs-type loss function \cite{Raissi2019Physics} for solving the constraint equations \cref{eq:elliptic:pde,eq:elliptic:bc} given $\gamma$. Furthermore, since 
\begin{equation*}
\bbE_{(X,Y)}\big[(Y-\calG[\gamma](X))^{2}\big]
=\|\calG[\gamma^{\dagger}]-\calG[\gamma]\|_{L^{2}(U)}^{2}+\delta^{2},\quad\text{for each}~\gamma:U\rightarrow\bbR,
\end{equation*}
it follows that $J(\gamma^{\dagger},u^{\dagger})\leq J(\gamma,u)$ for each $(\gamma,u)$ if $\alpha=0$. In addition, $G(\gamma,u)\geq0$ for any $(\gamma,u)$ and $G(\gamma^{\dagger},u^{\dagger})=0$, which means
\begin{equation*}
(\gamma^{\dagger},u^{\dagger})\in\argmin_{\gamma,u}L(\gamma,u),
\end{equation*}
provided $\alpha=0$.
\par However, in inverse problems, the distribution of $(X,Y)$ is typically unknown and only a random sample $S=\{(X_{i},Y_{i})\}_{i=1}^{n}\cup\{\bar{X}_{i}\}_{i=1}^{n}$ is available, where $\{(X_{i},Y_{i})\}_{i=1}^{n}$ are $n$ independent copies of $(X,Y)$ and $\{\bar{X}_{i}\}_{i=1}^{n}$ are $n$ independent random variables drawn from the uniform distribution on $\partial U$. Based on the measurement data, we employ the Monte Carlo method to discretize the population risk \cref{eq:population:risk} and yield the empirical risk
\begin{equation}\label{eq:empirical:risk}
L_{n}(\gamma,u):=J_{n}(\gamma,u)+G_{n}(\gamma,u),
\end{equation} 
where the empirical objective functional and penalty term is given by
\begin{align*}
J_{n}(\gamma,u)&=\frac{1}{n}\sum_{i=1}^{n}(Y_{i}-\gamma|\nabla u|(X_{i}))^{2}+\alpha\psi_{n}(\gamma), \quad\text{and} \\
G_{n}(\gamma,u)&=\frac{1}{n}\sum_{i=1}^{n}(\nabla\cdot(\gamma\nabla u)(X_{i}))^{2}+(Tu(\bar{X}_{i})-f(\bar{X}_{i}))^{2},
\end{align*}
respectively. Here $\psi_{n}(\gamma)$ is a discrete version of $\psi(\gamma)$, satisfying $\bbE_{S}\psi_{n}(\gamma)=\psi(\gamma)$ for any fixed function $\gamma:U\rightarrow\bbR$.
\par Finally, we represent the conductivity $\gamma$ and voltage $u$ by two dependent neural networks, respectively and then substitute them into the empirical risk \cref{eq:empirical:risk}. An estimator of $(\gamma^{\dagger},u^{\dagger})$ can be obtained by empirical risk minimization, which reads
\begin{equation}
(\hat{\gamma}_{n},\hat{u}_{n})\in\argmin_{(\gamma,u)\in\calF_{\gamma}\times\calF_{u}}L_{n}(\gamma,u):=J_{n}(\gamma,u)+G_{n}(\gamma,u),
\end{equation}  
where $\calF_{\gamma}$ and $\calF_{u}$ are two neural network classes chosen by the users.
\par To summarize, our approach constructs a loss function by merging the measurement data with an underlying physical model. A pair of neural networks representing conductivity and voltage, respectively, are coupled by this physics-informed loss function. Therefore, we call our method as Physics-Informed Neural Networks for Current Density Impedance Imaging (CDII-PINNs). A detailed reconstruction procedure is shown in \cref{alg:CDII}.
\par\noindent
\begin{minipage}{\textwidth}
\renewcommand*\footnoterule{}
\renewcommand{\thempfootnote}{\fnsymbol{mpfootnote}}
\begin{algorithm}[H]
\caption{CDII-PINNs.}
\label{alg:CDII}
\begin{algorithmic}[1]
\Require A measurement data set $S=\{(X_{i},Y_{i})\}_{i=1}^{n}\cup\{\bar{X}_{i}\}_{i=1}^{n}$.
\State Construct neural networks $(\gamma_{\phi},u_{\theta})$, parameterized by $(\phi,\theta)$.
\State Initialize parameters $(\phi,\theta)$ randomly.
\For {$k=1:K$}
\State \textcolor{white!30!black}{\texttt{\# Compute the loss (or empirical risk) based on the data set.}}
\State {$L_{n}(\gamma_{\phi},u_{\theta})=J_{n}(\gamma_{\phi},u_{\theta})+G_{n}(\gamma_{\phi},u_{\theta})$ as defined in \cref{eq:empirical:risk}.}
\State \textcolor{white!30!black}{\texttt{\# Back propagation:compute the gradient of the loss w.r.t.$(\phi,\theta)$.}}
\State {$(g_{\phi},g_{\theta})=\nabla_{(\phi,\theta)}L_{n}(\gamma_{\phi},u_{\theta})$.}
\State \textcolor{white!30!black}{\texttt{\# Update $(\phi,\theta)$ by some SGD-type algorithm.}}
\State {$(\phi,\theta)\leftarrow\operatorname{SGD}\big\{(\phi,\theta),(g_{\phi},g_{\theta}),\tau\big\}$.}\footnote{Denote by $\operatorname{SGD}\{\texttt{parameters},\texttt{gradient},\texttt{learning rate}\}$ a step of SGD update.}
\EndFor
\Ensure {$\tilde{\gamma}_{n}=\gamma_{\phi}$, $\tilde{u}_{n}=u_{\theta}$.}
\end{algorithmic}
\end{algorithm}
\end{minipage}
\par In the classical methods for solving CDII, such as \cite{Nachman2009Recovering}, to ensure that the reconstructions of conductivity and voltage are admissible, one must solve the forward problem after each update of the reconstructions, which is costly. In contrast, CDII-PINNs need not fit the measurement data and solve the forward problems alternatively, as the loss of which considers the measurement data and the PDE constraints simultaneously.

\section{Error analysis}\label{errana}
In this section, we present an error estimate to the cost functional \eqref{eq:population:risk}. For simplicity we consider the case with regularization parameter $\alpha=0$. Then the estimator $(\tilde{\gamma}_{n},\tilde{u}_{n})\in\calF_{\gamma}\times\calF_{u}$ can be evaluated via its expected excess risk
\begin{equation*}
R(\tilde{\gamma}_{n},\tilde{u}_{n})=\mathbb{E}_{S}\Big[\|\tilde{\gamma}_{n}|\nabla\tilde{u}_{n}|-\gamma^{\dagger}|\nabla u^{\dagger}|\|_{L^{2}(U)}^{2}\Big]+\mathbb{E}_{S}\Big[G(\tilde{\gamma}_{n},\tilde{u}_{n})\Big],
\end{equation*}
where the first term measures the difference between the magnitude of the recovered current density field and that of the ground truth current density field, and the second term reflects how the estimator satisfies the PDE constraint.
\begin{assumption}\label{ass:fg:bound}
Assume $f$ is bounded, i,e., $\|f\|_{H^{3/2}(\partial U)}\leq B_{f}$ with $B_{f} \geq 1$.
\end{assumption}
\begin{assumption}[Boundedness]\label{ass:gammau:bound}
Assume $\gamma^{\dagger}$, $u^{\dagger}$ and functions in $\calF_{\gamma}$, $\calF_{u}$ are bounded, which means
\begin{enumerate}[(i)]
\item $\{\gamma^{\dagger}\}\cup\calF_{\gamma}\subseteq\{\gamma:\|\gamma\|_{W^{1,\infty}(U)}\leq B_{\gamma}\}$ with $B_{\gamma}\geq1$, and
\item $\{u^{\dagger}\}\cup\calF_{u}\subseteq\{u:\|u\|_{W^{2,\infty}(U)}\leq B_{u}\}$ with $B_{u}\geq1$.
\end{enumerate}
\end{assumption}
To measuring the complexity or compactness of a subset of a metric space in a quantitative way, we introduce the covering number and the metric entropy.
\begin{definition}[Covering number and metric entropy] 
Let $(S,d)$ be a metric space, and $T\subseteq S$. A set $T_{\varepsilon}\subseteq S$ is called an $\varepsilon$-cover of $T$ if and only if for each $t\in T$, there exits $t_{\varepsilon}\in T_{\varepsilon}$, such that $d(t,t_{\varepsilon})\leq\varepsilon$. Moreover, 
\begin{equation*}
N(\varepsilon,T,d):=\inf\Big\{|T_{\varepsilon}|:\text{$T_{\varepsilon}$ is an $\varepsilon$-cover of $T$}\Big\}
\end{equation*}
is called the $\varepsilon$-cover number of $T$, and $H(\varepsilon,T,d)=\log N(\varepsilon,T,d)$ is called the $\varepsilon$ metric entropy of $T$.
\end{definition}
\par For each estimator $(\tilde{\gamma}_{n},\tilde{u}_{n})$ taking values in $\calF_{\gamma}\times\calF_{u}$, we define the corresponding quantity 
\begin{equation*}
\Delta_{n}(\tilde{\gamma}_{n},\tilde{u}_{n}):=\mathbb{E}_{S}\Big[L_{n}(\tilde{\gamma}_{n},\tilde{u}_{n})-L_{n}(\hat{\gamma}_{n},\hat{u}_{n})\Big],
\end{equation*}
which measures the difference between the expected empirical risk of $(\tilde{\gamma}_{n},\tilde{u}_{n})$ and the minimum over all functions in the hypothesis class $\calF_{\gamma}\times\calF_{u}$. It is obvious that $\Delta_{n}(\tilde{\gamma}_{n},\tilde{u}_{n})\geq0$ and $\Delta_{n}(\tilde{\gamma}_{n},\tilde{u}_{n})=0$ if and only if $\Delta_{n}(\tilde{\gamma}_{n},\tilde{u}_{n})$ is an empirical risk minimizer.
\par With the help of the preceding notations, we can decompose the expected excess risk as follows.
\begin{lemma}[Error decomposition]\label{lemma:err:dec}
Under the noise model \cref{eq:noise:model}. Suppose \cref{ass:gammau:bound,ass:fg:bound} are fulfilled. Then it holds for each estimator $(\tilde{\gamma}_{n},\tilde{u}_{n})$ taking values in $\calF_{\gamma}\times\calF_{u}$ that
\begin{align*}
R(\tilde{\gamma}_{n},\tilde{u}_{n})&\lesssim\inf_{\varepsilon>0}\Big\{(B_{\gamma}^{2}B_{u}^{2}+B_{f}^{2}+\delta^{2})(H_{\gamma}+H_{u})n^{-1}+(\delta+B_{\gamma}B_{u})(B_{\gamma}+B_{u})\varepsilon\Big\} \\
&\quad+\inf_{(\gamma,u)\in\calF_{\gamma}\times\calF_{u}}R(\gamma,u)+\Delta_{n}(\tilde{\gamma}_{n},\tilde{u}_{n}),
\end{align*}
provided $\varepsilon>0$ and $n$ large enough, where
\begin{equation*}
H_{\gamma}^{\varepsilon}=\log N(\varepsilon,\calF_{\gamma},\|\cdot\|_{W^{1,\infty}(U)})\quad\text{and}\quad H_{u}^{\varepsilon}=\log N(\varepsilon,\calF_{u},\|\cdot\|_{W^{2,\infty}(U)}).
\end{equation*}
\end{lemma}
\par The proof refers to \cref{sec:appendix:errdec}. By \cref{lemma:err:dec}, the expected excess risk can be decomposed into three parts: approximation error, optimization error, and statistical error. The approximation error, define as $\inf_{(\gamma,u)\in\calF_{\gamma}\times\calF_{u}}R(\gamma,u)$, is the minimum of the expected excess risk over all functions in the hypothesis class, measuring the expression power of the functions in $\calF_{\gamma}\times\calF_{u}$. The optimization error $\Delta_{n}(\tilde{\gamma}_{n},\tilde{u}_{n})$ measures the difference between the estimator and the empirical risk minimizer. The remaining part is defined as the statistical error, which is caused by the discretization of the population risk. By choosing an appropriate $\varepsilon$, we can find that the statistical error converges to zero as the sample size $n\rightarrow\infty$. Furthermore, the convergence rate $\calO(\frac{1}{n})$ obtained in \cref{lemma:err:dec} is an improvement of $\calO(\frac{1}{\sqrt{n}})$ in \cite{duan2022convergence,Jiao2022rate,jiao2021error,Jin2022imaging}.
\par In this paper, we ignore the optimization error, and focus on the approximation error, the statistical error and the trade-off between them.
\subsection{Approximation error}
\par In this part, we bound the approximation error. By the following lemma, it is sufficient to estimate the approximation error of functions by neural networks.
\begin{lemma}\label{lemma:app:1}
Suppose \cref{ass:gammau:bound} is fulfilled. Then
\begin{equation*}
\inf_{(\gamma,u)\in\calF_{\gamma}\times\calF_{u}}R(\gamma,u)\leq C_{\app}\cdot B_{\gamma}^{2}B_{u}^{2}\Big(\inf_{\gamma\in\calF_{\gamma}}\|\gamma-\gamma^{\dagger}\|_{C^{1}(\bU)}+\inf_{u\in\calF_{u}}\|u-u^{\dagger}\|_{H^{2}(U)}\Big),
\end{equation*}
where $C_{\app}$ is a constant depending on $\gamma_{0}$, $\gamma_{1}$, $U$ and $B_{f}$.
\end{lemma}
\par The proof refers to \cref{sec:appendix:proof:approx}.
\par We now turn to investigate the function approximation problems. To this end, we first introduce a type of activation functions.
\begin{definition}[Exponential PU-admissible function \cite{Guhring2021Approximation}]
Let $j\in\bbN$ and $\tau\in\{0,1\}$. We say that a function $\varrho:\bbR\rightarrow\bbR$ is exponential $(j,\tau)$-PU-admissible, if 
\begin{enumerate}[(a)]
\item $\varrho$ is bounded if $\tau=0$, and $\varrho$ is Lipschitz continuous if $\tau=1$;
\item There exists $R>0$ such that $\varrho\in C^{j}(\bbR\backslash[-R,R])$;
\item $\varrho^{\prime}\in W^{j-1,\infty}(\bbR)$, if $j\geq 1$;
\item There exists $A=A(\varrho),B=B(\varrho)\in\bbR$ with $A<B$, some $C=C(\varrho,j)>0$ and some $D=D(\varrho,j)>0$ such that  
\begin{itemize}
\item[(d.1)] $|B-\varrho^{(\tau)}(x)|\leq C\exp(-Dx)$ for all $x>R$;
\item[(d.2)] $|A-\varrho^{(\tau)}(x)|\leq C\exp(Dx)$ for all $x<-R$;
\item[(d.3)] $|\varrho^{(k)}(x)|\leq C\exp(-D|x|)$ for all $x\in\bbR\backslash[-R,R]$ and all $k=\tau+1,\ldots,j$.
\end{itemize}
\end{enumerate}
\end{definition}
Several commonly-used exponential PU-admissible activation functions are shown in \cref{tab:exp:pu:activation}.
\begin{table}[H]
\caption{Commonly-used exponential PU-admissible activation functions.}
\centering
\begin{threeparttable}
\begin{tabular}{lcccccc}
\toprule
Name & Definition & PU-Decay $(j,\tau)$ & $\calB_{\varrho,0}$ & $\calB_{\varrho,1}$ & $\calB_{\varrho,2}$ & $\calB_{\varrho,3}$ \tnote{*} \\
\midrule
\texttt{sigmoid} & $\frac{1}{1+e^{-x}}$ & $(j,0)$ for any $j\in\bbN$ & 1 & 1/4 & 8/9 & $\leq 1$ \\
\texttt{tanh} & $\frac{e^{x}-e^{-x}}{e^{x}+e^{-x}}$ & $(j,0)$ for any $j\in\bbN$ & 1 & 1 & 1 & 2/3 \\
\texttt{softplus} & $\ln(1+e^{x})$ & $(j,1)$ for any $j\in\bbN$ &- &- &- & -\\
\texttt{swish} & $\frac{x}{1+e^{-x}}$ & $(j,1)$ for any $j\in\bbN$ &- &- &- & -\\
\bottomrule
\end{tabular}
\label{tab:exp:pu:activation}
\begin{tablenotes}
\footnotesize
\item[*] $|\varrho(x)|\leq\calB_{\varrho,0}$, $|\varrho^{\prime}(x)|\leq\calB_{\varrho,1}$, $|\varrho^{\prime\prime}(x)|\leq\calB_{\varrho,2}$ and $|\varrho^{\prime\prime\prime}(x)|\leq\calB_{\varrho,3}$ for each $x\in\bU$.
\end{tablenotes}
\end{threeparttable}
\end{table}
\par Applying Proposition 4.8 in \cite{Guhring2021Approximation} and \cref{lemma:app:1}, we have the follow approximation error estimate.
\begin{lemma}[Approximation error]\label{lemma:approx}
Let $d,s\in\mathbb{N}_{+}$ and $U\subseteq(0,1)^{d}$ be a domain. Suppose $\gamma^{\dagger}\in C^{s+1}(\bU)$ and $u^{\dagger}\in H^{s+2}(U)$. Suppose $\calD_{\gamma},\calS_{\gamma},\calD_{u},\calS_{u}\in\bbN_{+}$, and $\varrho$ is an exponential PU-admissible activation function. Suppose \cref{ass:fg:bound,ass:gammau:bound} are fulfilled. Set
\begin{enumerate}[(i)]
\item $\calF_{\gamma}=\calN_{\varrho}(\calD_{\gamma},\calS_{\gamma},\calB_{\gamma})$ with $\calD_{\gamma}=C\log(d+s)$ and $\calB_{\gamma}=C\calS_{\gamma}^{\frac{2s}{d}+7}$, and
\item $\calF_{u}=\calN_{\varrho}(\calD_{u},\calS_{u},\calB_{u})$ with $\calD_{u}=C\log(d+s+1)$ and $\calB_{u}=C\calS_{u}^{\frac{2s+2}{d}+7}$,
\end{enumerate}
where the constant $C$ depends on $d$, $s$, $\mu$ and $U$. Then
\begin{equation*}
\inf_{(\gamma,u)\in\calF_{\gamma}\times\calF_{u}}R(\gamma,u)\leq C_{\app}\cdot B_{\gamma}^{2}B_{u}^{2}\Big\{\|\gamma^{\dagger}\|_{C^{s+1}(\bU)}\cdot\calS_{\gamma}^{-\frac{s-\mu}{d}}+\|u^{\dagger}\|_{H^{s+2}(U)}\cdot\calS_{u}^{-\frac{s+1-\mu}{d}}\Big\},
\end{equation*}
where $C_{\app}$ is a constant depending on $d$, $s$, $\mu$, $\gamma_{0}$, $\gamma_{1}$, $U$ and $B_{f}$.
\end{lemma}
The proof refers to \cref{sec:appendix:proof:approx}. \cref{lemma:approx} shows that as long as the number of parameters is large enough, the neural networks can approximate the ground truth conductivity and voltage with arbitrarily small errors. In addition, when the size of the neural networks is the same, the smoother the ground truth conductivity and voltage are, the smaller the approximation error will be.
\subsection{Statistical error}
\par An estimate of the statistical error is presented in this section. With reference to the definition of statistical error in \cref{lemma:err:dec}, an upper bound on the statistical error will be obtained if we can estimate the metric entropy of the hypothesis classes$\calF_{\gamma}$ and $\calF_{u}$. In \cite{anthony1999neural}, the author gives the $L^{\infty}$ covering number bounds for some commonly used classes of neural networks. However, in PDE-related problems, we need to estimate not only the $L^{\infty}$ covering number of the set of networks, but also that of the set of network derivatives. In other words, we need to bound the $W^{m,\infty}$ ($m\geq1$) covering numbers of the set of networks. In \cite{duan2022convergence,Jiao2022rate}, the derivatives of ReLU$^{k}$-networks are implemented by ReLU-type networks, whose covering numbers can be bounded via \cite{anthony1999neural}. Unfortunately, this approach can not be extended to networks with other activation functions. In \cite{jiao2021error,Jin2022imaging}, the covering numbers of the derivatives of the networks are estimated by converting them to the covering numbers of the parameter class, which is a compact subset of $\bbR^{n}$, by Lipschitz continuity. This method requires some tedious manipulations. In this paper, we consider the classes of neural networks as relatively compact subsets of the smooth function spaces. We then estimate their metric entropy using the well-known entropy bounds for classes of smooth functions.
\begin{lemma}[Statistical error]\label{lemma:stat}
Let $\mu>0$, $s\in\bbN_{+}$ and $\calS\in\bbN_{+}$. Suppose $U\subset\subset(0,1)^{d}$ be a domain with $C^{\infty}$-boundary. Let $\varrho$ be an activation function which, together with its up to third-order derivatives, are continuous and uniformly bounded. Suppose \cref{ass:gammau:bound} is fulfilled. Set
\begin{enumerate}[(i)]
\item $\calF_{\gamma}=\calN_{\varrho}(\calD_{\gamma},\calS,\calB_{\gamma})$ with $\calD_{\gamma}=C\log(d+s)$ and $\calB_{\gamma}=C\calS^{\frac{2s}{d}+7}$, and
\item $\calF_{u}=\calN_{\varrho}(\calD_{u},\calS,\calB_{u})$ with $\calD_{u}=C\log(d+s+1)$ and $\calB_{u}=C\calS^{\frac{2s+2}{d}+7}$,
\end{enumerate}
where the constant $C$ depends on $d$, $s$, $\mu$ and $U$. Then
\begin{align*}
&\inf_{\varepsilon>0}\Big\{(B_{\gamma}^{2}B_{u}^{2}+B_{f}^{2}+\delta^{2})(H_{\gamma}+H_{u})n^{-1}+(\delta+B_{\gamma}B_{u})(B_{\gamma}+B_{u})\varepsilon\Big\} \\
&\leq C_{\sta}\cdot\calS^{\frac{6(4d+s+1)}{d+1}\log^{3}(d+s+1)}n^{-\frac{1}{d+1}},
\end{align*}
where $c$ is an absolute constant and $C_{\sta}$ is a constant depending on $d$, $s$, $\varrho$, $\delta$, $B_{\gamma}$, $B_{u}$ and $U$.
\end{lemma}
The proof refers to \cref{sec:appendix:proof:staerr}. Note that both \texttt{sigmoid} and \texttt{tanh} satisfy the conditions in \cref{lemma:stat}, as shown in \cref{tab:exp:pu:activation}.
\subsection{Convergence rate}
\par Up to now, we have bounded the approximation error and the statistical error, respectively. \cref{lemma:approx,lemma:stat} show that the approximation error will decrease with the size of the neural networks, while at the same time the statistical error will increase. By making a trade-off between the two, we obtain the following theorem.
\begin{theorem}[Convergence rate]\label{thm:rate}
Under the noise model \cref{eq:noise:model}. Let $\mu>0$, $d,s\in\mathbb{N}_{+}$ and $U\subset\subset(0,1)^{d}$ be a domain with $C^{\infty}$-boundary. Suppose $\gamma^{\dagger}\in C^{s+1}(\bU)$ and $u^{\dagger}\in H^{s+2}(U)$. Let $\varrho$ be an exponential PU-admissible activation function which, together with its up to third-order derivatives, are continuous and uniformly bounded. Suppose \cref{ass:fg:bound,ass:gammau:bound} are fulfilled. Set $\calF_{\gamma}=\calF_{u}=\calN_{\varrho}(\calD,\calS,\calB)$ with
\begin{equation*}
\calD=C(d,s,\mu,U),\quad\calS=\calO\Big(n^{\frac{1}{6(4d+s+1)\log^{3}(d+s+1)}}\Big)\quad\text{and}\quad\calB=\calO\Big(n^{\frac{s+7d}{3d(4d+s+1)\log^{3}(d+s+1)}}\Big).
\end{equation*}
Then it follows for each estimator $(\tilde{\gamma}_{n},\tilde{u}_{n})$ taking values in $\calF_{\gamma}\times\calF_{u}$ that
\begin{equation*}
R(\tilde{\gamma}_{n},\tilde{u}_{n})=\calO\Big(n^{-\frac{s-\mu}{7d(4d+s+1)\log^{3}(d+s+1)}}\Big).
\end{equation*}
\end{theorem}
\par Theoretical guarantees for deep learning-based methods for solving PDEs have attracted considerable attention. In \cite{duan2022convergence,Jiao2022rate,lu2022machine}, the convergence rate of deep Ritz methods and PINNs with ReLU$^{k}$ networks is presented. In particular, the minimax optimal rate is proved in \cite{lu2022machine}. An error estimate of deep Ritz method with different activation functions is given in \cite{jiao2021error}. In addition, \cite{lu2021Priori} gives an error analysis for deep Ritz methods for eigenvalue problems. However, limited work has been done on error analysis of deep learning based methods for inverse problems. In \cite{Jin2022imaging}, a convergence rate is proposed for a neural network method for solving CDII. To the best of our knowledge, \cref{thm:rate} is the first rigorous theoretical analysis of PINNs for reconstructing non-constant coefficients.
\par \cref{thm:rate} demonstrates the consistency of the expected excess risk, that is, both the approximation error and the statistical error in \cref{lemma:err:dec} vanish as the sample size $n\rightarrow\infty$, provided the number of parameters in networks is large enough. In addition, the convergence rate  $\simeq\calO(n^{-\frac{s}{c\cdot d^{2}}})$ depends on the dimension $d$ and the smoothness $s$ of the ground truth conductivity and voltage. We first consider the dependence of the convergence rate on $d$. Note that the rate $\calO(n^{-\frac{1}{c\cdot d^{2}}})$ in \cref{thm:rate} is slower than $\calO(n^{-\frac{1}{c\cdot d}})$ in \cite{Jin2022imaging}, which is because the statistical error bound in \cref{lemma:stat} is not sharp enough. However, the number of parameters required in this paper $\calO(n^{\frac{1}{c\cdot d}})$ is significantly smaller than that $\calO(n^{\frac{1}{c}})$ in \cite{Jin2022imaging} in high-dimensional problems. Furthermore, taking the smoothness into consideration, the rate $\simeq\calO(n^{-\frac{s}{c\cdot d^{2}}})$ shows that our approach may overcome the curse of dimensionality, provided the ground truth functions are smooth enough.

\begin{remark}
The analysis in this work focuses on the error estimation to the cost functional, which is applicable to any regularization parameter $\alpha$. If the error analysis to the conductivity $\gamma$ is needed, we need to choose the regularization parameter $\alpha$ properly, to ensure the stability of unknown conductivity from the cost functional with respect to the noise level. In that case, the number of samples should also depend on the noise level. We will leave the analysis to the  conductivity $\gamma$ in the future work. 
\end{remark}

\section{Numerical experiments and discussions}\label{simu}
\par Let $U=(0,1)^{2}$. We use the two-to-one voltage potential of $f$, which equals the trace of the harmonic function $u_{h}(x,y)=y$. Given an admissible pair $(\gamma,f)$, we solve numerically the problem
\begin{equation}\label{eq:numer}
\nabla\cdot(\gamma\nabla u)=0,\left.\quad u\right|_{\partial \Omega}=f.
\end{equation}
Once the solution $u$ is found, the interior data $a=\gamma|\nabla u|$ are computed. Then the noisy data $a^{\delta}$ is generated by adding Gaussian random noise 
\begin{equation*}
a^{\delta}(x)=a^{\dagger}(x)+\delta\cdot a^{\dagger}(x)\xi(x),\quad\xi(x)\sim N(0,1).
\end{equation*}
\par In all of our experiments, we parameterize networks $\gamma_{\phi}$ and $u_{\theta}$ as four-layer \texttt{tanh}-MLP with width 32 (or 64), of which the parameters are initialized by Xavier's method \cite{Glorot2010Understanding}. For each example, we use $n=100000$ interior measurement data points and $n=100000$ boundary points and run \cref{alg:CDII} for 50000 epochs with batch size 2048. We minimize the loss by ADAM \cite{kingma2015adam}, and the learning rate is set as $1.0\times10^{-3}$, which does not change during the training. Thanks to the robustness of our methods, there is no need for denoising and other preprocessing to the noisy measurements. We use the relative $L^{2}$-error $\err(\hat{\gamma})=\|\gamma^{\dagger}-\hat{\gamma}\|_{2}/\|\gamma^{\dagger}\|_{2}$ to measure the accuracy of the reconstruction $\hat{\gamma}$. In \cref{example:fourmode,example:dc}, we use $L^{2}$-regularization $\psi(\gamma)=\|\gamma\|_{L^{2}(U)}^{2}$ with regularization parameter $\alpha=1.0\times10^{-5}$, while both $L^{2}$ and TV regularization $\psi(\gamma)=|\nabla\gamma|_{L^{1}(U)}$ are applied in \cref{example:disjoint} with $\alpha=1.0\times10^{-3}$ and $\alpha=1.0\times10^{-3}$, respectively.
\par We solve \cref{eq:numer} by MATLAB 2022b. Our models are implemented by PyTorch 12.1 \cite{pytorch}, and trained with one NVIDIA Tesla V100 GPU.
\begin{example}\label{example:fourmode}
We employ the four-mode model conductivity distribution \cite{Nachman2009Recovering}
\begin{equation*}
\gamma^{\dagger}(x,y)=1+\gamma_{s}(x,y),
\end{equation*}
where $\gamma_{s}$ is a function with support in $U$, which is given by
\begin{equation*}
\begin{aligned}
\gamma_{s}(x,y)&=0.3 \cdot(a(x, y)-b(x, y)-c(x, y)), \\
a(x,y)&=0.3\cdot(1-3(2x-1))^{2}\cdot\exp[-9\cdot(2x-1)^{2}-(6y-2)^{2}], \\
b(x,y)&=\Big(\frac{3(2 x-1)}{5}-27 \cdot(2 x-1)^3-(3 \cdot(2 y-1))^5\Big)\exp[-(9\cdot(2x-1)^2+9\cdot(2y-1)^2)], \\
c(x,y)&=\exp[-(3\cdot(2x-1)+1)^{2}-9\cdot(2y-1)^{2}].
\end{aligned}
\end{equation*}
\end{example}
\par We first apply our proposed method to the reconstruction of conductivity and voltage from data with 1\% noise. \cref{fig:example1:iter} displays the reconstructed conductivity obtained by training the neural network for different epochs. Our results demonstrate that a rough approximation of the conductivity can be obtained after a few epochs, while further training is required to capture the fine-scale details. More exactly, we find that the neural network first learns to capture the low-frequency features, before fitting the high-frequency information. These findings align with the frequency principle \cite{John2020Frequency}.
\par \cref{tab:example1:error} presents the relative errors of recovered conductivity and voltage from data with different noise levels, showing that our method is robust to the noise. \cref{fig:example1:a} compares the measurement and recovered data at different noise levels. The reconstructed conductivity and voltage from the data at different noise levels and their point-wise absolute error are shown in \cref{fig:example1:gamma,fig:example1:u}.
\begin{figure}
\centering
\subfloat[\texttt{epoch=100}]
{\includegraphics[width=0.30\linewidth]{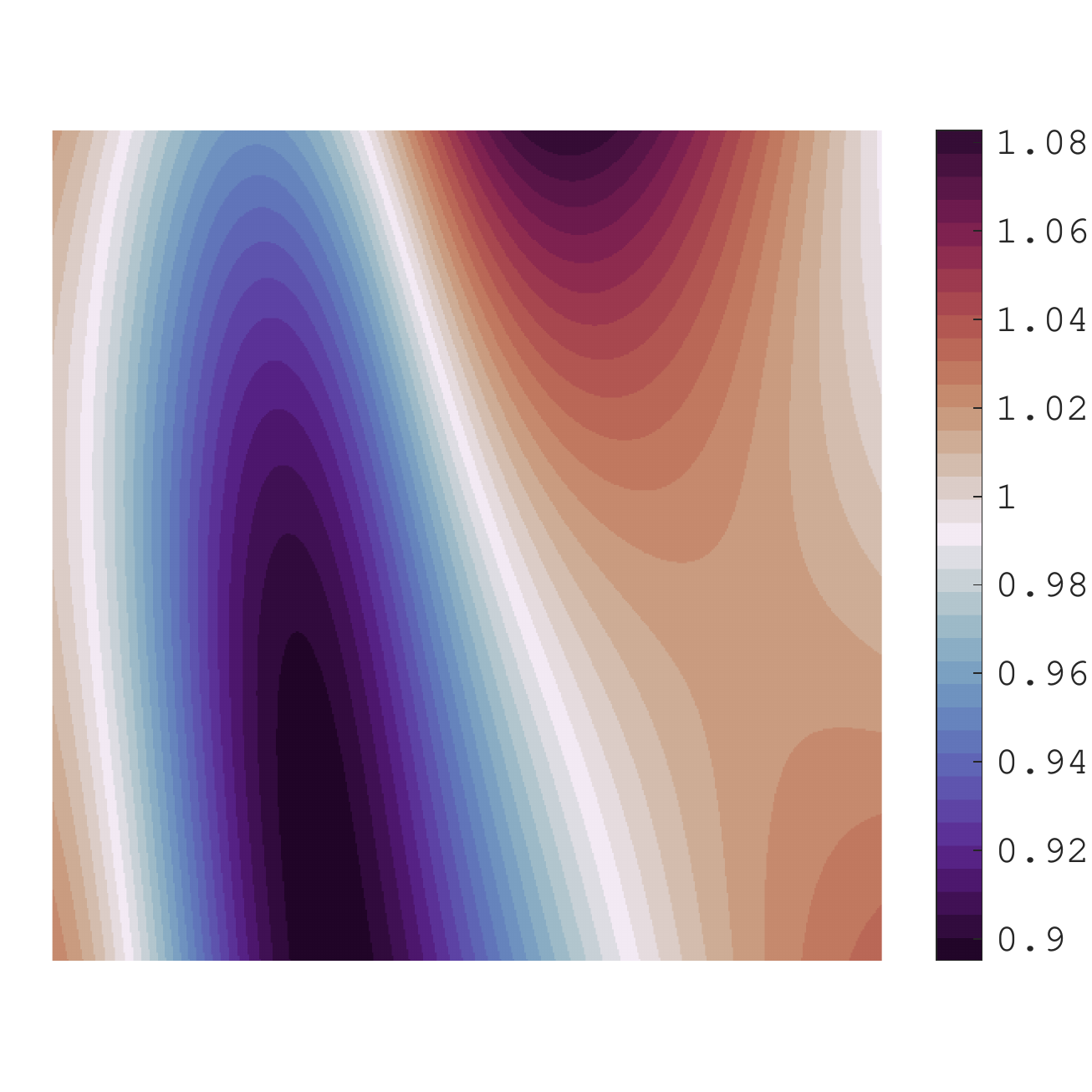}}
\subfloat[\texttt{epoch=500}]
{\includegraphics[width=0.30\linewidth]{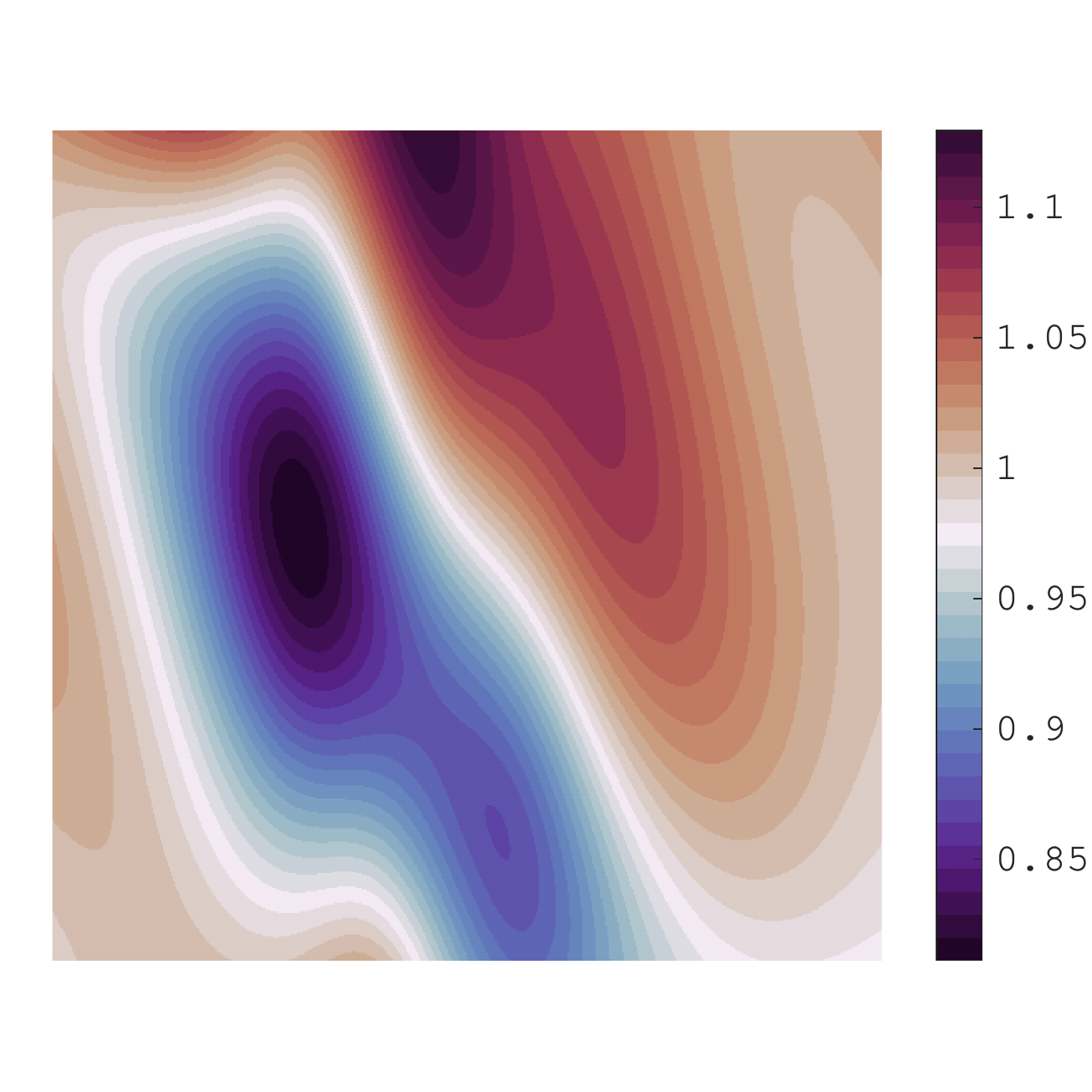}} 
\subfloat[\texttt{epoch=1000}]
{\includegraphics[width=0.30\linewidth]{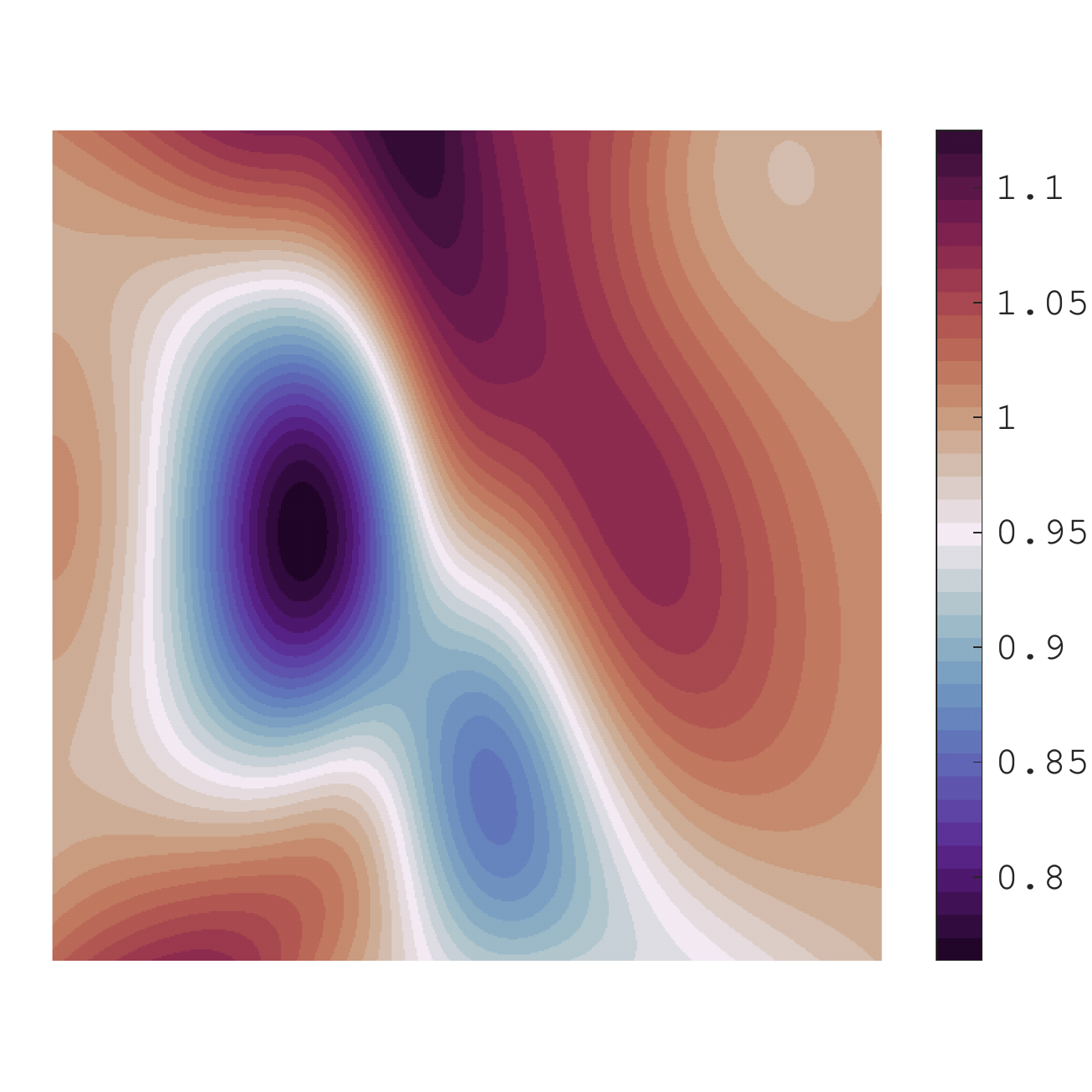}} 
\\
\subfloat[\texttt{epoch=5000}]
{\includegraphics[width=0.30\linewidth]{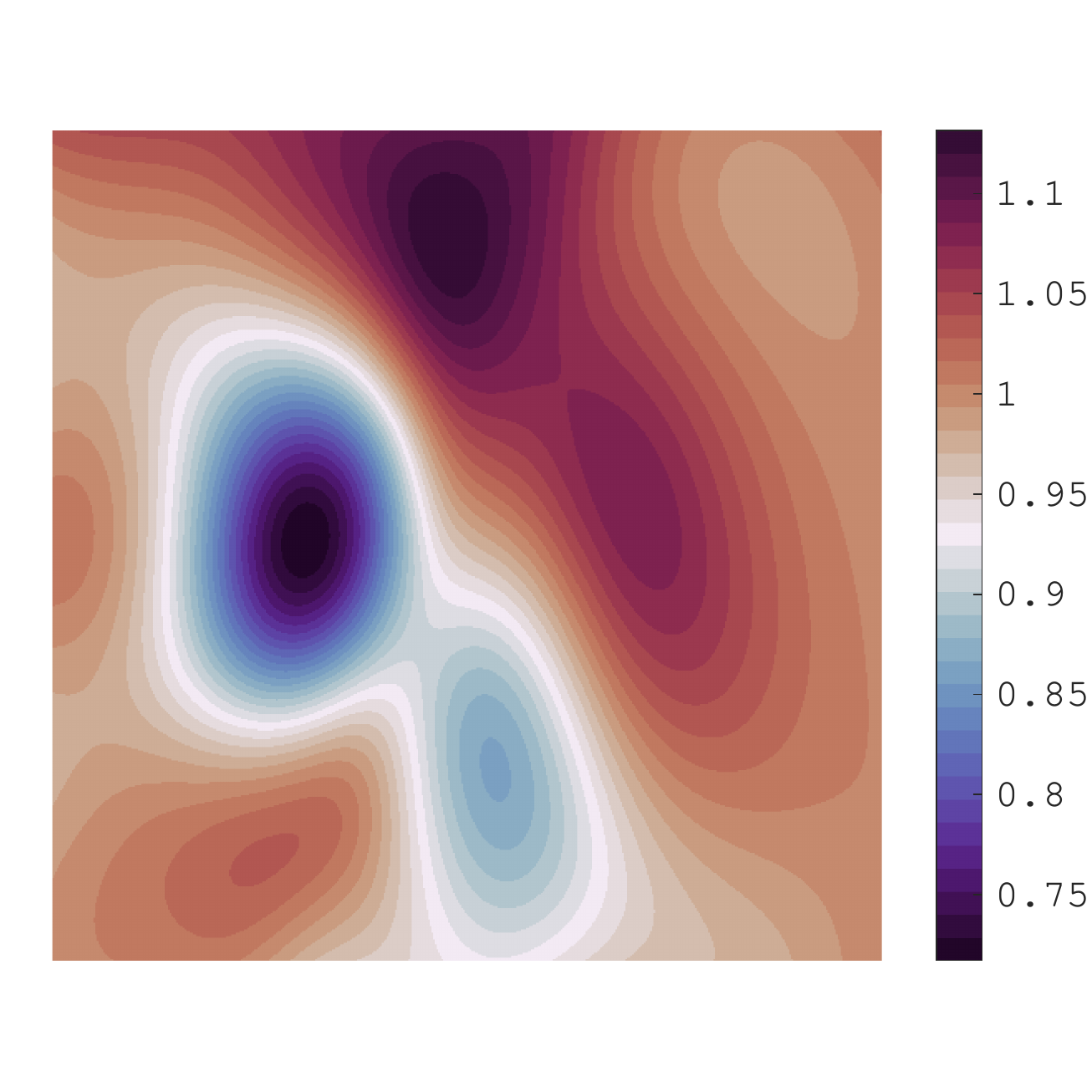}}
\subfloat[\texttt{epoch=10000}]
{\includegraphics[width=0.30\linewidth]{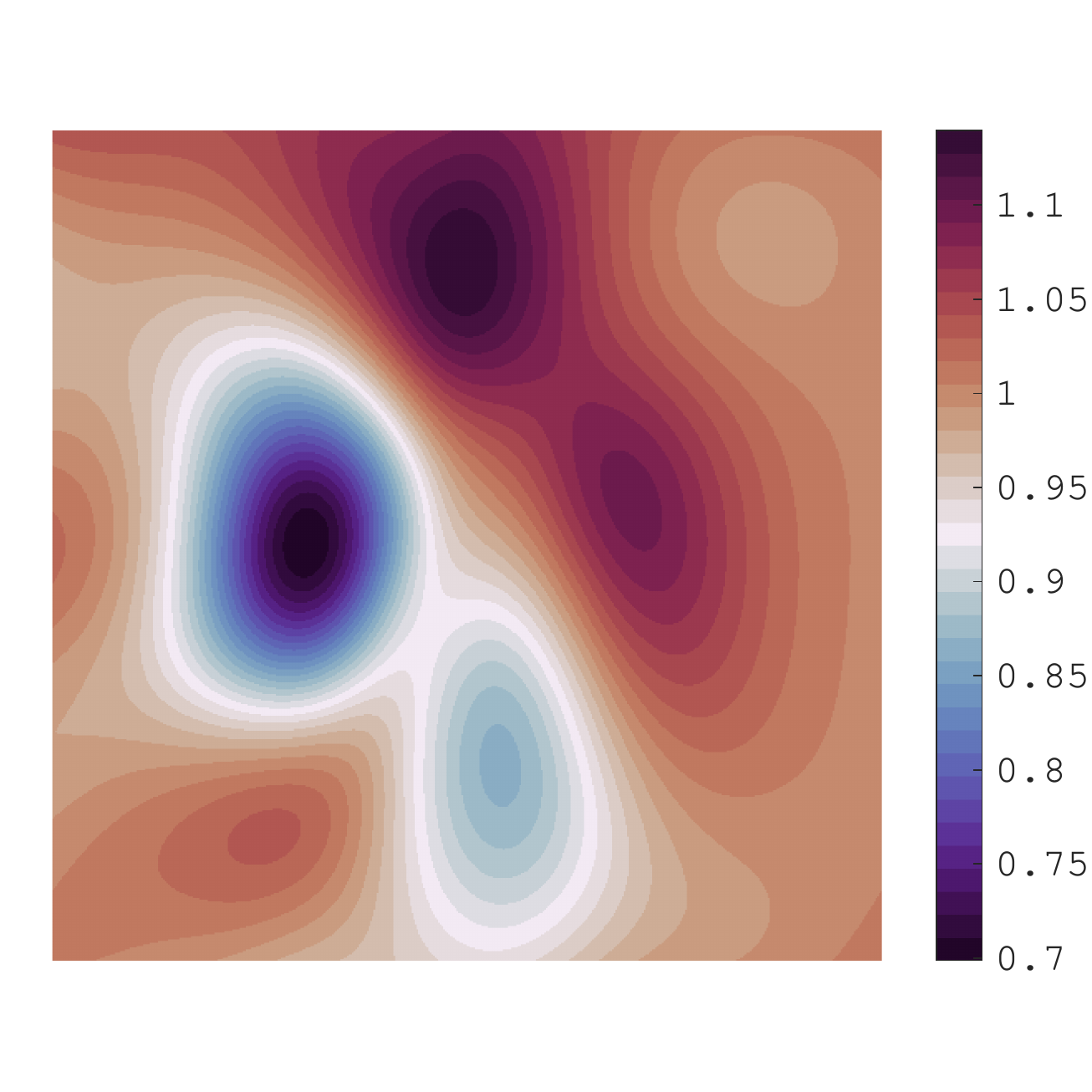}} 
\subfloat[\texttt{epoch=50000}]
{\includegraphics[width=0.30\linewidth]{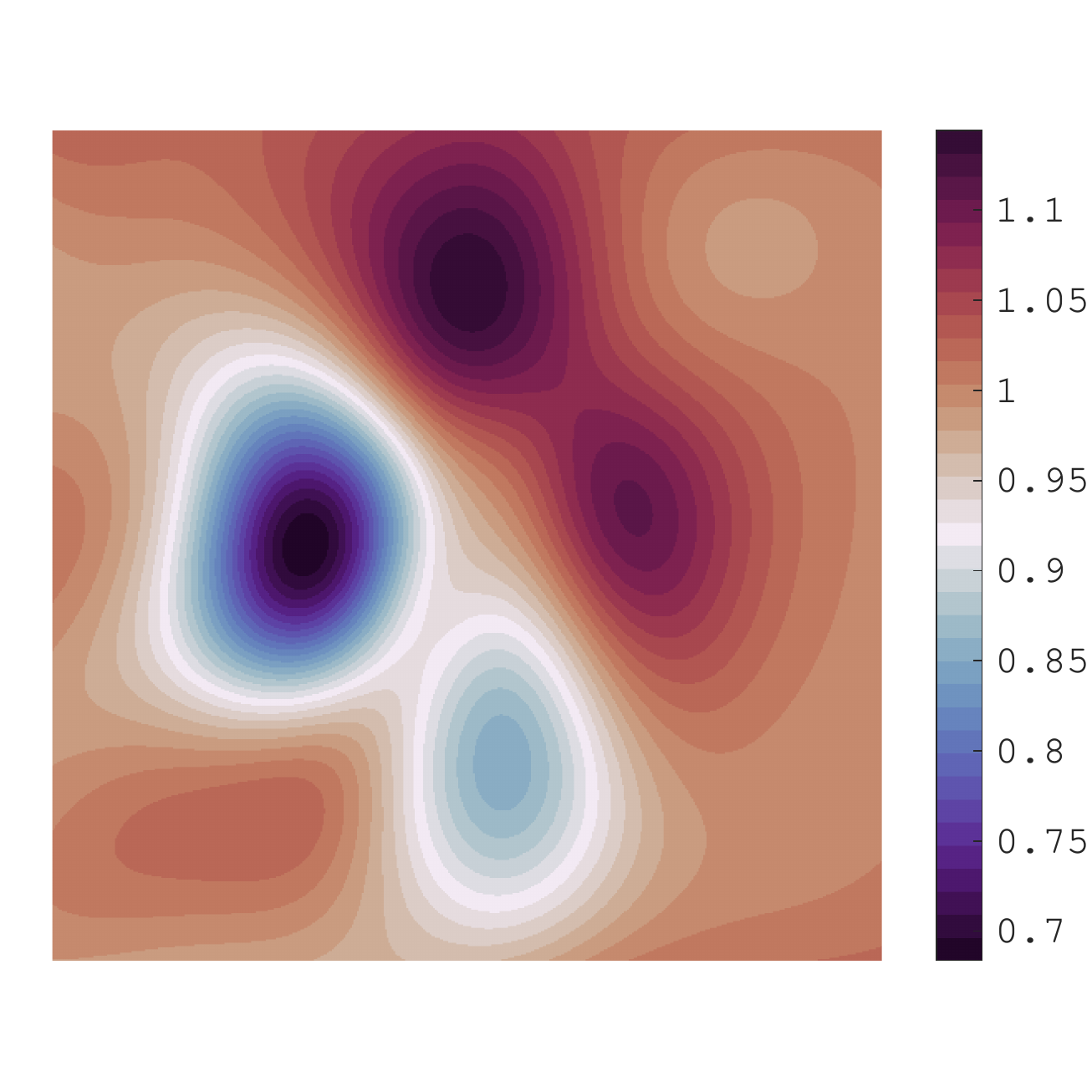}} 
\caption{The reconstructions of conductivity after different epochs. $\delta=1\%$}
\label{fig:example1:iter}
\end{figure}
\begin{table}
\caption{The relative $L^{2}$-error of the recovered conductivity $\hat{\gamma}$ and voltage $\hat{u}$ in \cref{example:fourmode}.}
\centering
\begin{tabular}{lccc}
\toprule
& \multicolumn{3}{c}{$\delta$}   \\
\cmidrule(lr){2-4}
 & $1\%$ & $10\%$ & $20\%$  \\
\cmidrule{1-4}
$\err(\gamma)$ & $2.86 \times 10^{-2}$ & $3.22 \times 10^{-2}$ & $3.75 \times 10^{-2}$  \\
$\err(u)$ & $3.03 \times 10^{-3}$ & $3.85 \times 10^{-3}$ & $4.22 \times 10^{-3}$  \\
\bottomrule
\end{tabular}
\label{tab:example1:error}
\end{table}
\begin{figure}
\centering
\subfloat[$a^{\dagger}$]
{\includegraphics[width=0.30\linewidth]{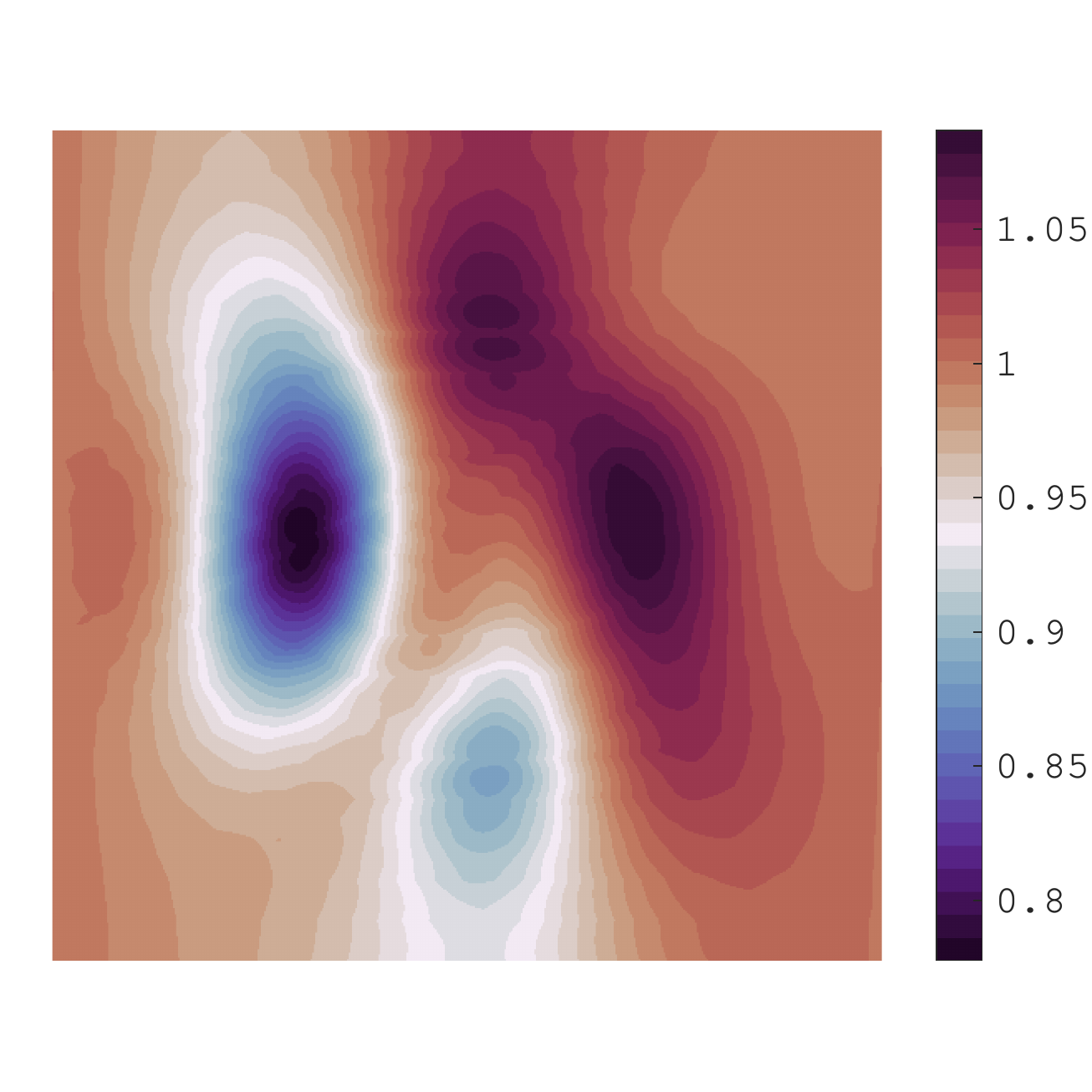}}
\subfloat[$a^{\delta}$]
{\includegraphics[width=0.30\linewidth]{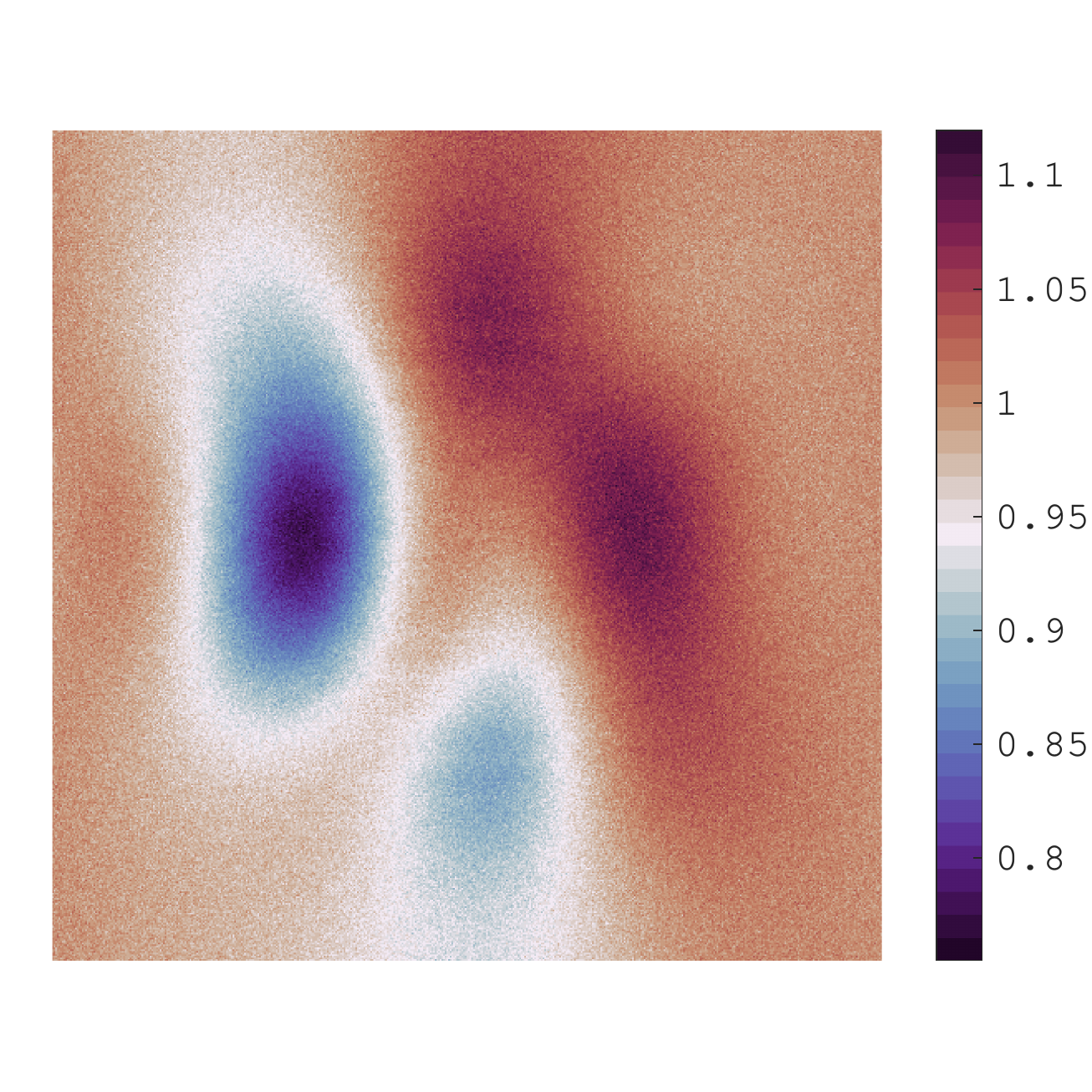}} 
\subfloat[$\hat{a}$]
{\includegraphics[width=0.30\linewidth]{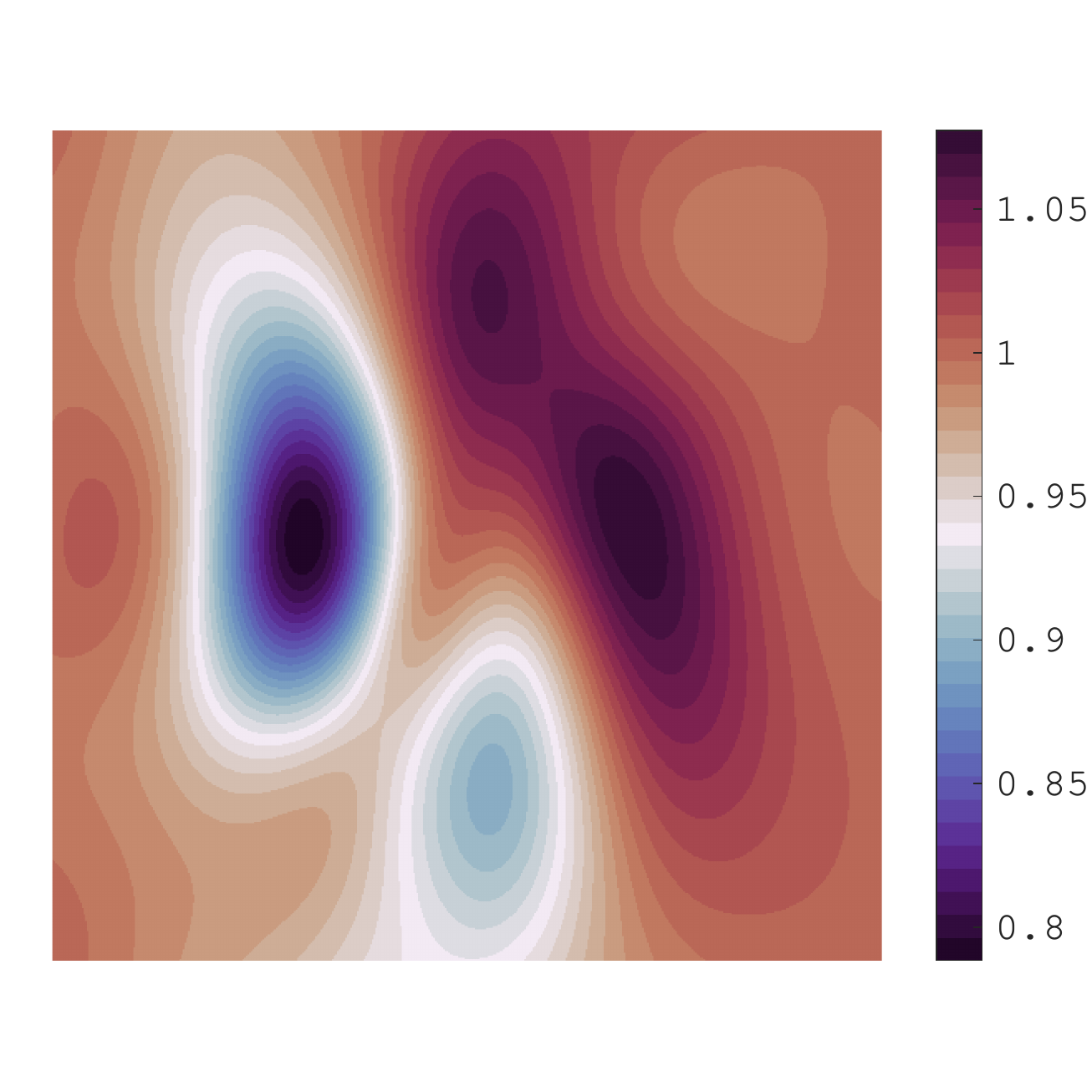}} 
\\
\subfloat[$a^{\dagger}$]
{\includegraphics[width=0.30\linewidth]{./figures/example01observations.pdf}}
\subfloat[$a^{\delta}$]
{\includegraphics[width=0.30\linewidth]{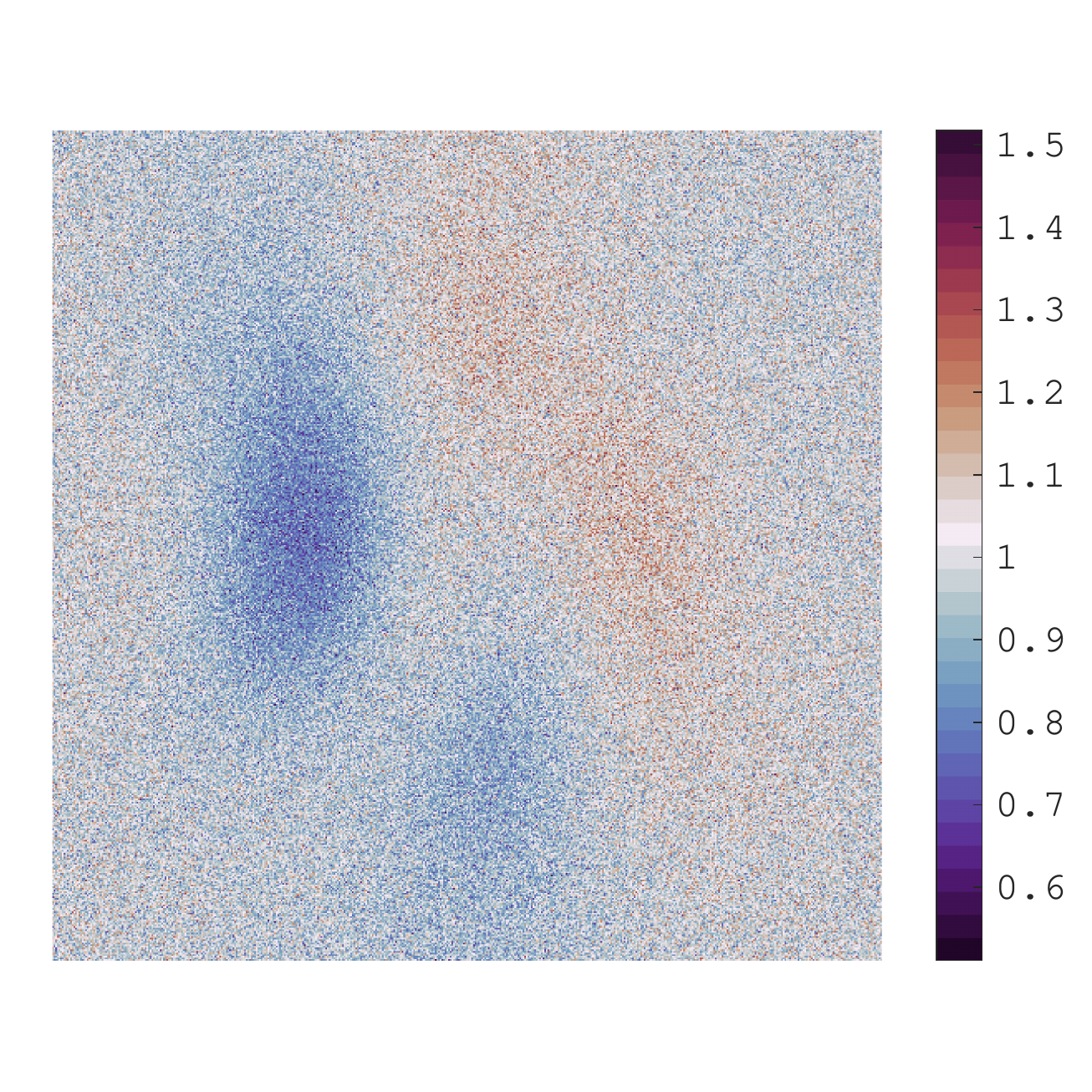}} 
\subfloat[$\hat{a}$]
{\includegraphics[width=0.30\linewidth]{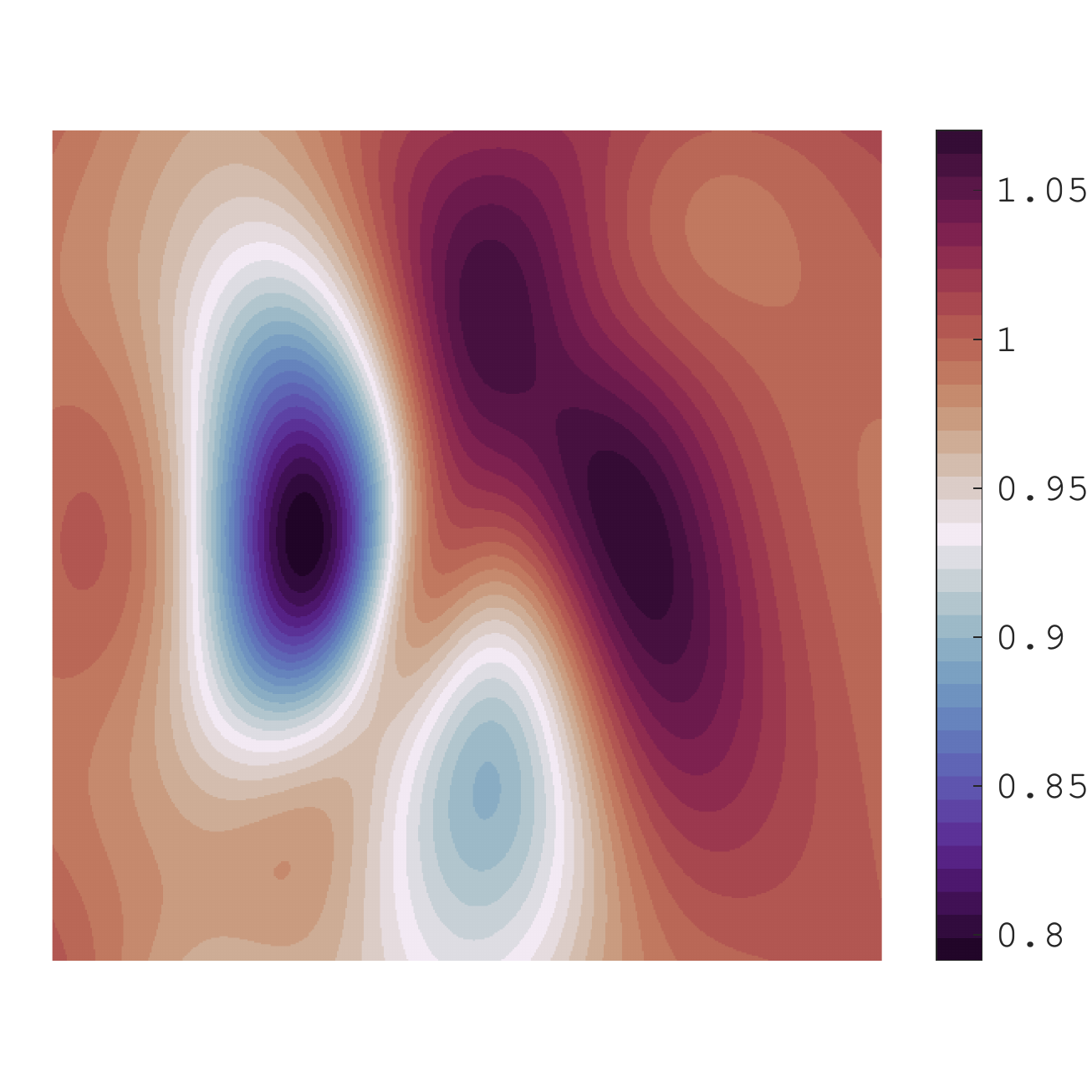}} 
\\
\subfloat[$a^{\dagger}$]
{\includegraphics[width=0.30\linewidth]{./figures/example01observations.pdf}}
\subfloat[$a^{\delta}$]
{\includegraphics[width=0.30\linewidth]{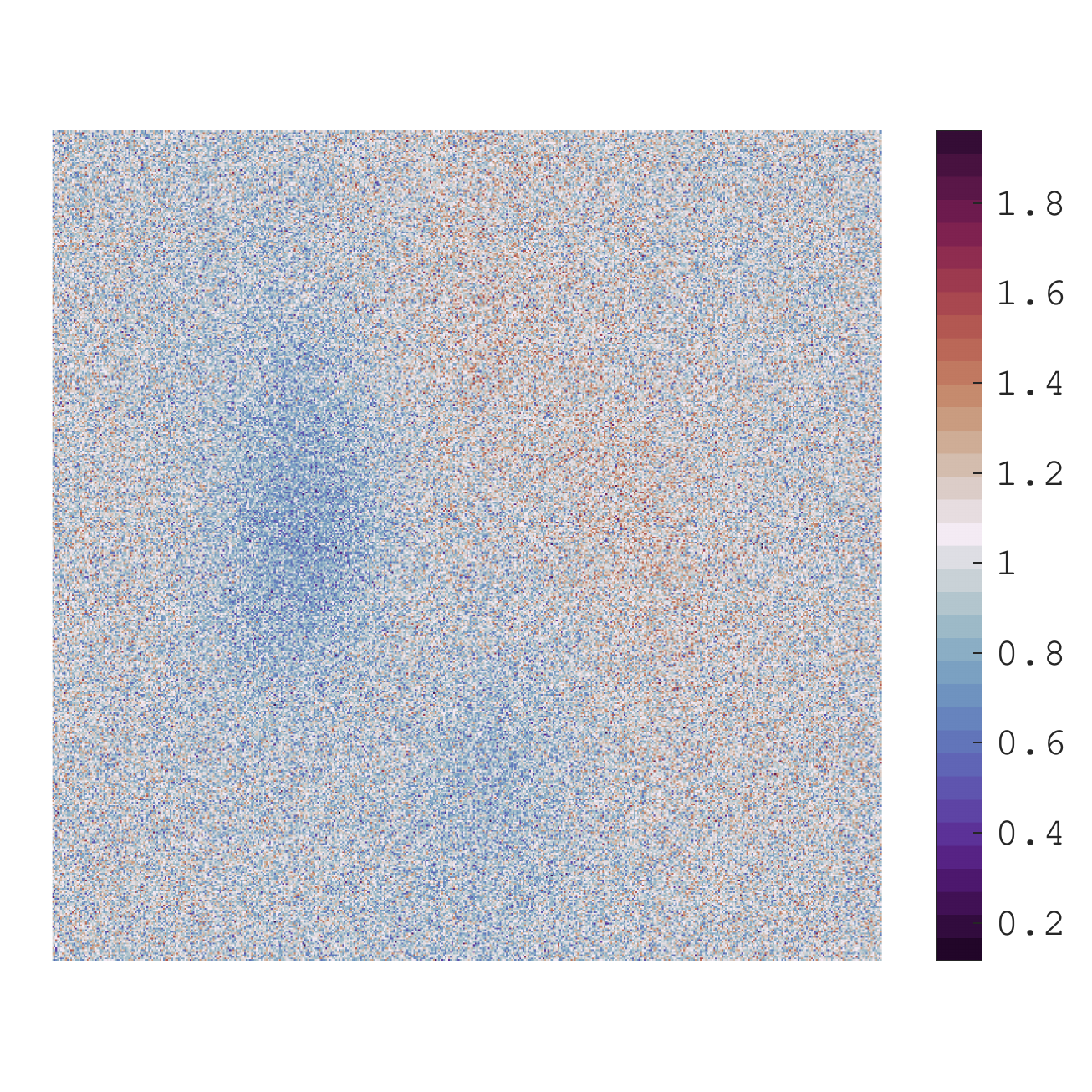}} 
\subfloat[$\hat{a}$]
{\includegraphics[width=0.30\linewidth]{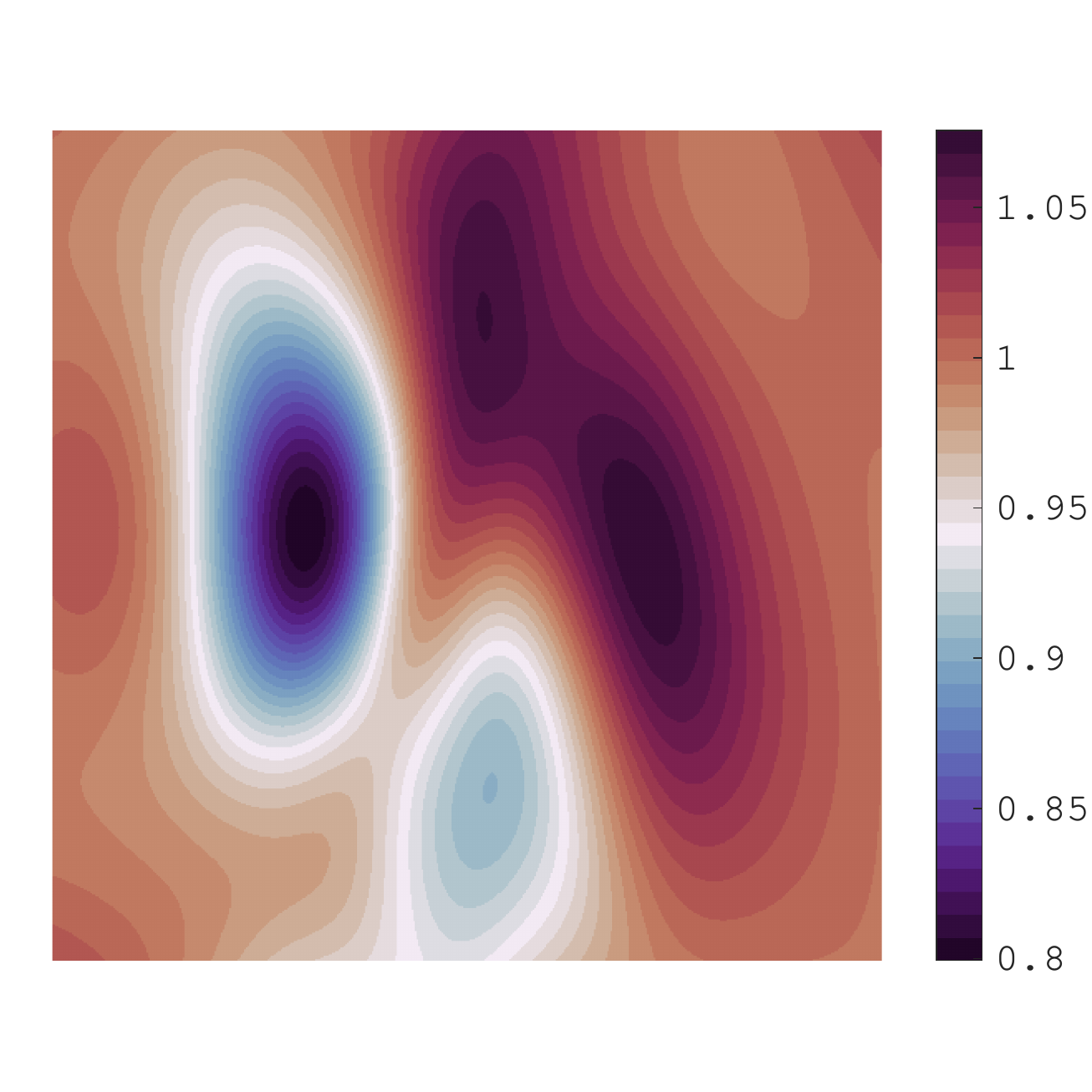}} 
\caption{The ground truth $a^{\dagger}$, noisy measurements $a^{\delta}$, and reconstruction $\hat{a}$ by our method in \cref{example:fourmode}. $\delta=1\%$ (top) $\delta=10\%$ (middle) $\delta=20\%$ (bottom).}
\label{fig:example1:a}
\end{figure}
\begin{figure}
\centering
\subfloat[$\gamma^{\dagger}$]
{\includegraphics[width=0.30\linewidth]{./figures/example01gamma\_exact.pdf}}
\subfloat[$\hat{\gamma}$]
{\includegraphics[width=0.30\linewidth]{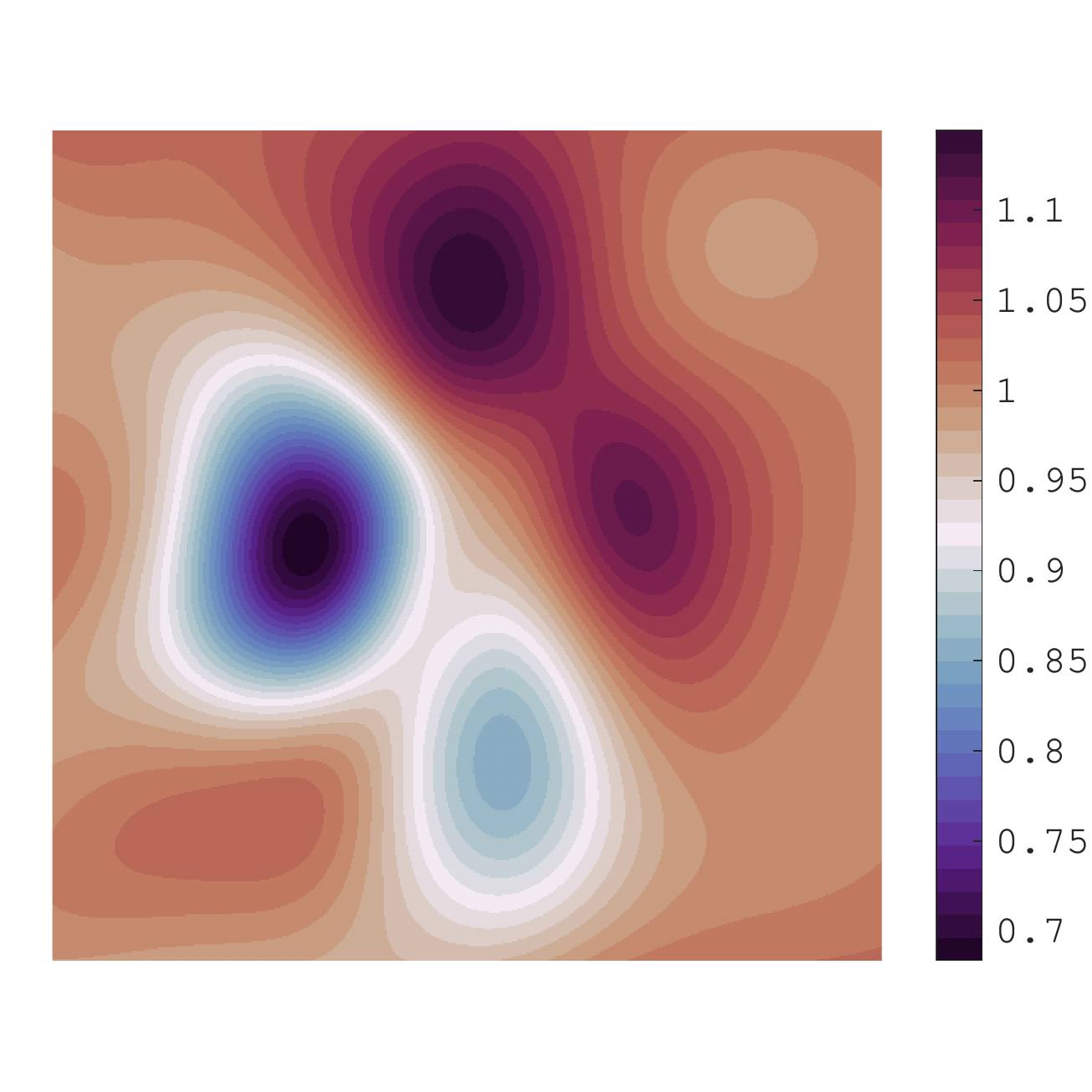}} 
\subfloat[$|\hat{\gamma}-\gamma^{\dagger}|$]
{\includegraphics[width=0.30\linewidth]{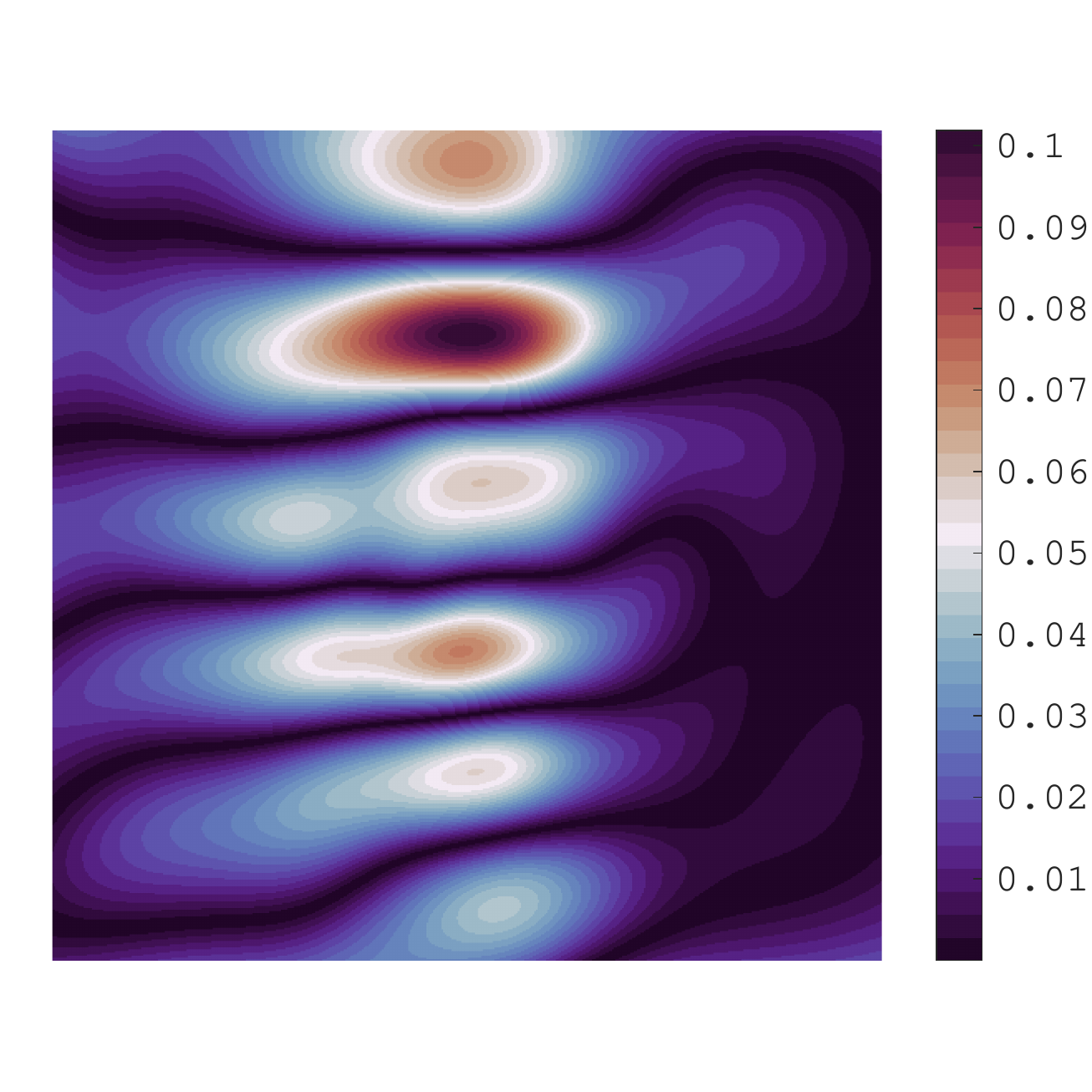}} 
\\
\subfloat[$\gamma^{\dagger}$]
{\includegraphics[width=0.30\linewidth]{./figures/example01gamma\_exact.pdf}}
\subfloat[$\hat{\gamma}$]
{\includegraphics[width=0.30\linewidth]{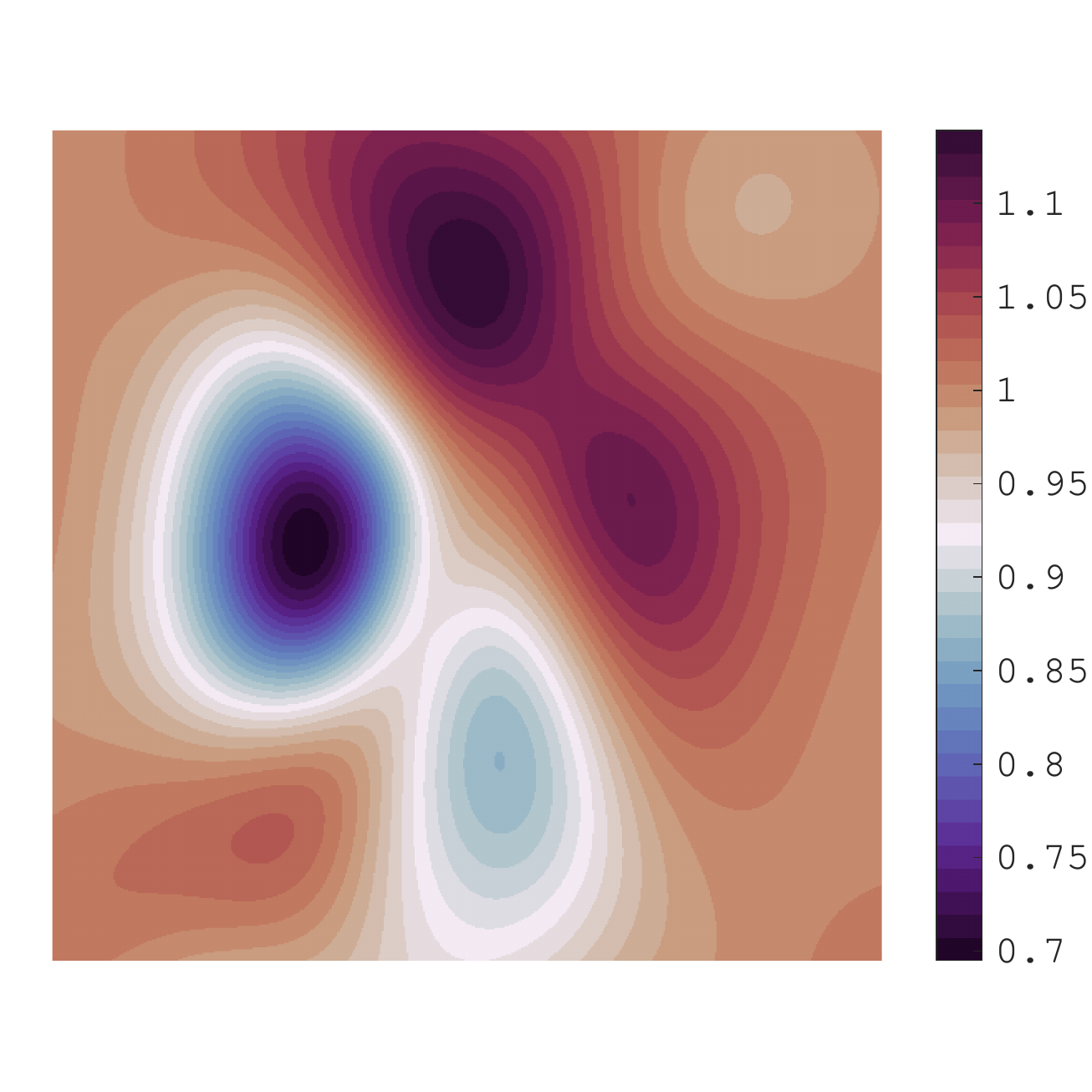}} 
\subfloat[$|\hat{\gamma}-\gamma^{\dagger}|$]
{\includegraphics[width=0.30\linewidth]{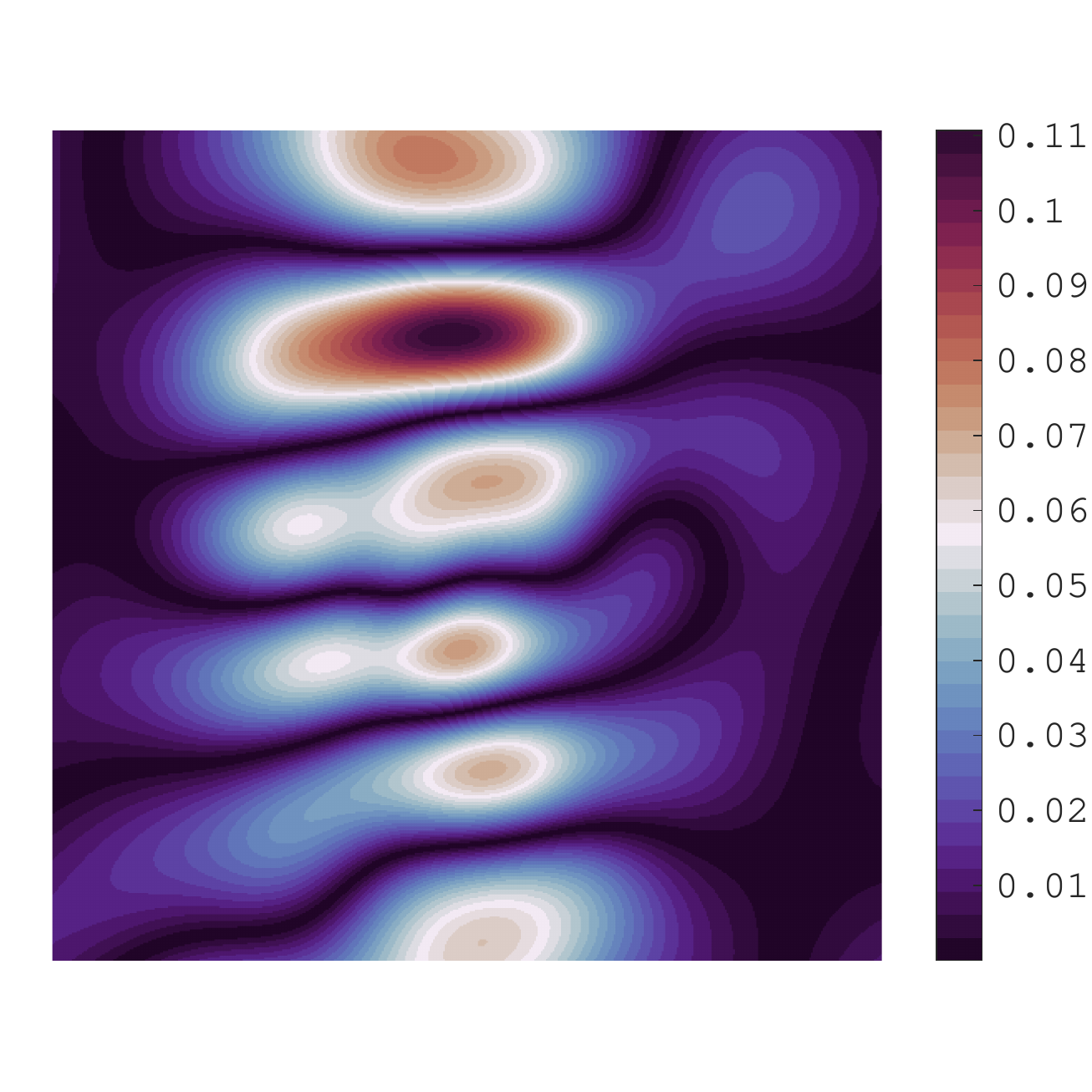}} 
\\
\subfloat[$\gamma^{\dagger}$]
{\includegraphics[width=0.30\linewidth]{./figures/example01gamma\_exact.pdf}}
\subfloat[$\hat{\gamma}$]
{\includegraphics[width=0.30\linewidth]{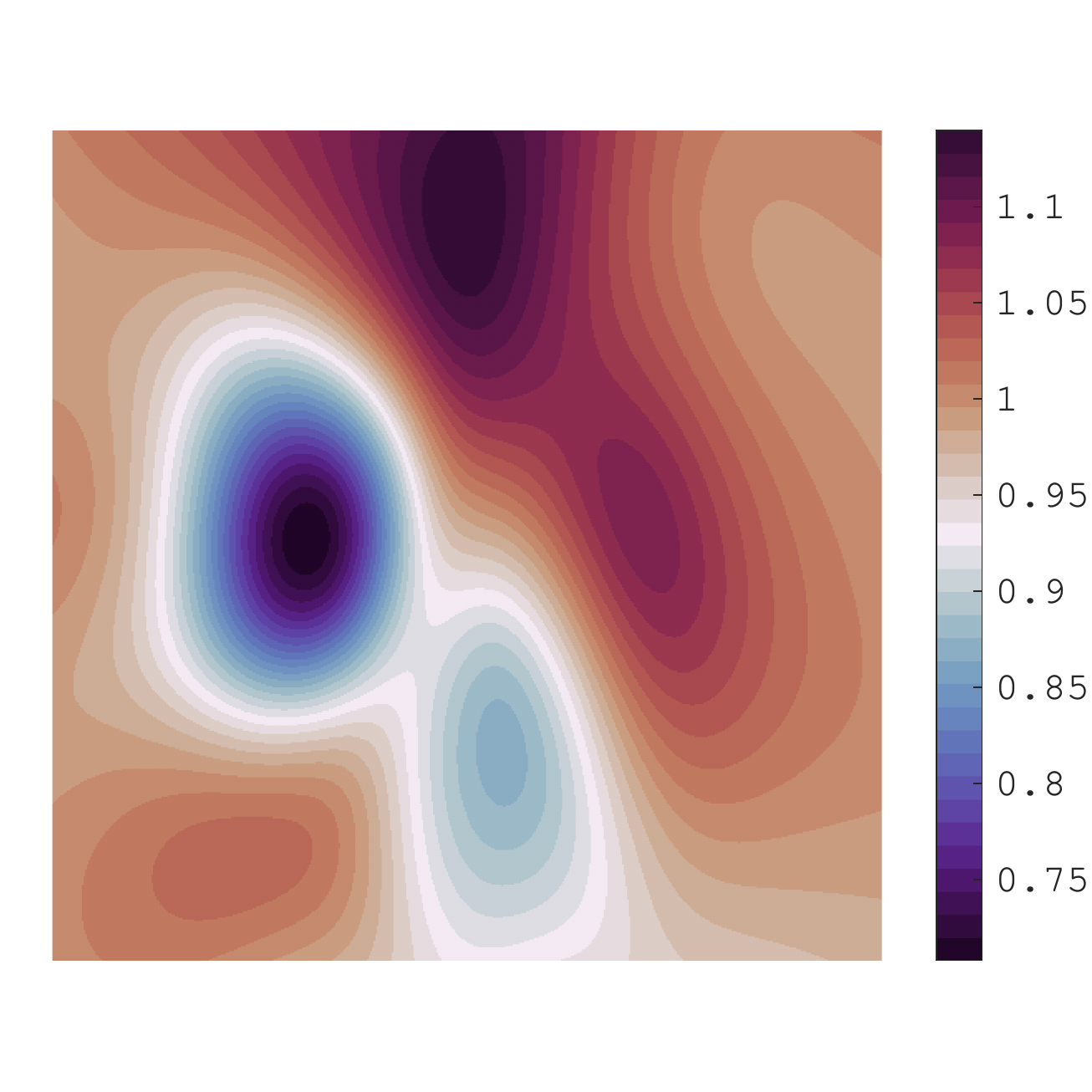}} 
\subfloat[$|\hat{\gamma}-\gamma^{\dagger}|$]
{\includegraphics[width=0.30\linewidth]{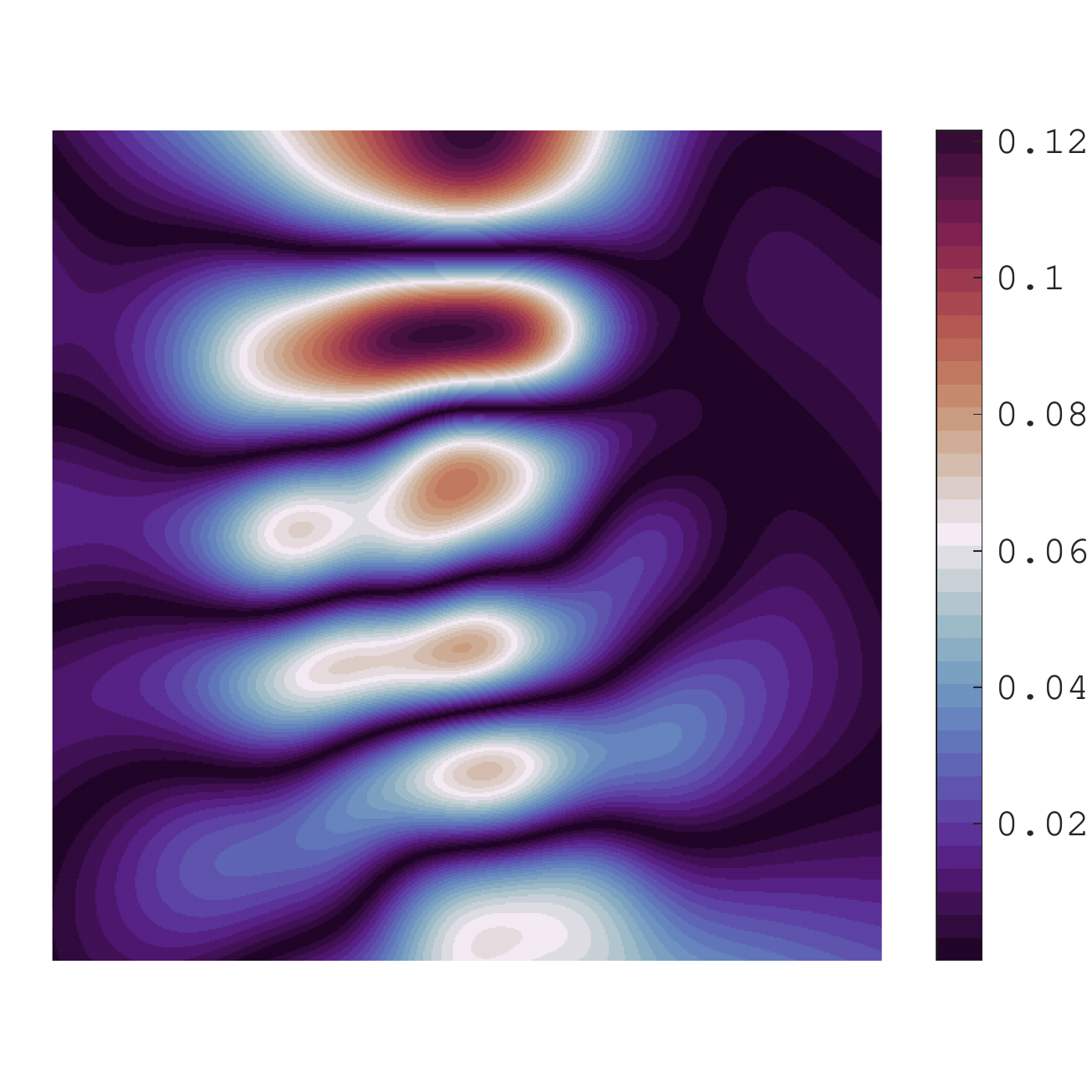}} 
\caption{The ground truth conductivity $\gamma^{\dagger}$, reconstruction $\hat{\gamma}$ by our method, and the point-wise absolute error $|\hat{\gamma}-\gamma^{\dagger}|$ in \cref{example:fourmode}. $\delta=1\%$ (top) $\delta=10\%$ (middle) $\delta=20\%$ (bottom).}
\label{fig:example1:gamma}
\end{figure}
\begin{figure}
\centering
\subfloat[$u^{\dagger}$]
{\includegraphics[width=0.30\linewidth]{./figures/example01solution\_exact.pdf}}
\subfloat[$\hat{u}$]
{\includegraphics[width=0.30\linewidth]{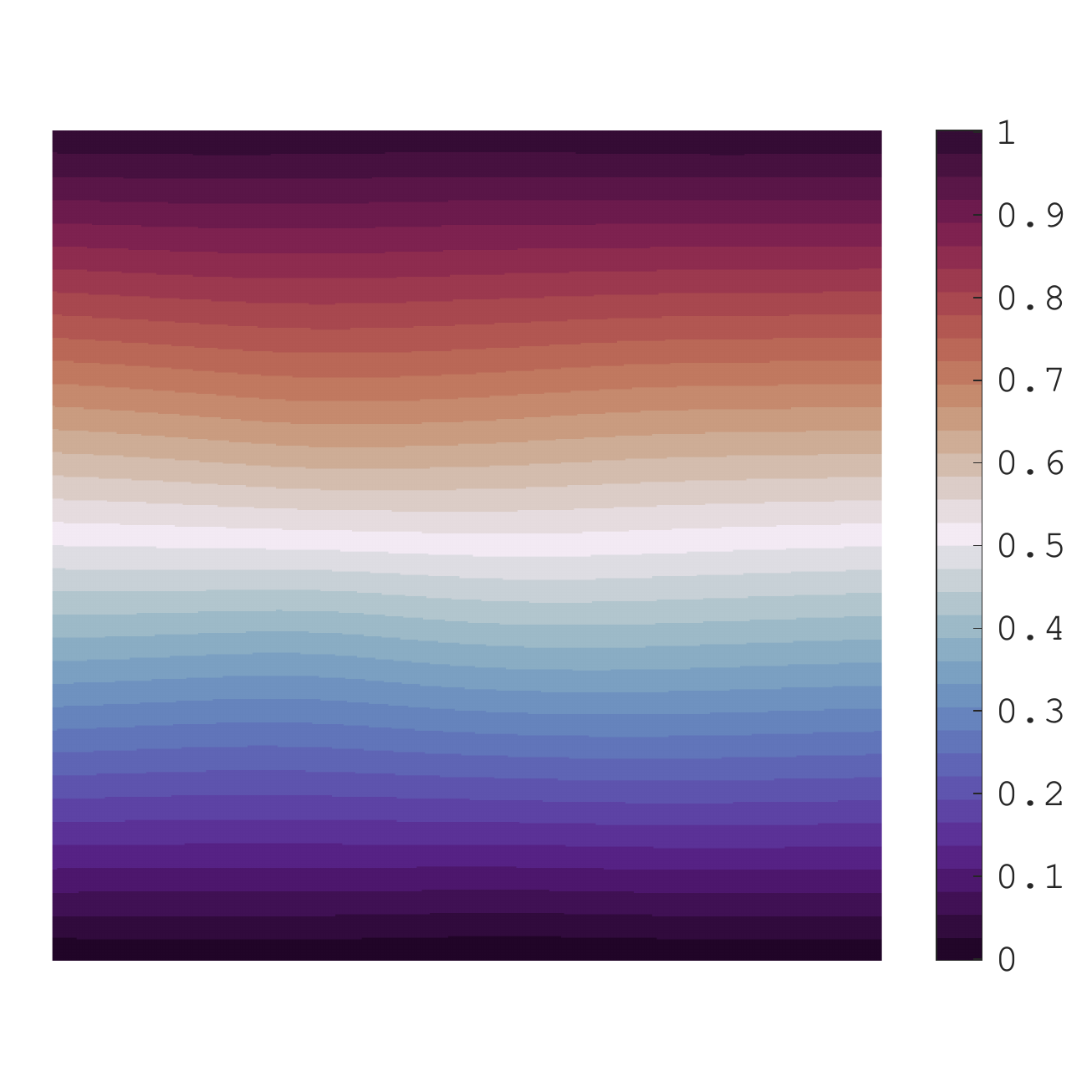}} 
\subfloat[$|\hat{u}-u^{\dagger}|$]
{\includegraphics[width=0.30\linewidth]{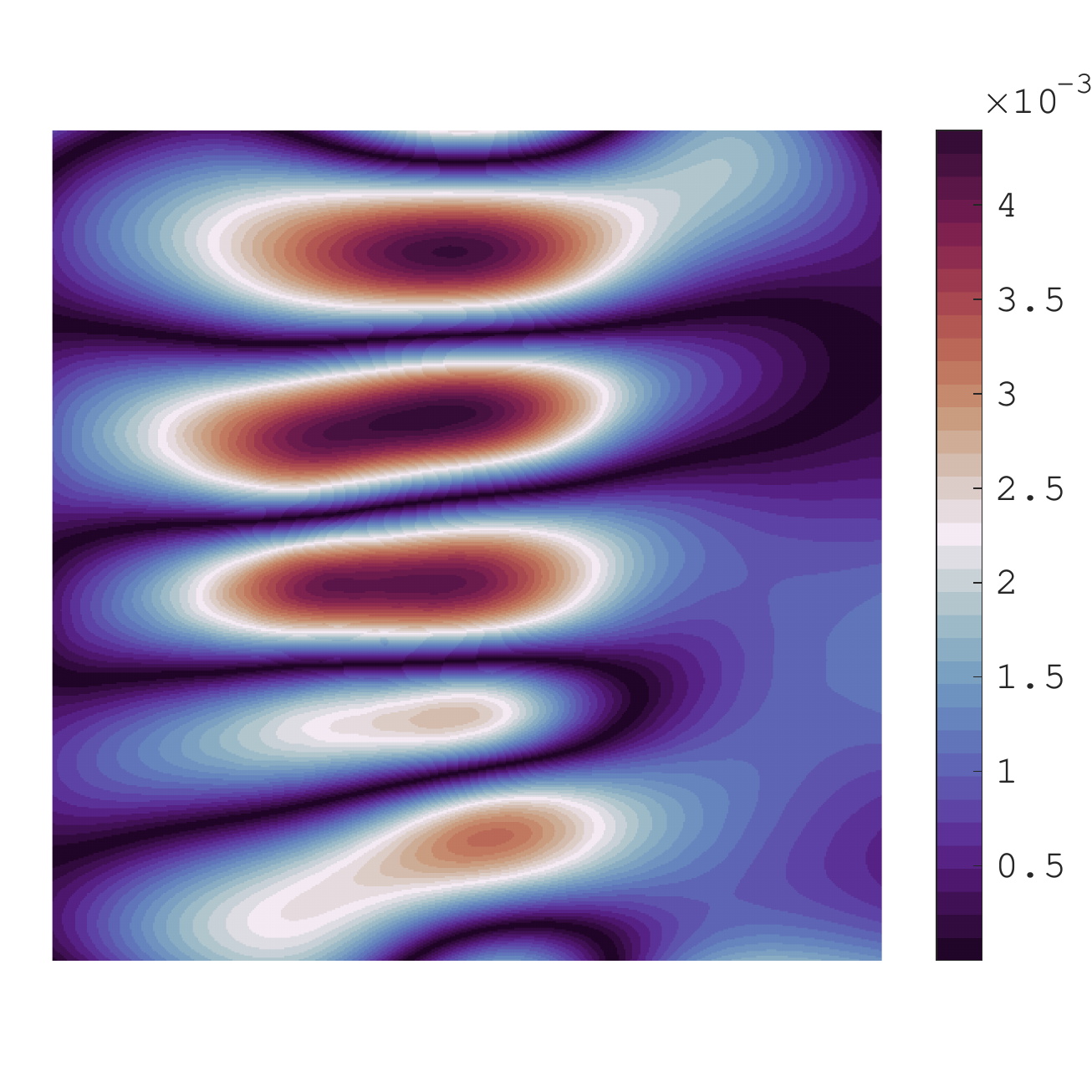}} 
\\
\subfloat[$u^{\dagger}$]
{\includegraphics[width=0.30\linewidth]{./figures/example01solution\_exact.pdf}}
\subfloat[$\hat{u}$]
{\includegraphics[width=0.30\linewidth]{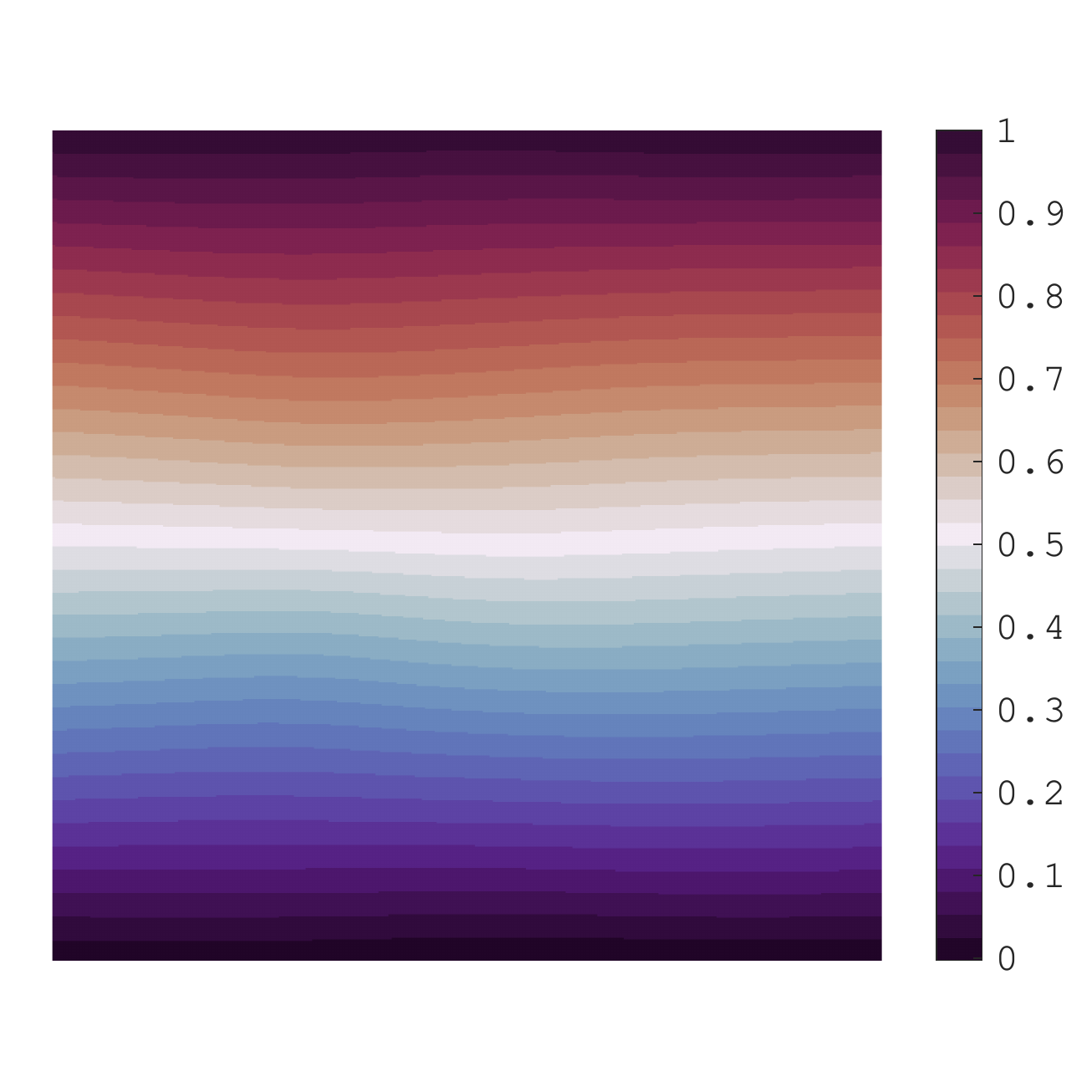}} 
\subfloat[$|\hat{u}-u^{\dagger}|$]
{\includegraphics[width=0.30\linewidth]{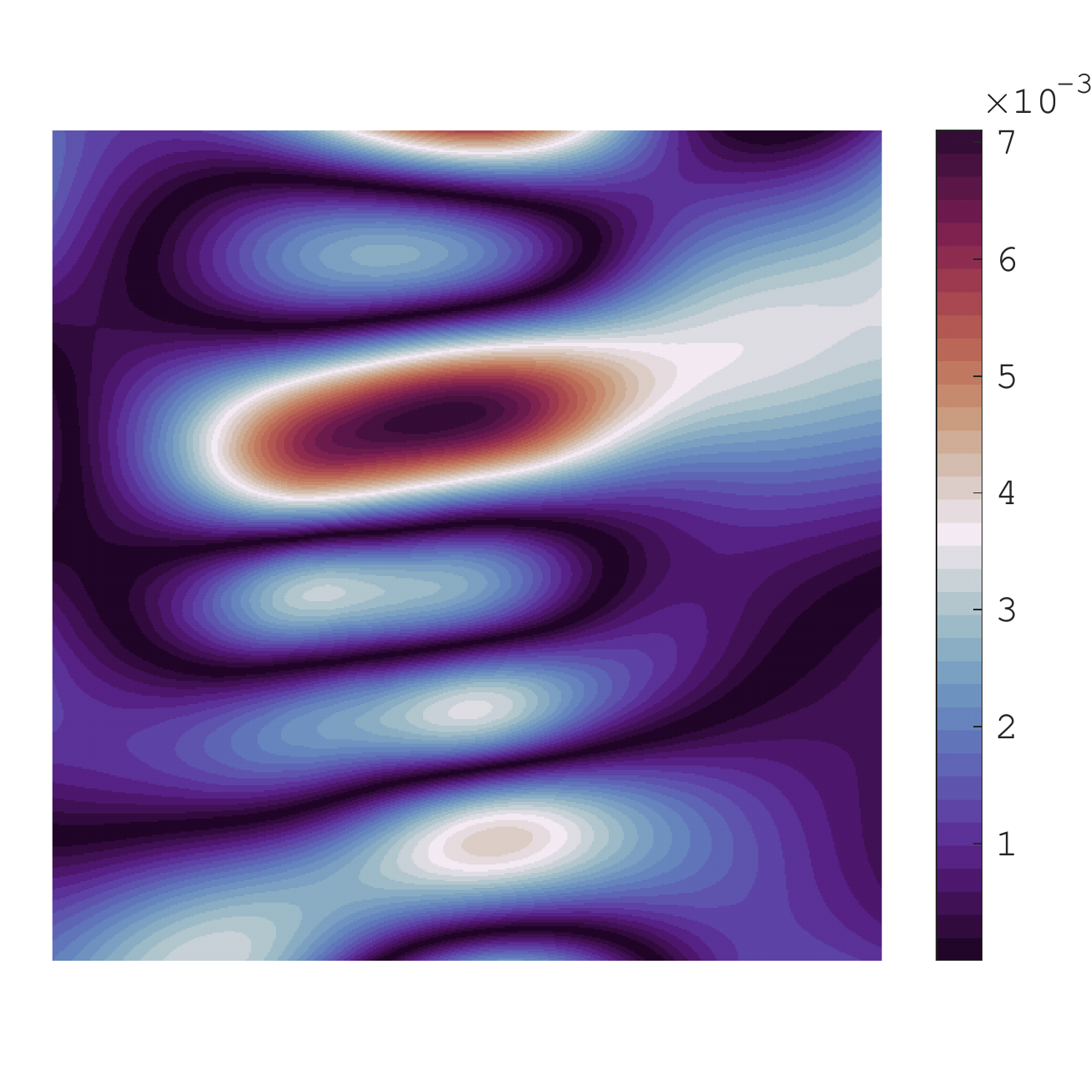}} 
\\
\subfloat[$u^{\dagger}$]
{\includegraphics[width=0.30\linewidth]{./figures/example01solution\_exact.pdf}}
\subfloat[$\hat{u}$]
{\includegraphics[width=0.30\linewidth]{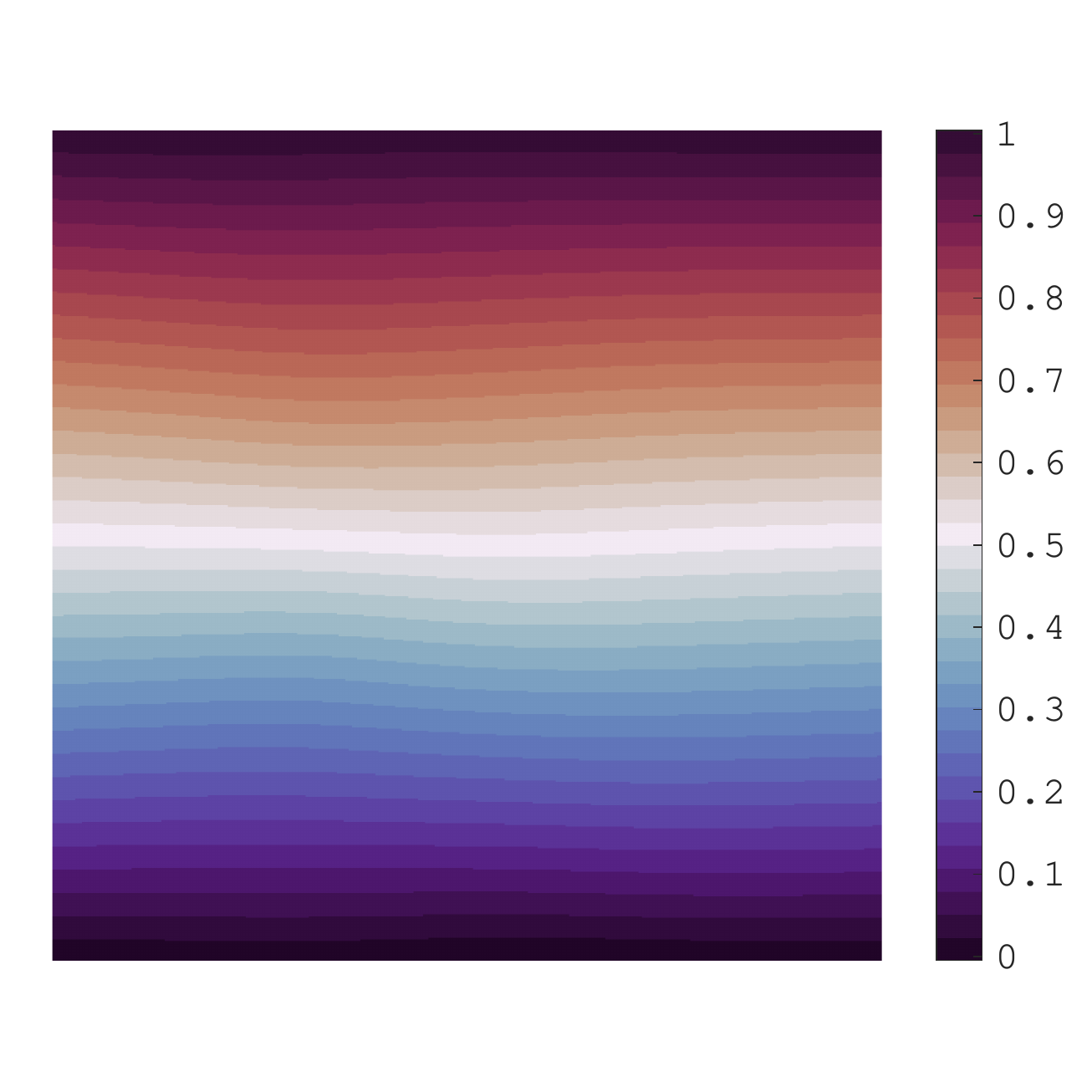}} 
\subfloat[$|\hat{u}-u^{\dagger}|$]
{\includegraphics[width=0.30\linewidth]{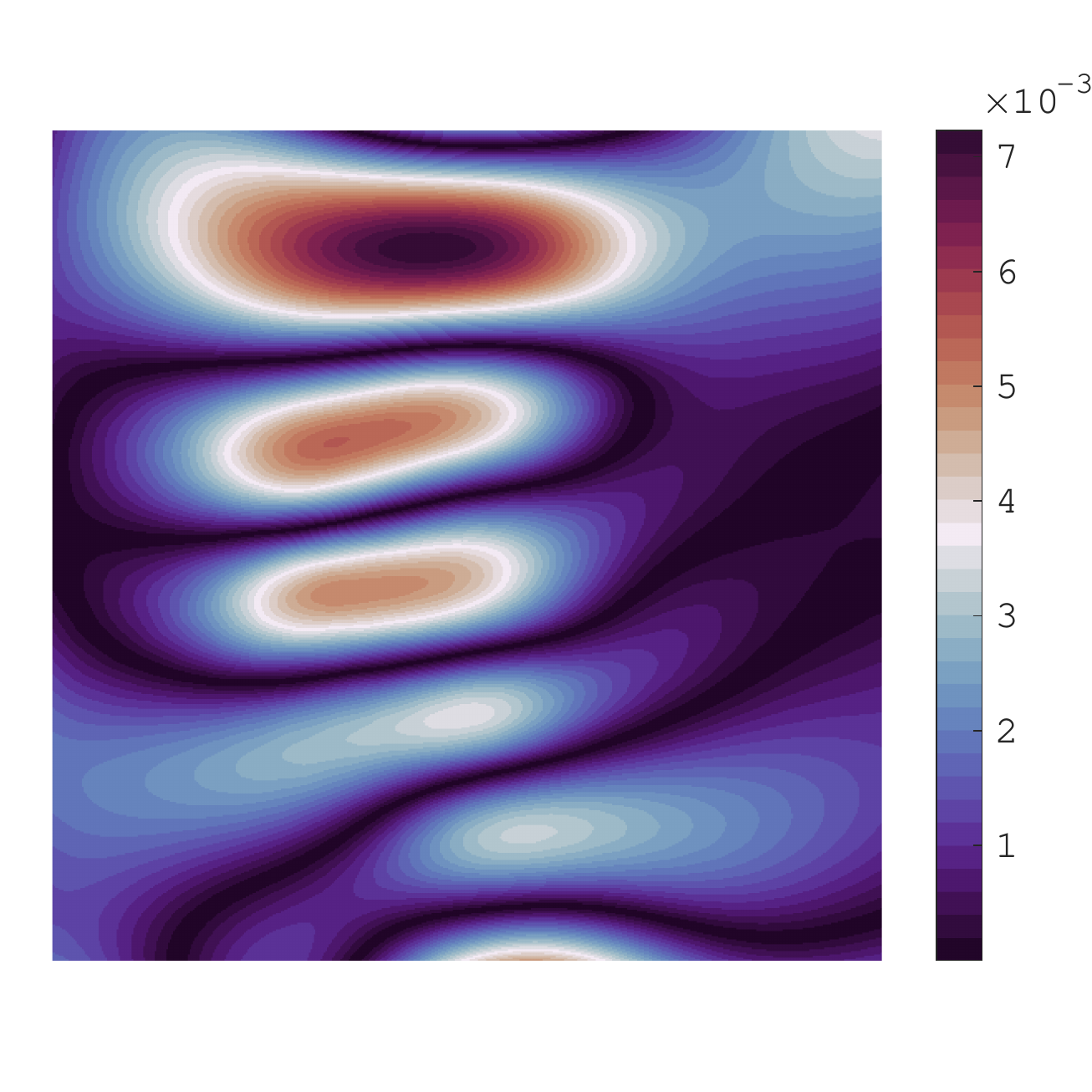}} 
\caption{The ground truth voltage $u^{\dagger}$, reconstruction $\hat{u}$ by our method, and the point-wise absolute error $|\hat{u}-u^{\dagger}|$ in \cref{example:fourmode}. $\delta=1\%$ (top) $\delta=10\%$ (middle) $\delta=20\%$ (bottom).}
\label{fig:example1:u}
\end{figure}

\begin{example}[CDII with discontinuous conductivity]\label{example:dc}
Let $\rho(x,y)=\|(x,y)-(\frac{1}{2},\frac{1}{2})\|_{2}$.
\begin{equation*}
\gamma^{\dagger}(x,y)=1+\chi_{\{x>\frac{1}{2}\}}\exp(-2\rho^{2}).
\end{equation*}
\end{example}
\par We apply our approach to reconstruct the conductivity and voltage from the measurement data with 1\%, 10\%, and 20\% noise, respectively. The relative $L^{2}$-error of the recovered voltage $\hat{u}$ and conductivity $\hat{\gamma}$ are shown in \cref{tab:example2:error}. Our results demonstrate that the proposed method can effectively reconstruct both the conductivity and voltage from data sets with noise levels of up to 20\%, while maintaining high accuracy. At the same time, the accuracy does not decrease significantly as the noise level increases. We also compare the ground truth measurement data, noisy data, and recovered data in \cref{fig:example2:a}. As shown in \cref{fig:example2:gamma,fig:example2:u}, the error in the reconstructed conductivity and voltage is primarily concentrated at the discontinuous interface, which aligns with the findings of \cite{Jin2022imaging}. Additionally, our method has no significant advantage over that in \cite{Jin2022imaging} on this example, mainly because our method is based on a strong form of the PDE.
\begin{table}
\caption{The relative $L^{2}$-error of the recovered conductivity $\hat{\gamma}$ and voltage $\hat{u}$ in \cref{example:dc}.}
\centering
\begin{tabular}{lccc}
\toprule
& \multicolumn{3}{c}{$\delta$}   \\
\cmidrule(lr){2-4}
 & $1\%$ & $10\%$ & $20\%$  \\
\cmidrule{1-4}
$\err(\gamma)$ & $3.71 \times 10^{-2}$ & $4.38 \times 10^{-2}$ & $4.81 \times 10^{-2}$  \\
$\err(u)$ & $8.48 \times 10^{-3}$ & $9.08 \times 10^{-3}$ & $9.67 \times 10^{-3}$  \\
\bottomrule
\end{tabular}
\label{tab:example2:error}
\end{table}
\begin{figure}
\centering
\subfloat[$a^{\dagger}$]
{\includegraphics[width=0.30\linewidth]{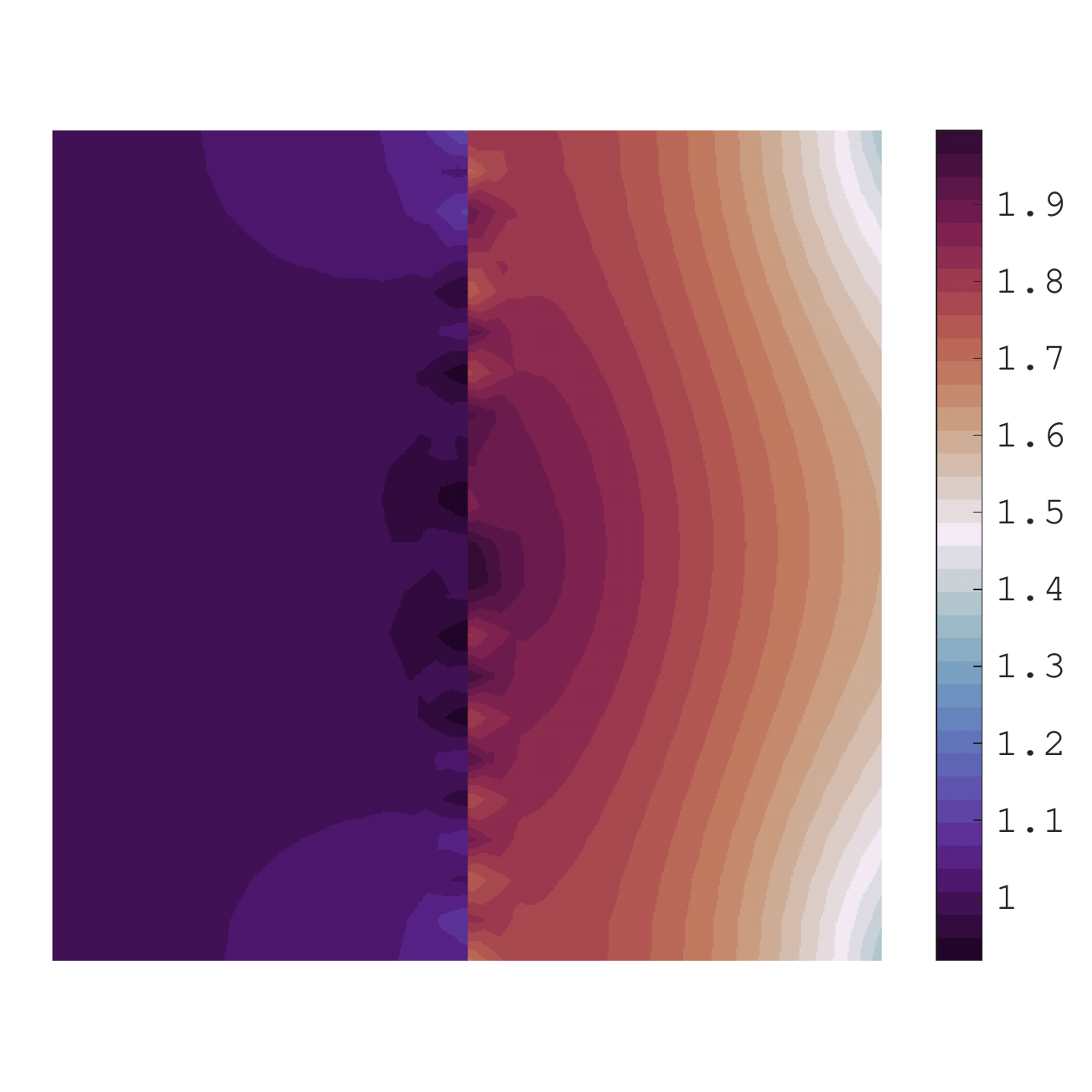}}
\subfloat[$a^{\delta}$]
{\includegraphics[width=0.30\linewidth]{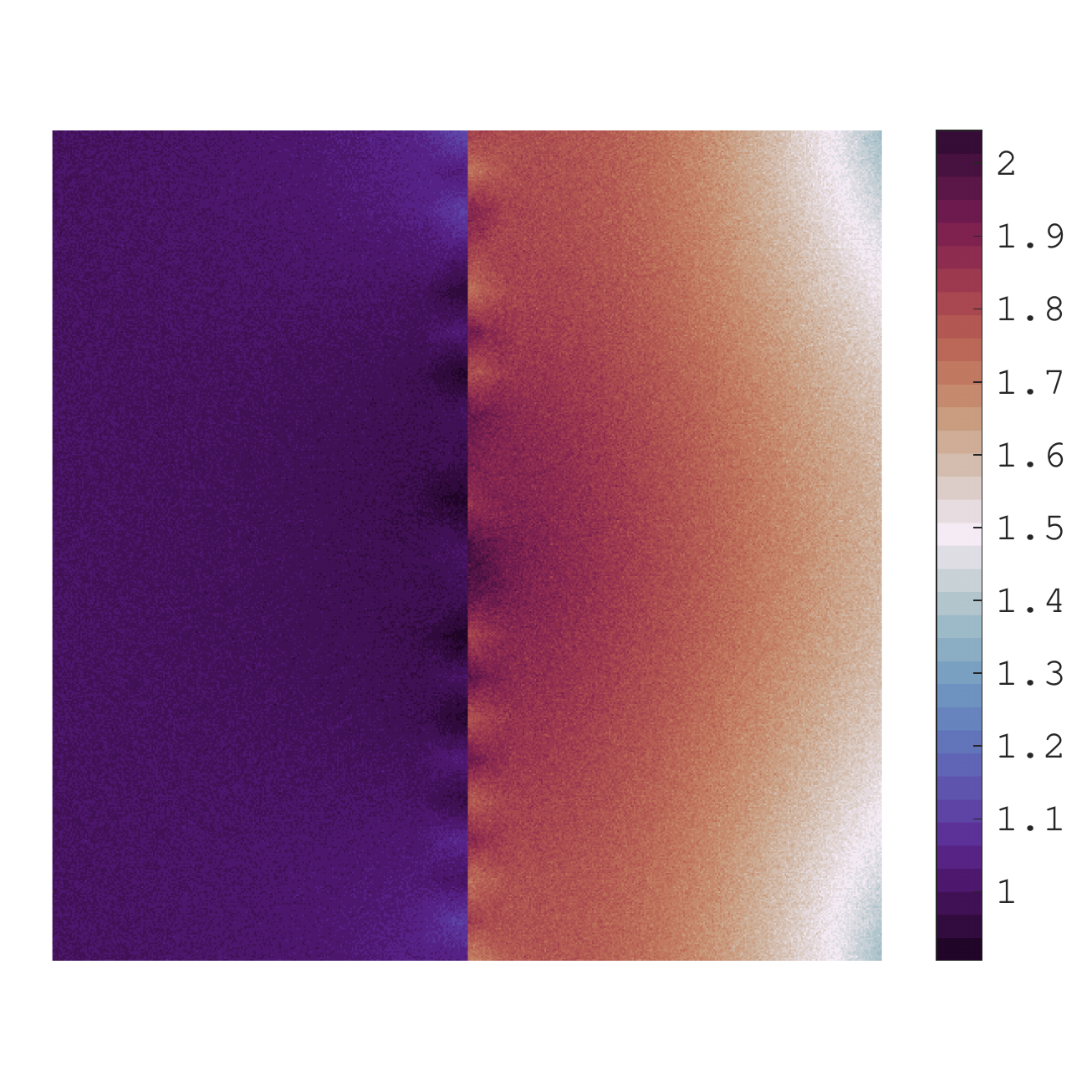}} 
\subfloat[$\hat{a}$]
{\includegraphics[width=0.30\linewidth]{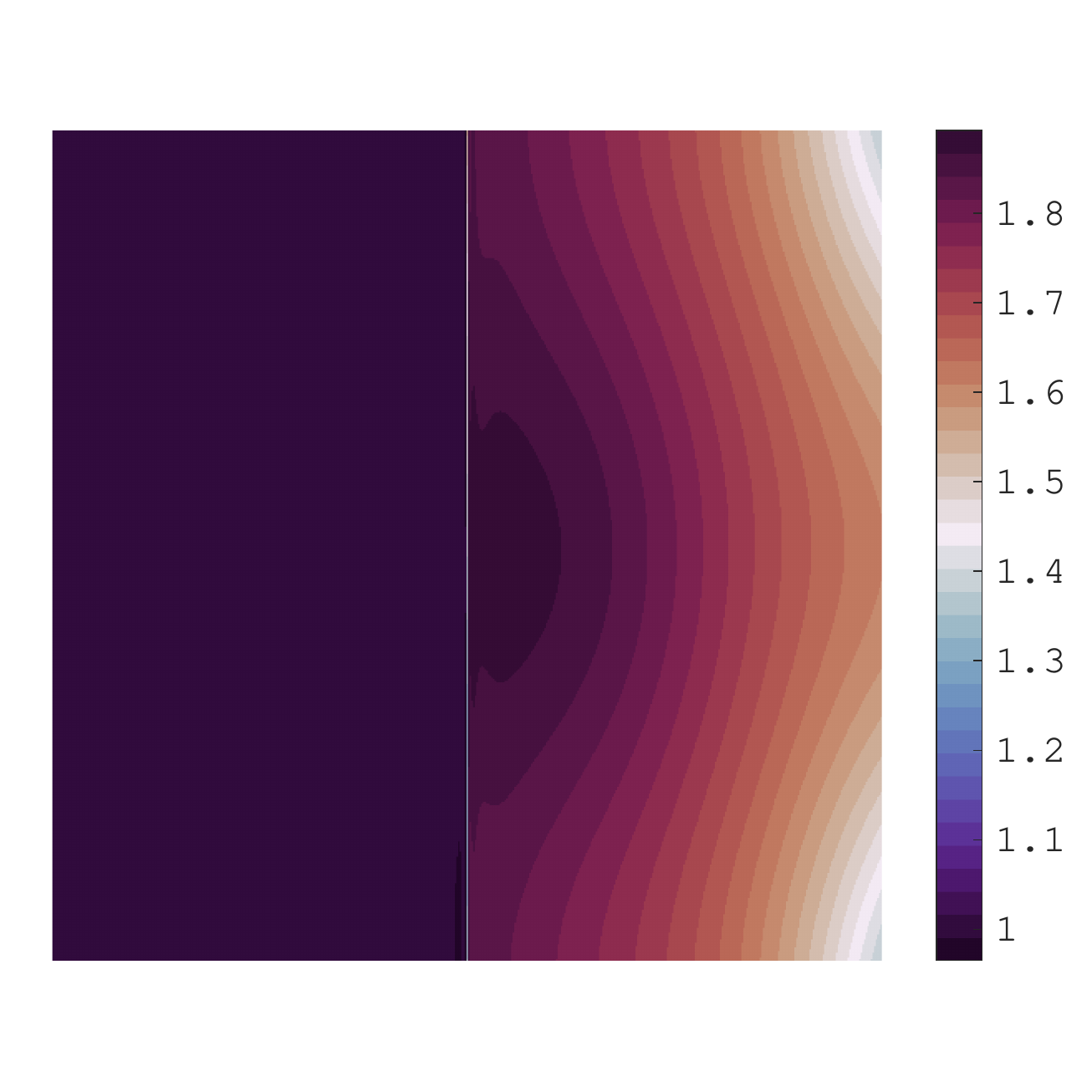}} 
\\
\subfloat[$a^{\dagger}$]
{\includegraphics[width=0.30\linewidth]{./figures/example02observations.pdf}}
\subfloat[$a^{\delta}$]
{\includegraphics[width=0.30\linewidth]{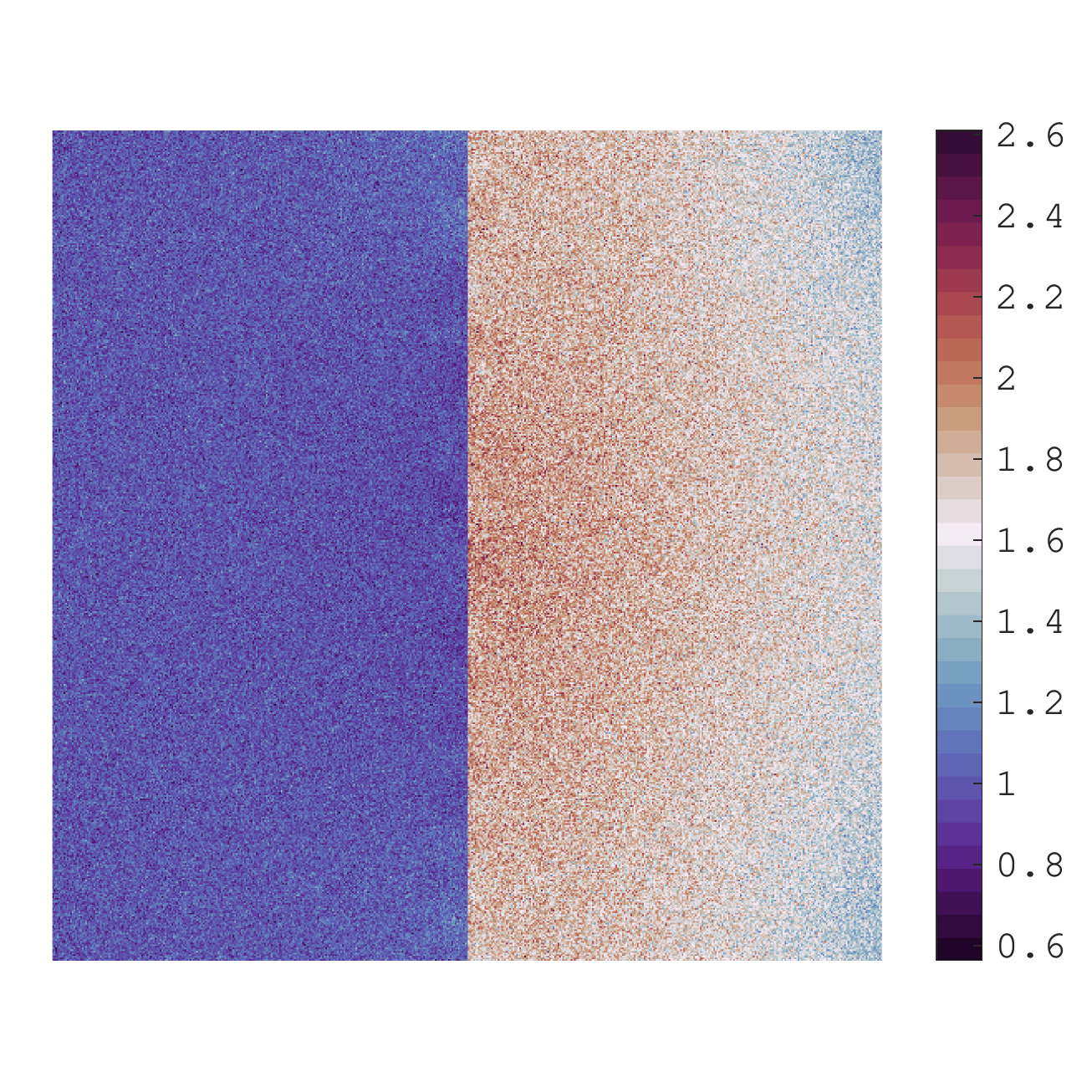}} 
\subfloat[$\hat{a}$]
{\includegraphics[width=0.30\linewidth]{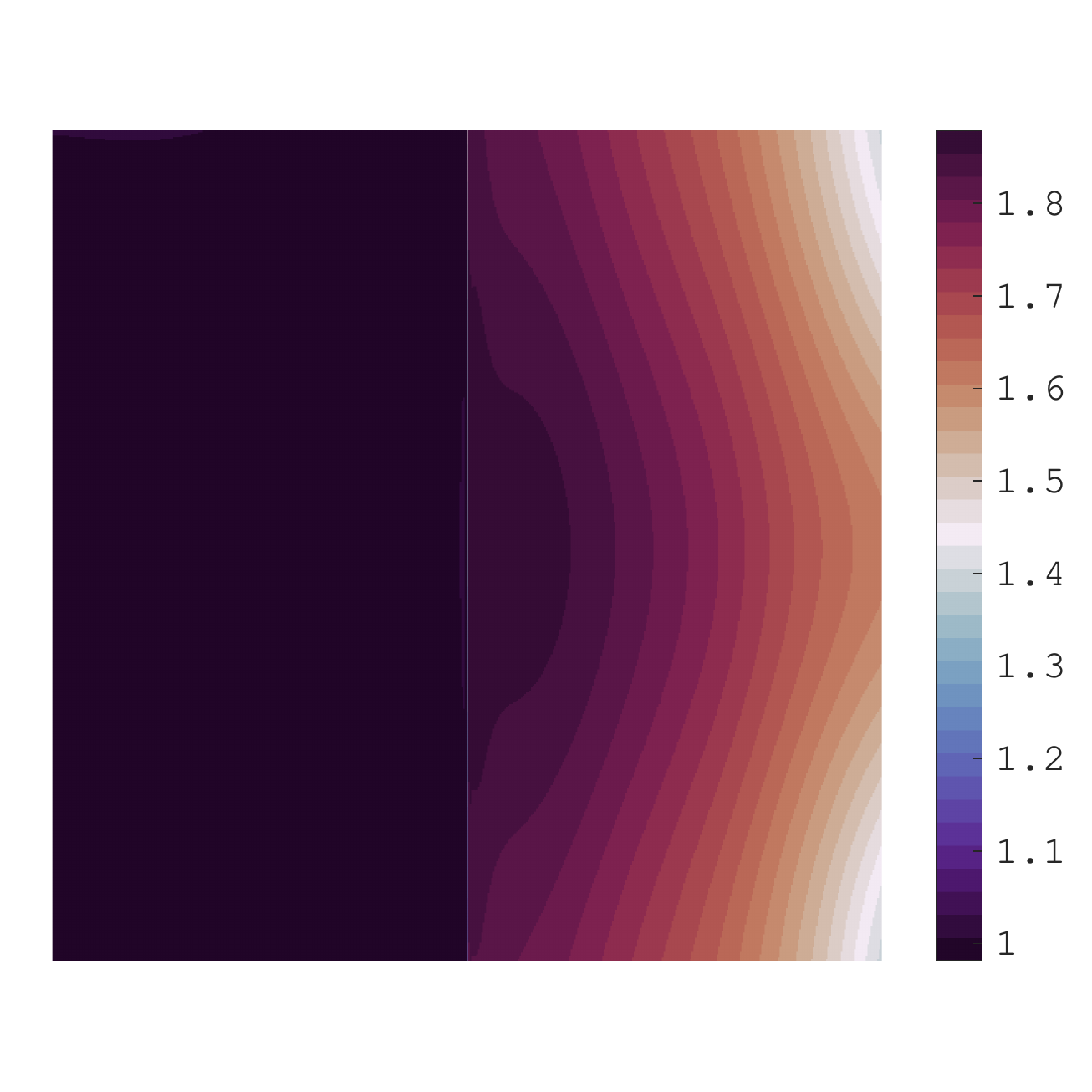}} 
\\
\subfloat[$a^{\dagger}$]
{\includegraphics[width=0.30\linewidth]{./figures/example02observations.pdf}}
\subfloat[$a^{\delta}$]
{\includegraphics[width=0.30\linewidth]{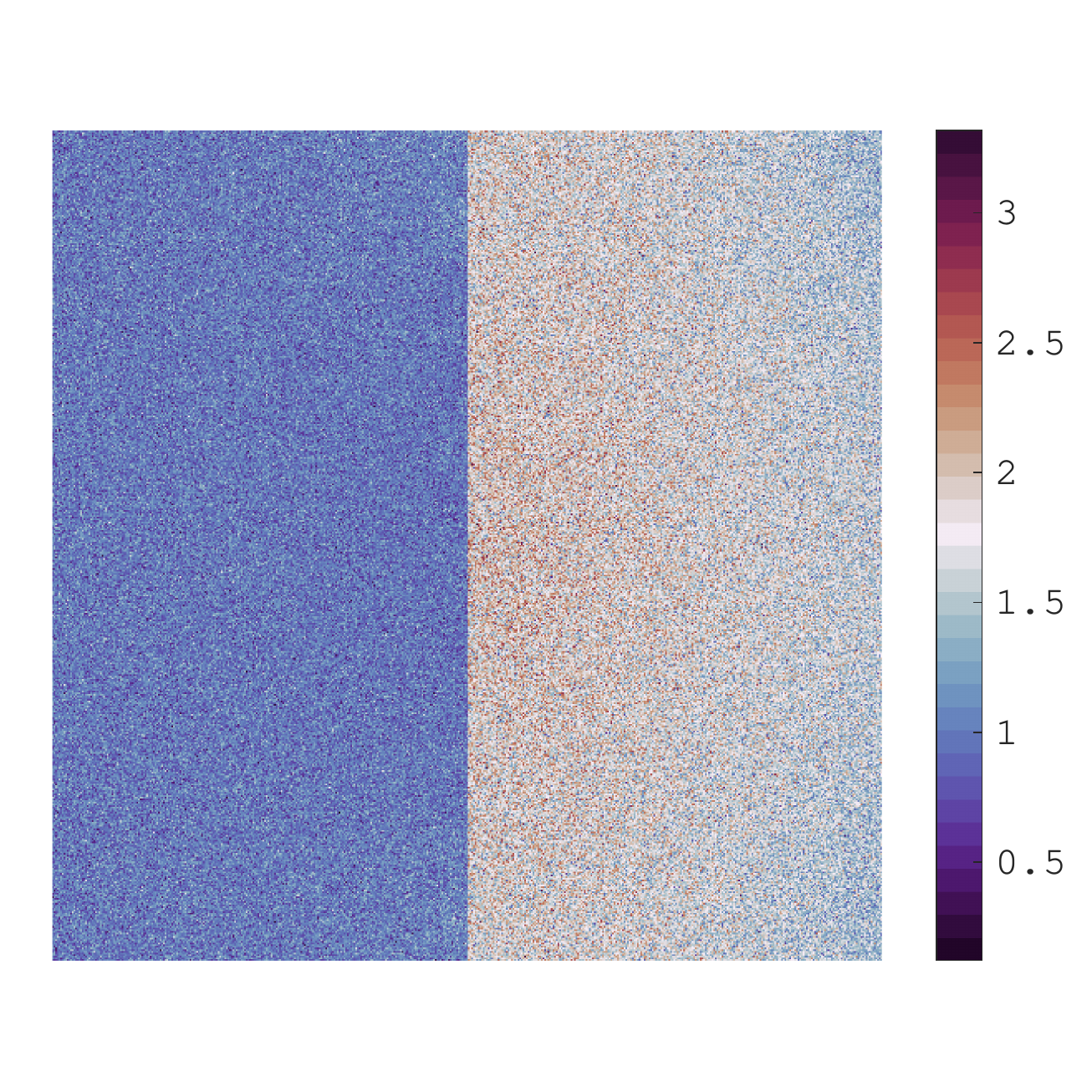}} 
\subfloat[$\hat{a}$]
{\includegraphics[width=0.30\linewidth]{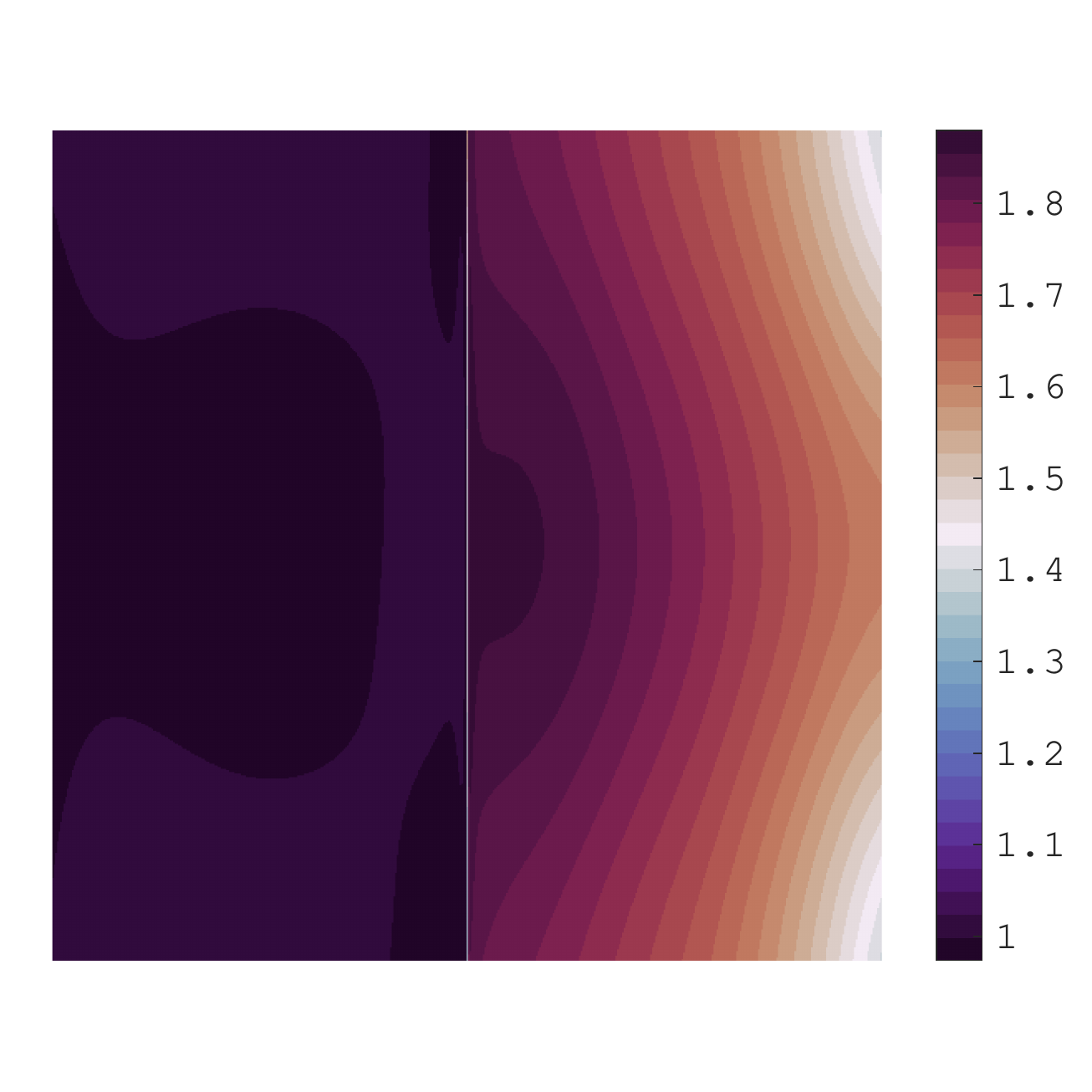}} 
\caption{The ground truth $a^{\dagger}$, noisy measurements $a^{\delta}$, and reconstruction $\hat{a}$ by our method in \cref{example:dc}. $\delta=1\%$ (top) $\delta=10\%$ (middle) $\delta=20\%$ (bottom)}
\label{fig:example2:a}
\end{figure}
\begin{figure}
\centering
\subfloat[$\gamma^{\dagger}$]
{\includegraphics[width=0.30\linewidth]{./figures/example02gamma\_exact.pdf}}
\subfloat[$\hat{\gamma}$]
{\includegraphics[width=0.30\linewidth]{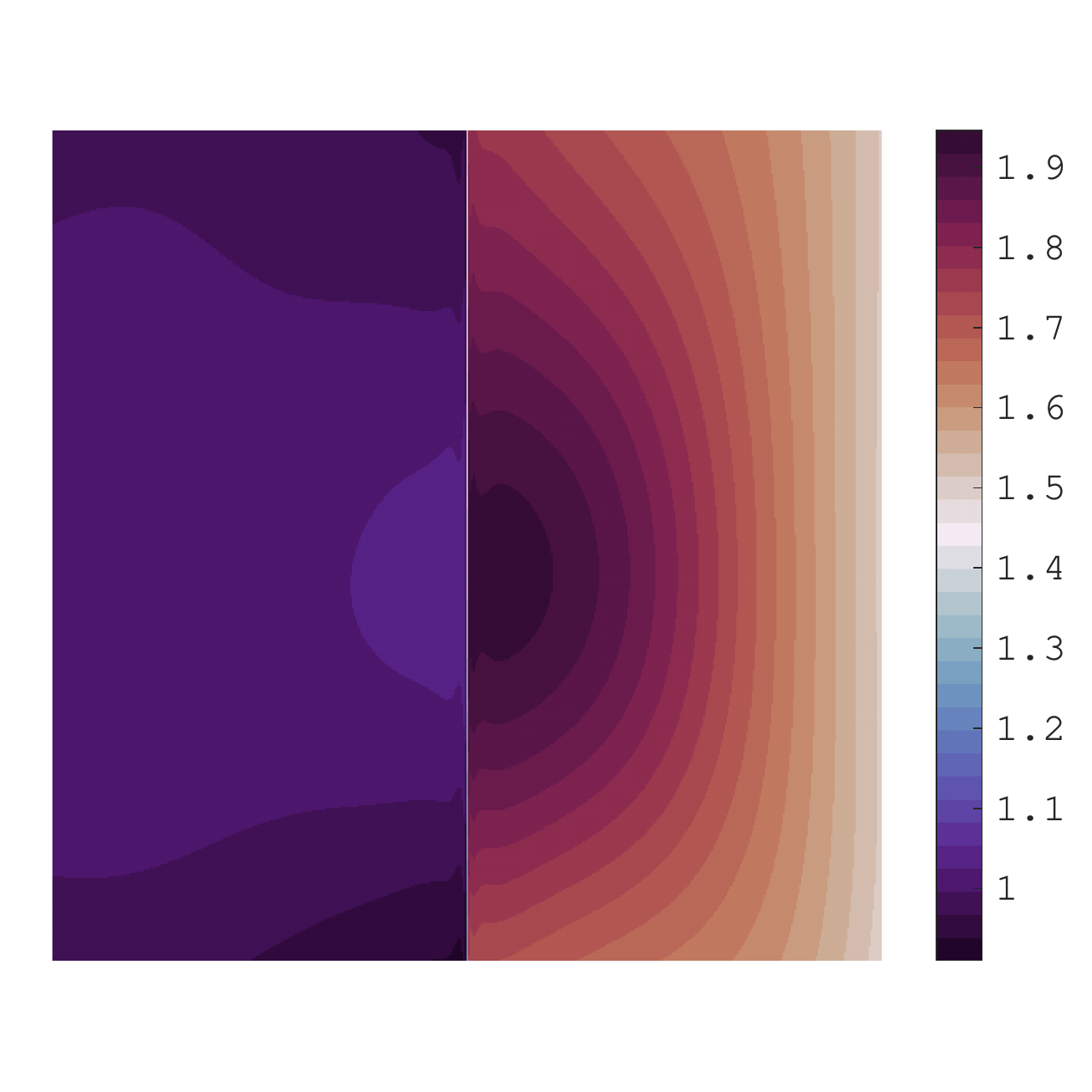}} 
\subfloat[$|\hat{\gamma}-\gamma^{\dagger}|$]
{\includegraphics[width=0.30\linewidth]{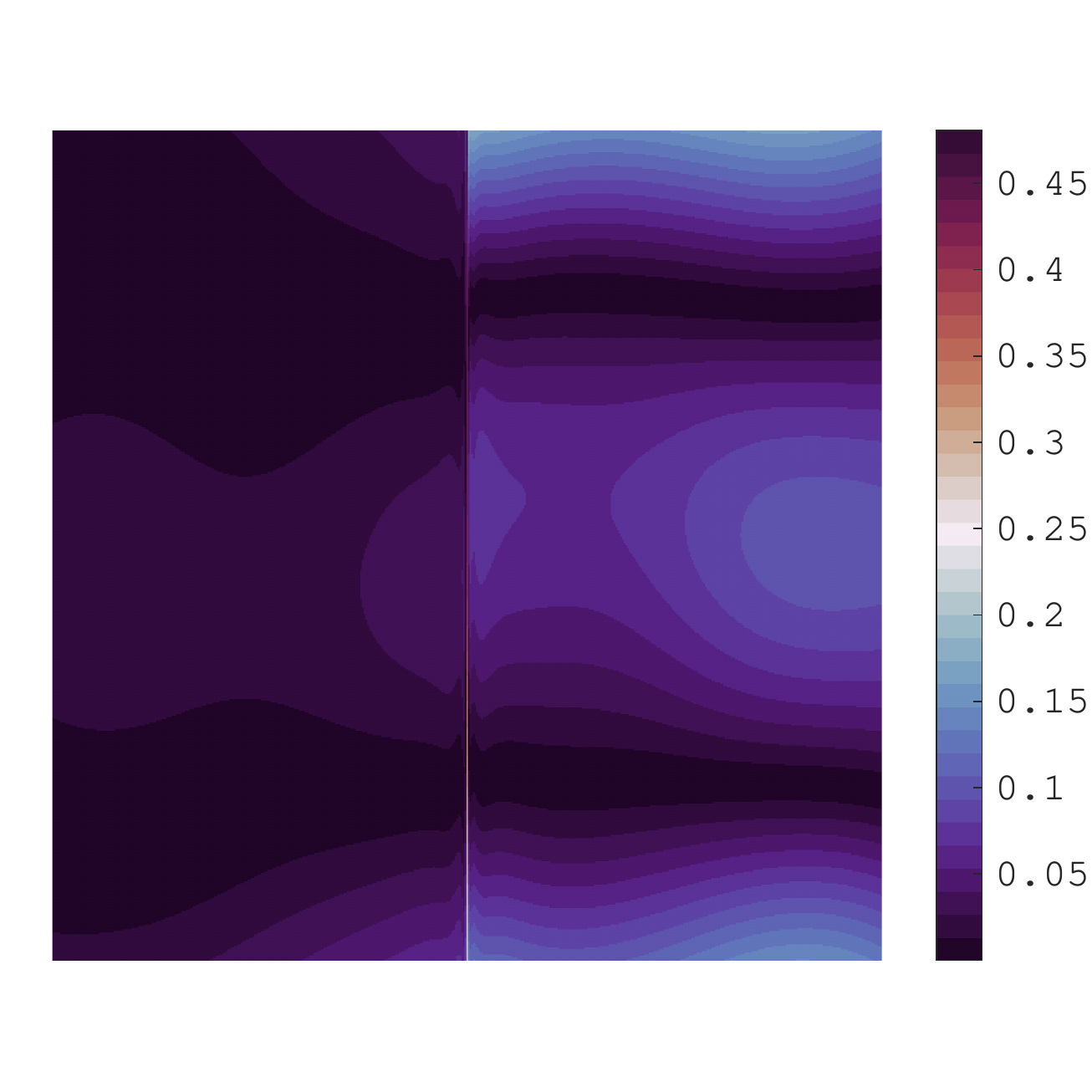}} 
\\
\subfloat[$\gamma^{\dagger}$]
{\includegraphics[width=0.30\linewidth]{./figures/example02gamma\_exact.pdf}}
\subfloat[$\hat{\gamma}$]
{\includegraphics[width=0.30\linewidth]{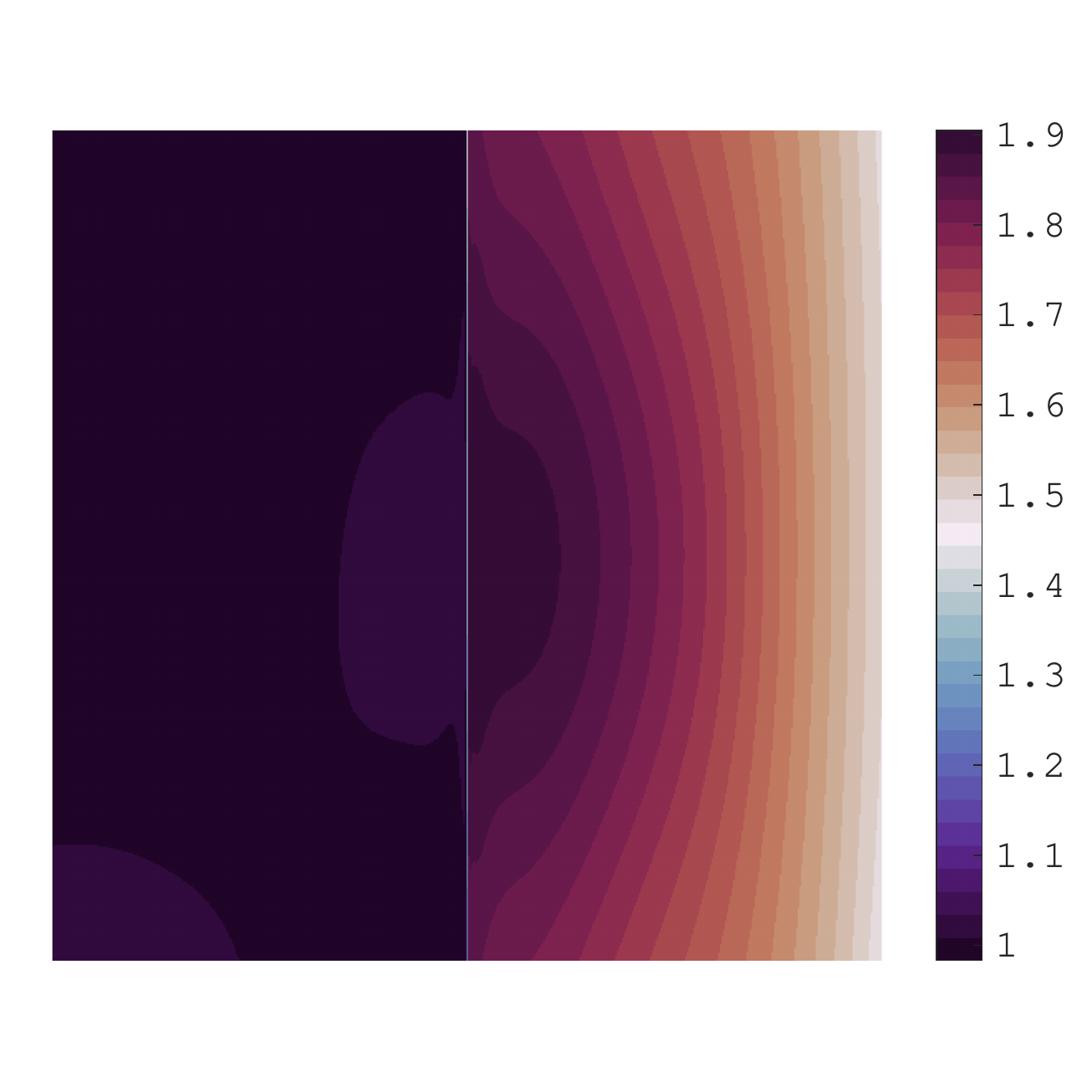}} 
\subfloat[$|\hat{\gamma}-\gamma^{\dagger}|$]
{\includegraphics[width=0.30\linewidth]{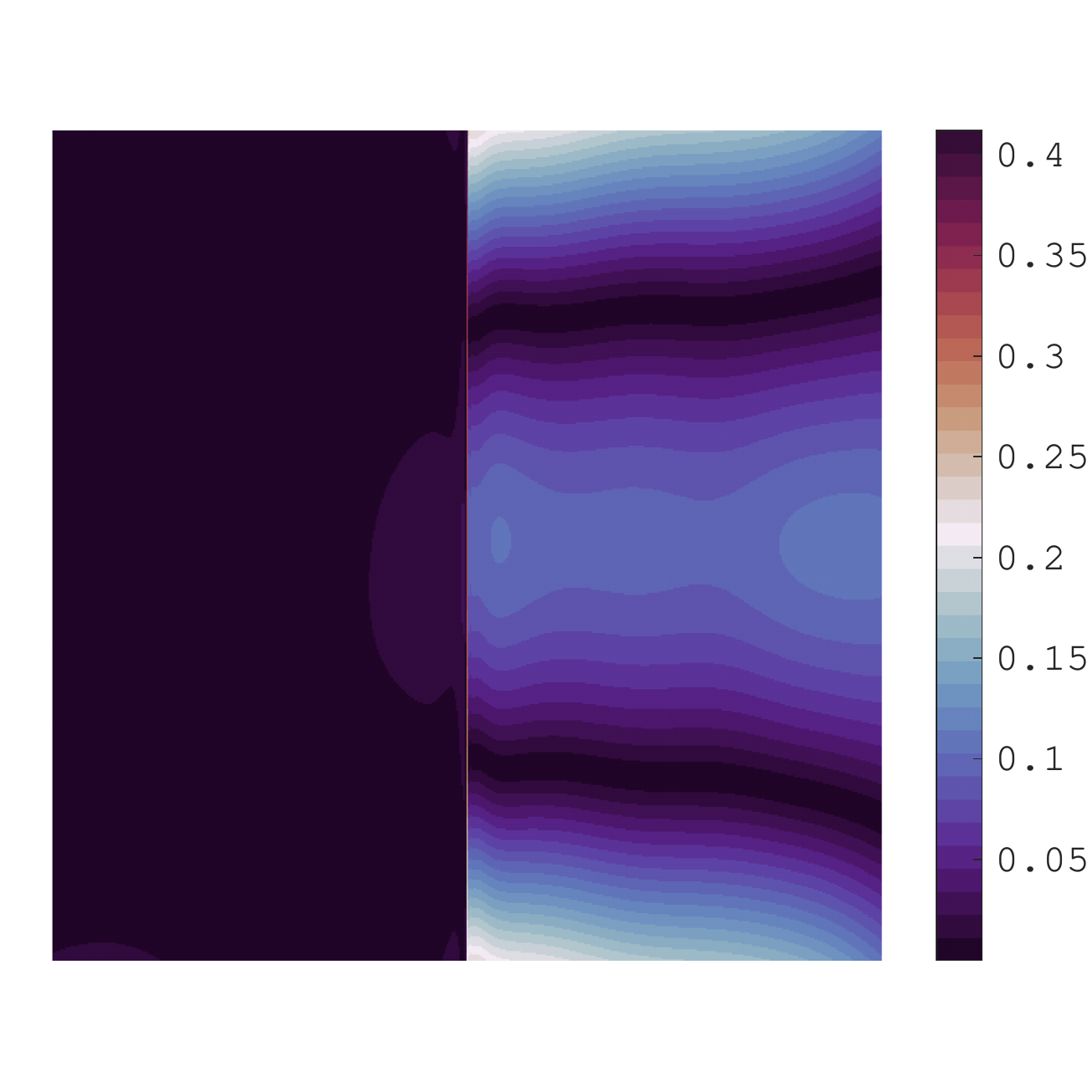}} 
\\
\subfloat[$\gamma^{\dagger}$]
{\includegraphics[width=0.30\linewidth]{./figures/example02gamma\_exact.pdf}}
\subfloat[$\hat{\gamma}$]
{\includegraphics[width=0.30\linewidth]{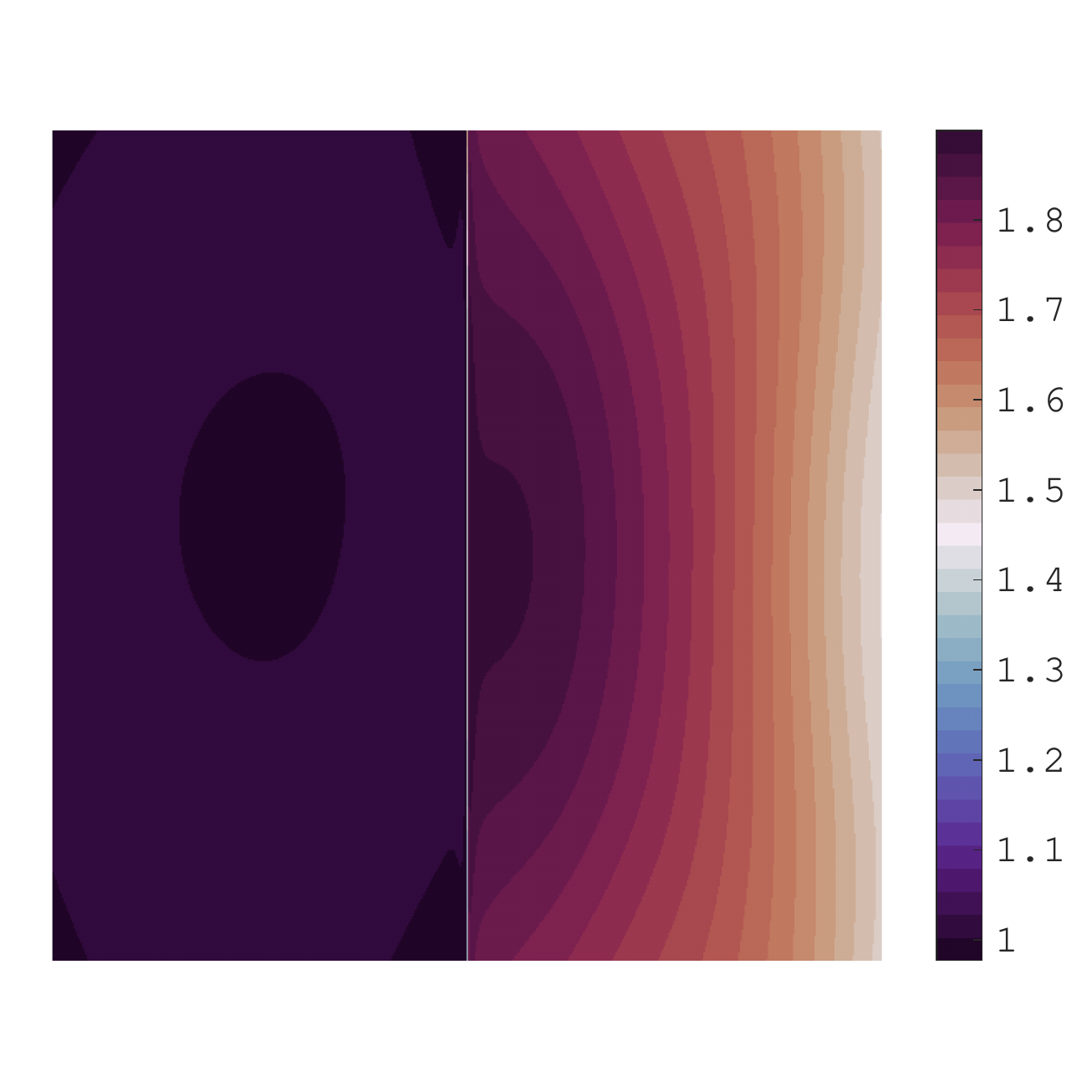}} 
\subfloat[$|\hat{\gamma}-\gamma^{\dagger}|$]
{\includegraphics[width=0.30\linewidth]{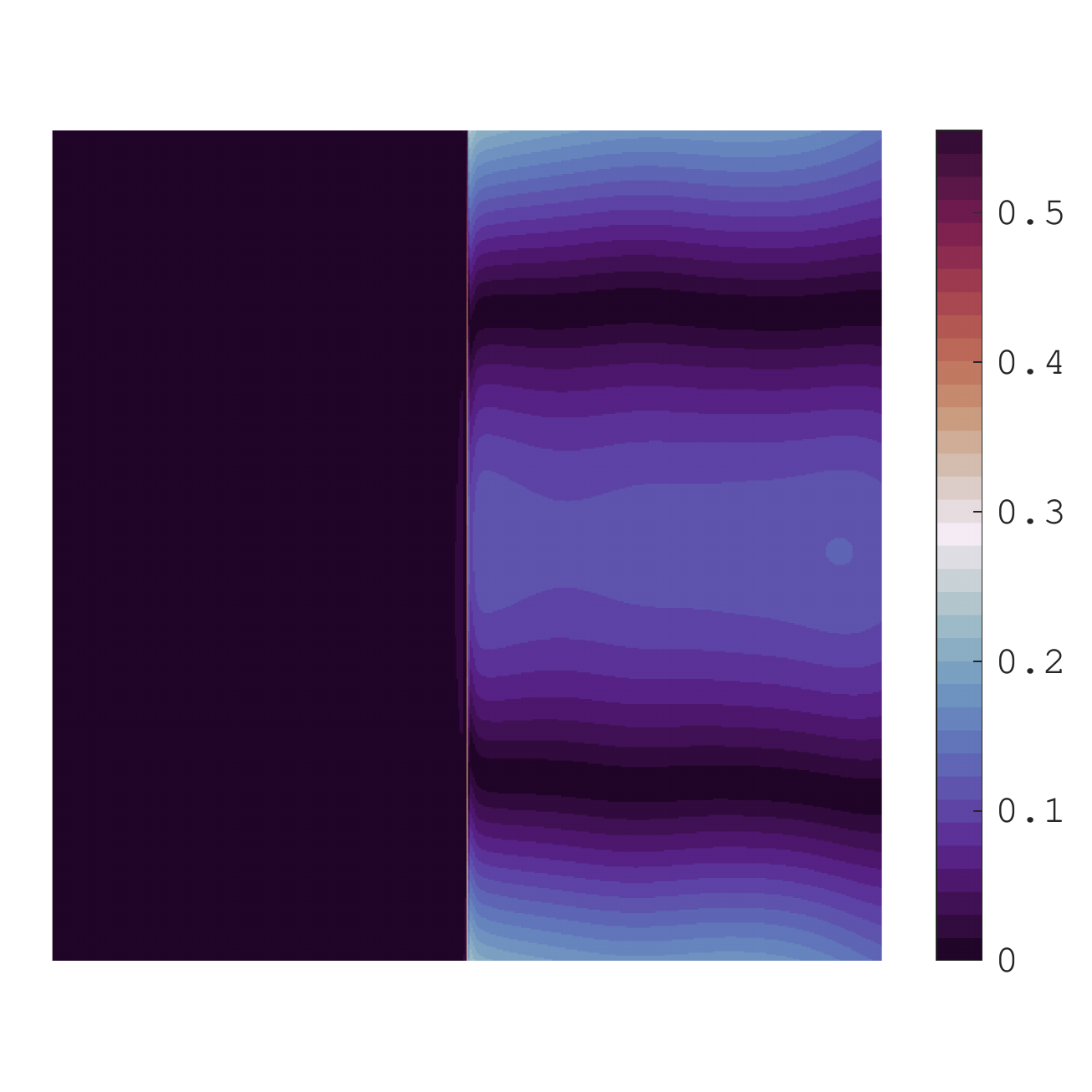}} 
\caption{The ground truth conductivity $\gamma^{\dagger}$, reconstruction $\hat{\gamma}$ by our method, and the point-wise absolute error $|\hat{\gamma}-\gamma^{\dagger}|$ in \cref{example:dc}. $\delta=1\%$ (top) $\delta=10\%$ (middle) $\delta=20\%$ (bottom).}
\label{fig:example2:gamma}
\end{figure}
\begin{figure}
\centering
\subfloat[$u^{\dagger}$]
{\includegraphics[width=0.30\linewidth]{./figures/example02solution\_exact.pdf}}
\subfloat[$\hat{u}$]
{\includegraphics[width=0.30\linewidth]{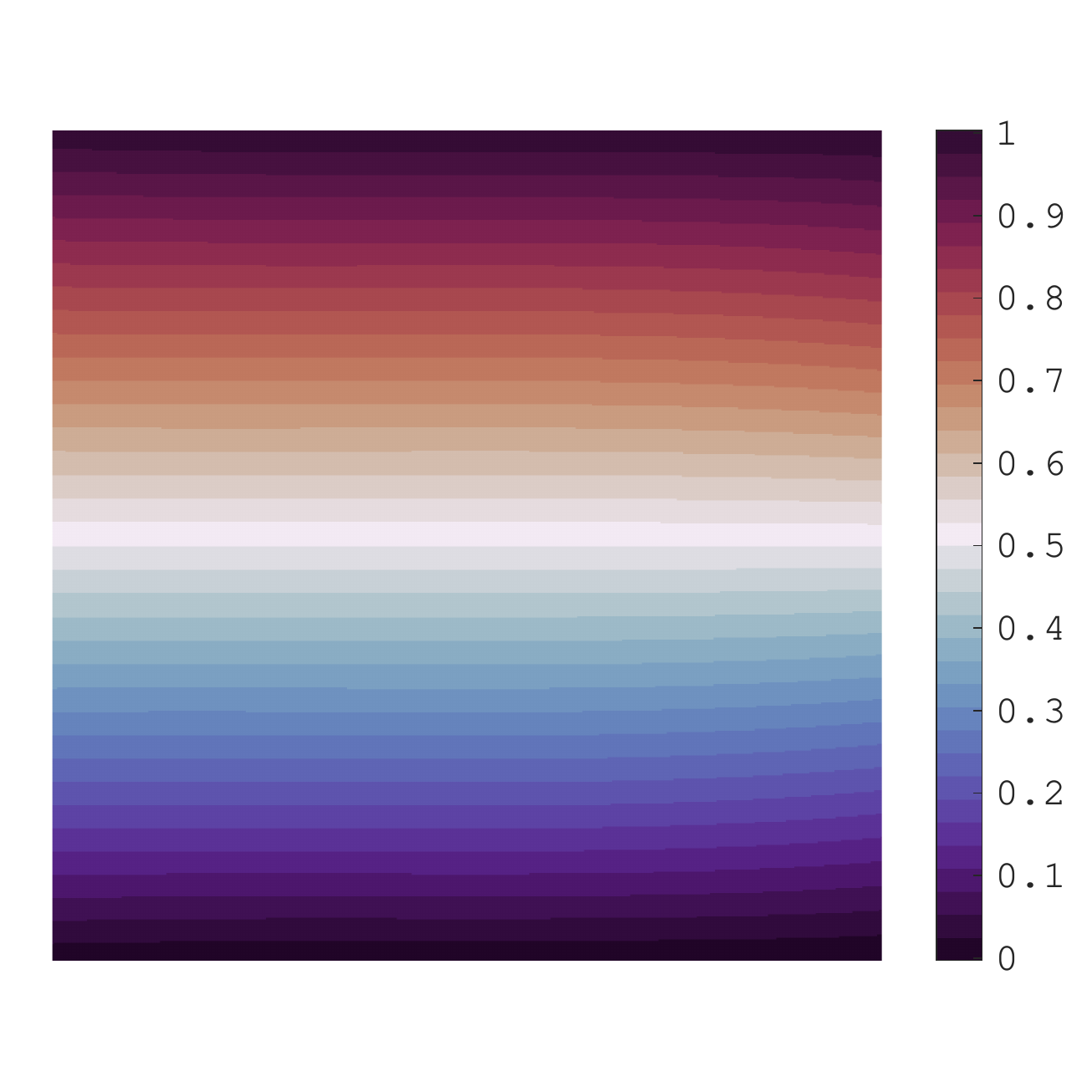}} 
\subfloat[$|\hat{u}-u^{\dagger}|$]
{\includegraphics[width=0.30\linewidth]{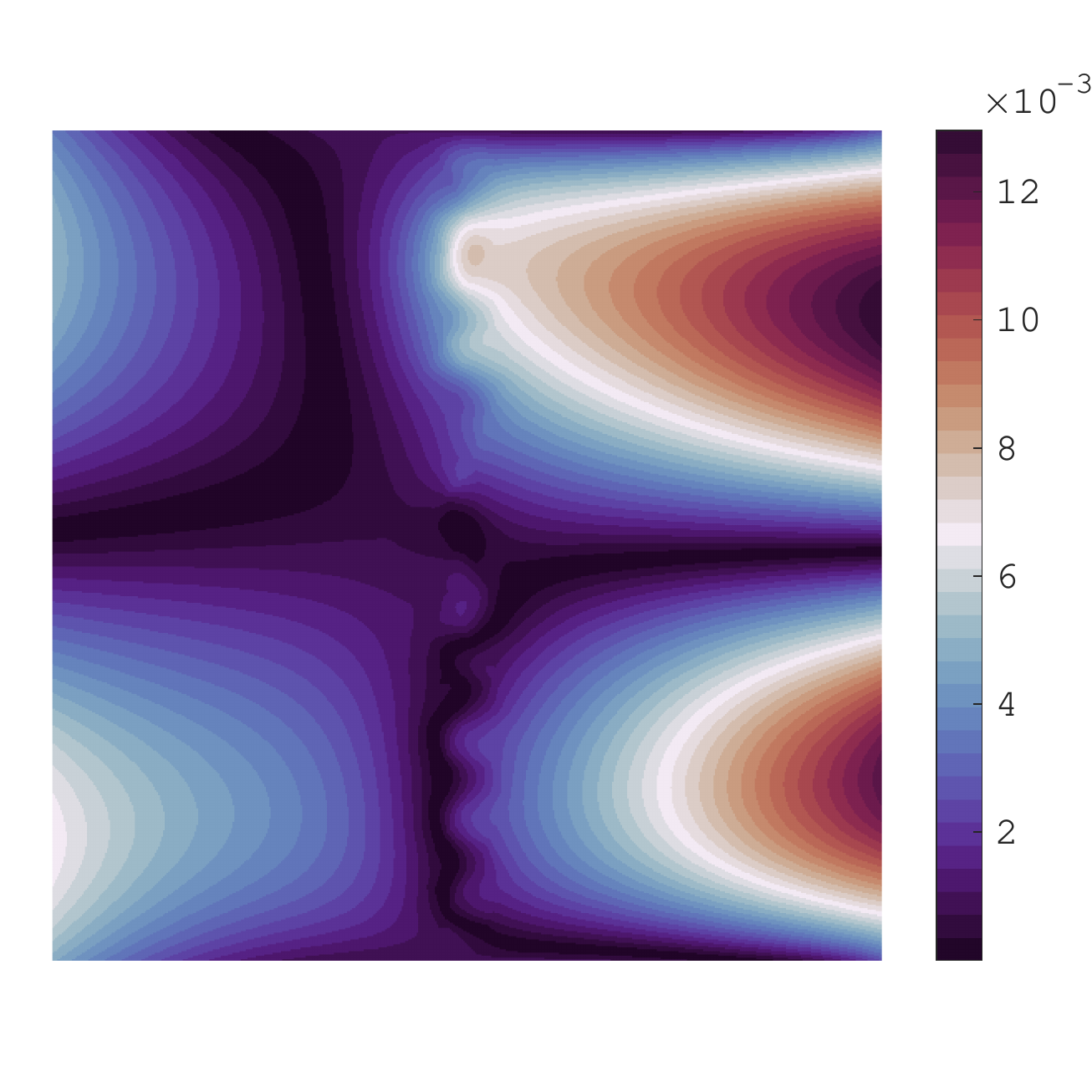}} 
\\
\subfloat[$u^{\dagger}$]
{\includegraphics[width=0.30\linewidth]{./figures/example02solution\_exact.pdf}}
\subfloat[$\hat{u}$]
{\includegraphics[width=0.30\linewidth]{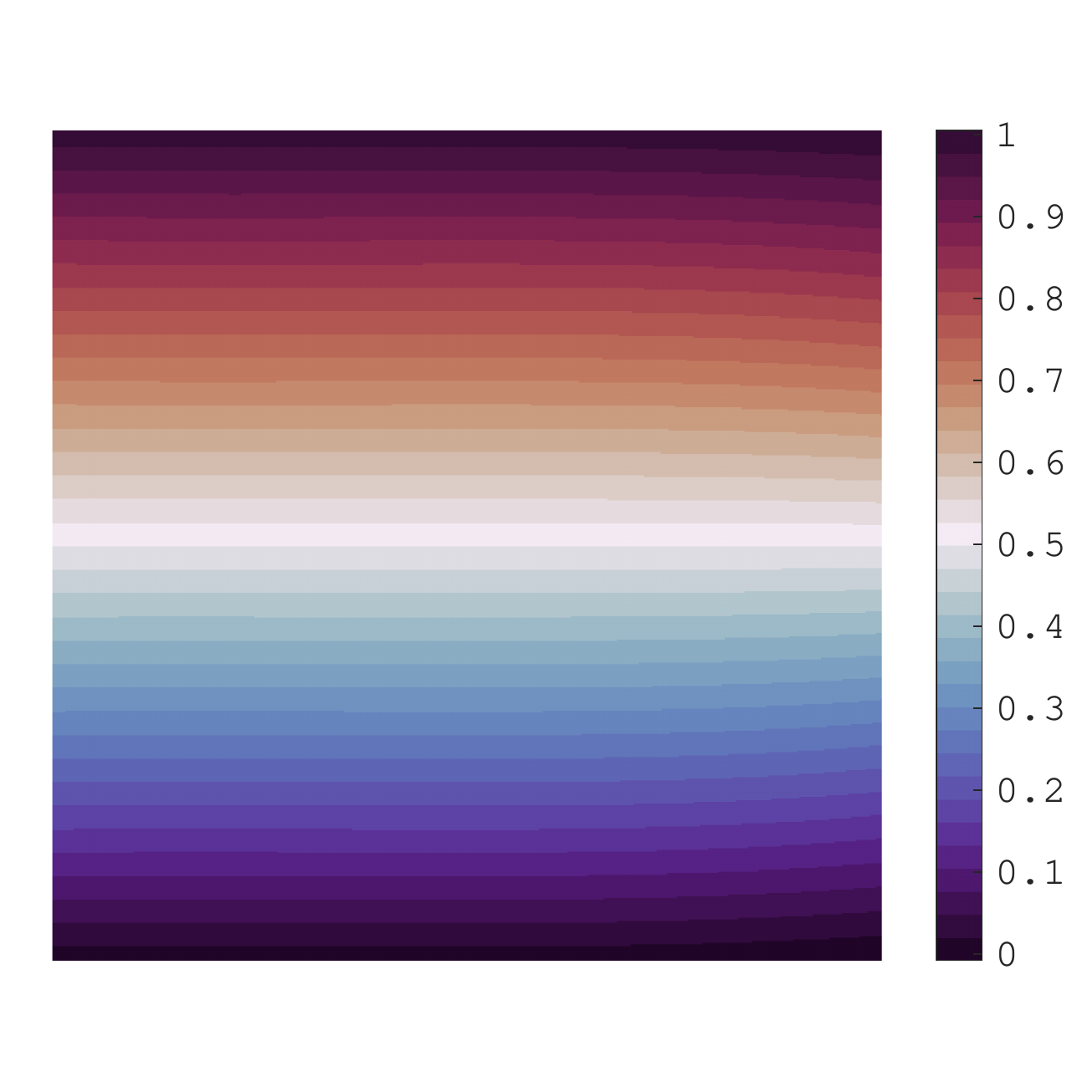}} 
\subfloat[$|\hat{u}-u^{\dagger}|$]
{\includegraphics[width=0.30\linewidth]{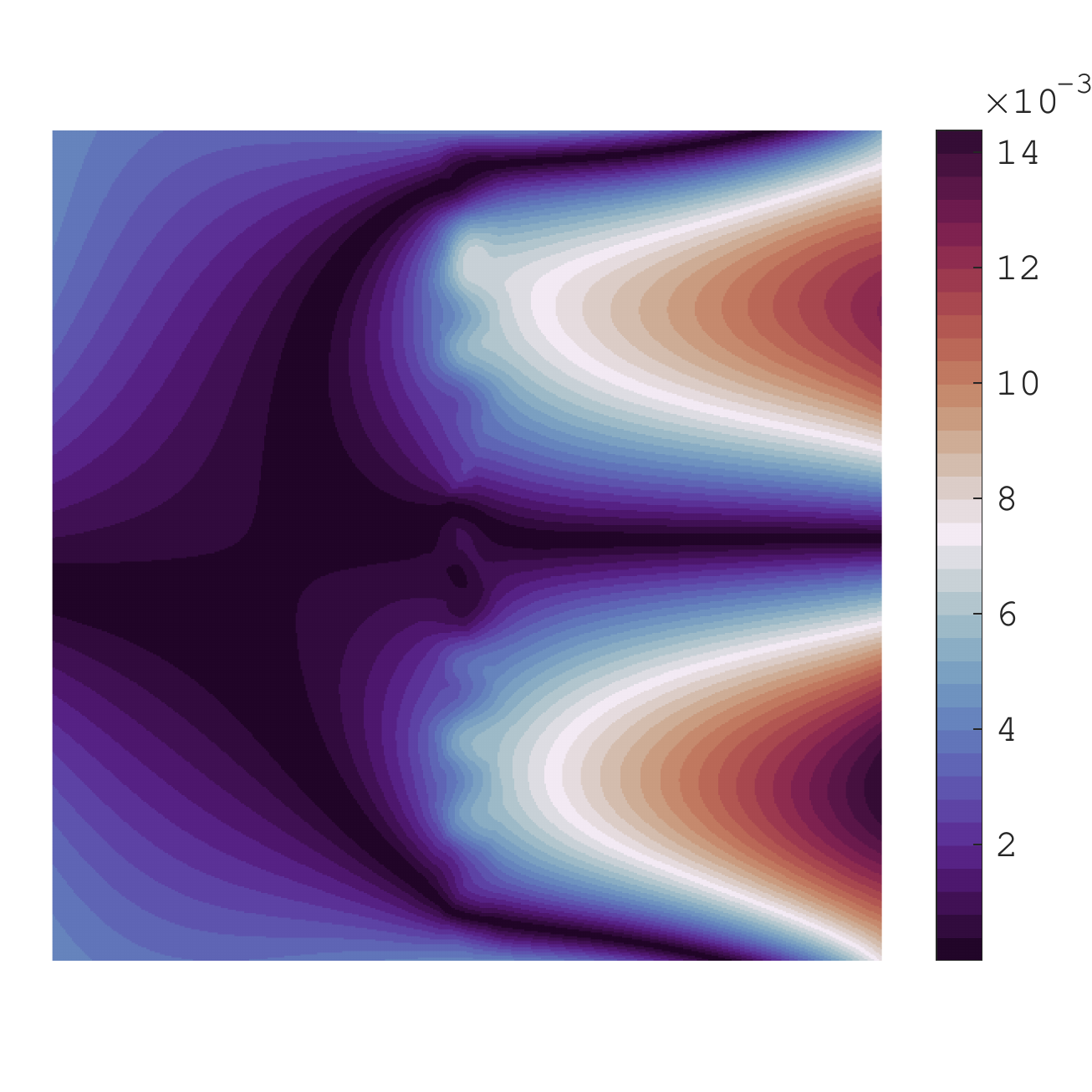}} 
\\
\subfloat[$u^{\dagger}$]
{\includegraphics[width=0.30\linewidth]{./figures/example02solution\_exact.pdf}}
\subfloat[$\hat{u}$]
{\includegraphics[width=0.30\linewidth]{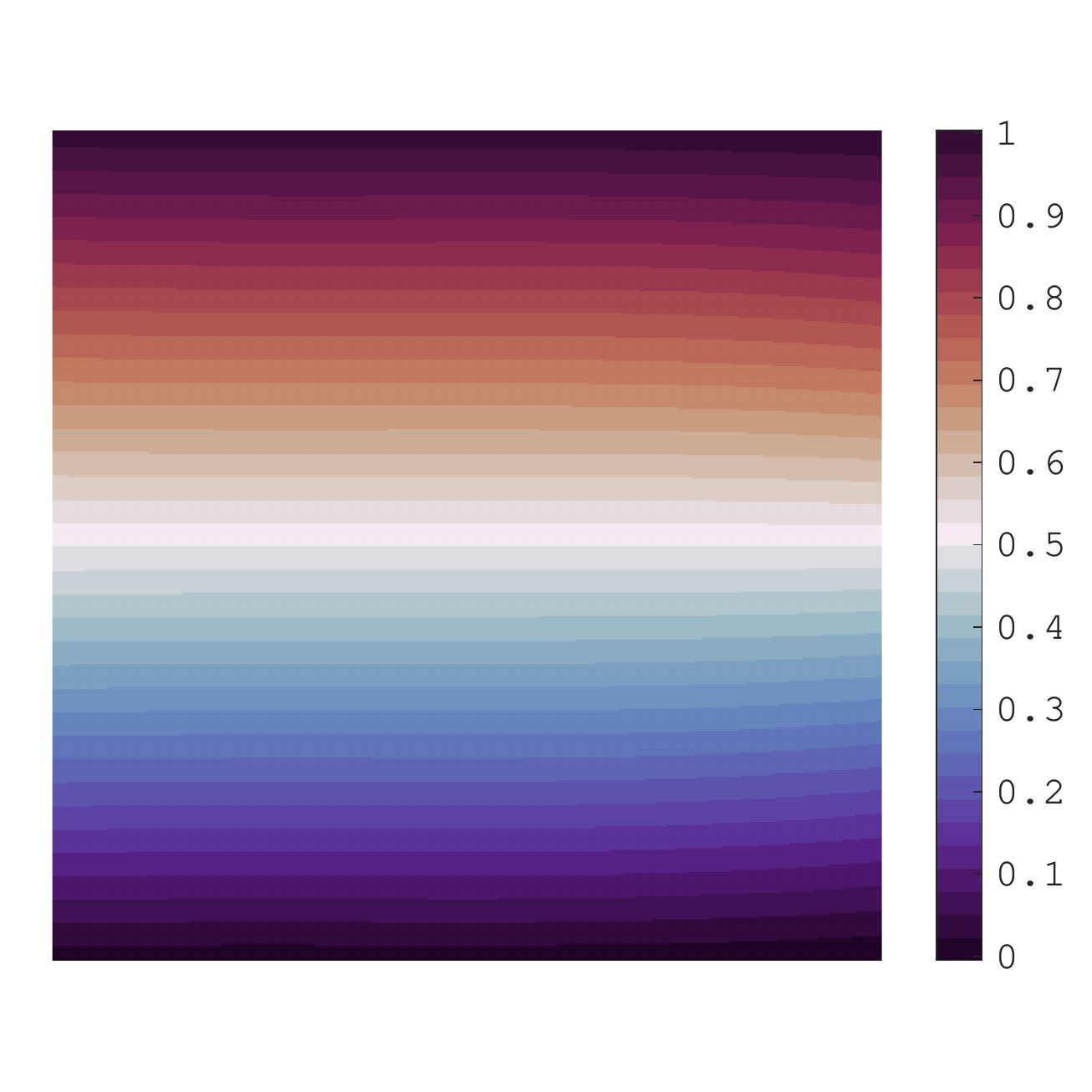}} 
\subfloat[$|\hat{u}-u^{\dagger}|$]
{\includegraphics[width=0.30\linewidth]{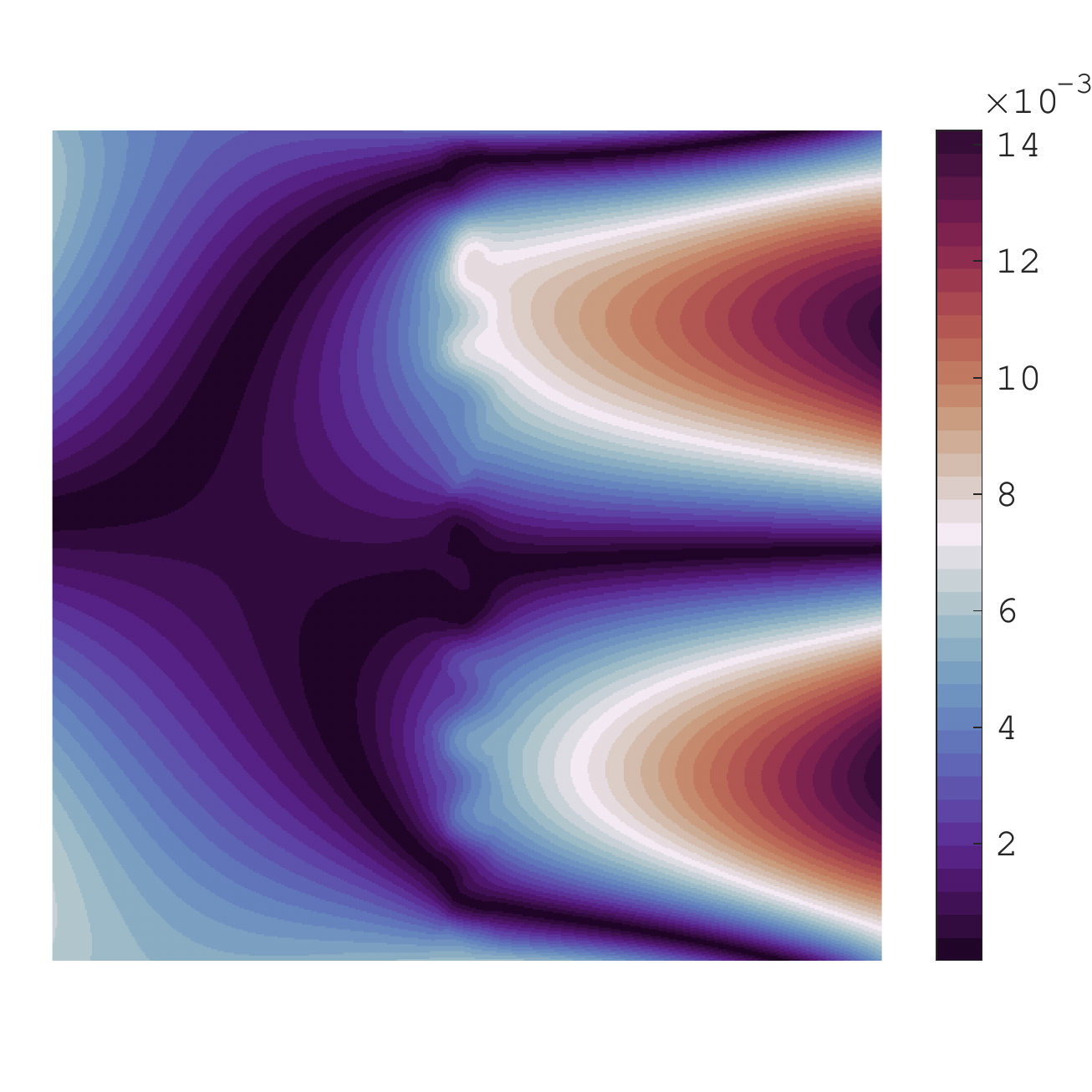}} 
\caption{The ground truth voltage $u^{\dagger}$, reconstruction $\hat{u}$ by our method, and the point-wise absolute error $|\hat{u}-u^{\dagger}|$ in \cref{example:dc}. $\delta=1\%$ (top) $\delta=10\%$ (middle) $\delta=20\%$ (bottom).}
\label{fig:example2:u}
\end{figure}

\begin{example}[CDII with disjoint modes in conductivity]\label{example:disjoint}
Let $\Omega_{1}=\{(x, y):100(x-0.3)^2+36(y-0.7)^2-72(x-0.3)(y-0.7)<1\}$ and $\Omega_{2}=\{(x, y): 36(x-0.6)^2+36(y-0.4)^2<1\}$. We set the ground true conductivity as  
\begin{equation*}
\gamma^{\dagger}(x,y)=1+\chi_{\Omega_{1}}-\chi_{\Omega_{2}}.
\end{equation*}
\end{example}
\par In this example, we set the noise level $\delta=10\%$ and use the TV-regularization. Referring to (2.7) in \cite{Jin2022imaging}, we replace the term $|\nabla\gamma(x)|$ in TV-regularization with Huber function
\begin{equation*}
h(\gamma)=
\begin{cases}
|\nabla\gamma(x)|, & |\nabla\gamma(x)|\geq\zeta, \\
\frac{|\nabla\gamma(x)|^{2}}{2\zeta}+\frac{\zeta}{2}, & \text{otherwise},
\end{cases}
\end{equation*}
where $\zeta>0$ is a small constant. We set $\zeta=0.001$ in this example.
\par The noisy data $a^{\delta}$ and it reconstructed counterpart $\hat{a}$ are illustrated in \cref{fig:example3:observation}. \cref{fig:example3:gamma} shows a comparison between ground truth conductivity $\gamma^{\dagger}$ and its reconstruction $\hat{\gamma}$. As can be seen, our method determines the shape of the different modes precisely, and thanks to the TV-regularization, it can capture the nature of the piece-wise constants of the ground truth conductivity, while the $L^{2}$-regularization can not. A comparison between the ground truth voltage $u^{\dagger}$ and its reconstruction $\hat{u}$ is shown in \cref{fig:example3:u}. 
\begin{table}
\caption{The relative $L^{2}$-error of the recovered conductivity $\hat{\gamma}$ and voltage $\hat{u}$ obtained by $L^{2}$-regularization and TV-regularization in \cref{example:disjoint}.}
\centering
\begin{tabular}{lcc}
\toprule
 & $L^{2}$-regularization & TV-regularization  \\
\midrule
$\err(\gamma)$ & $7.19 \times 10^{-2}$ & $5.98 \times 10^{-2}$ \\
$\err(u)$ & $7.62 \times 10^{-3}$ & $7.13 \times 10^{-3}$ \\
\bottomrule
\end{tabular}
\label{tab:example3:error}
\end{table}
\begin{figure}
\centering
\subfloat[$a^{\dagger}$]
{\includegraphics[width=0.30\linewidth]{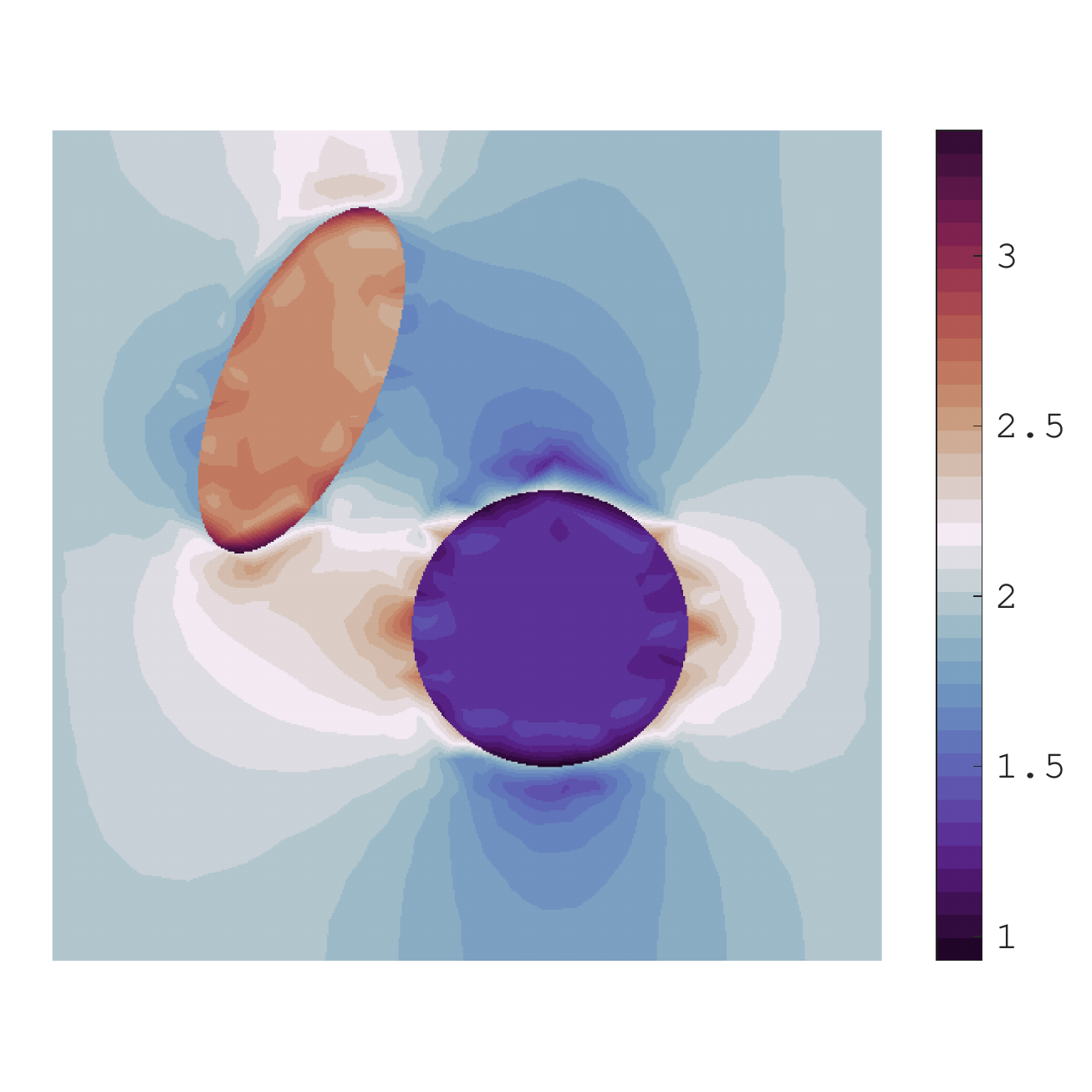}}
\subfloat[$a^{\delta}$]
{\includegraphics[width=0.30\linewidth]{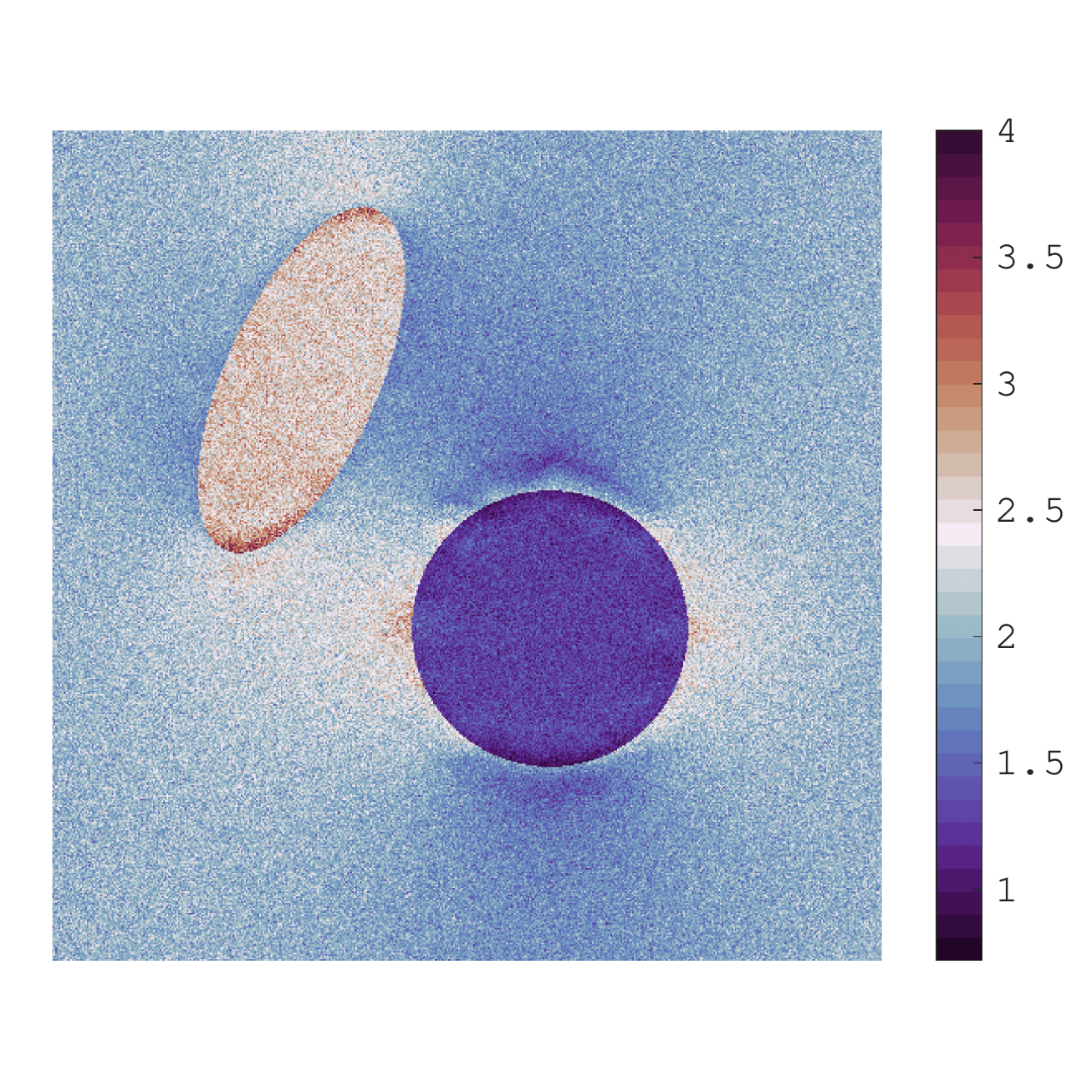}}
\subfloat[$\hat{a}$]
{\includegraphics[width=0.30\linewidth]{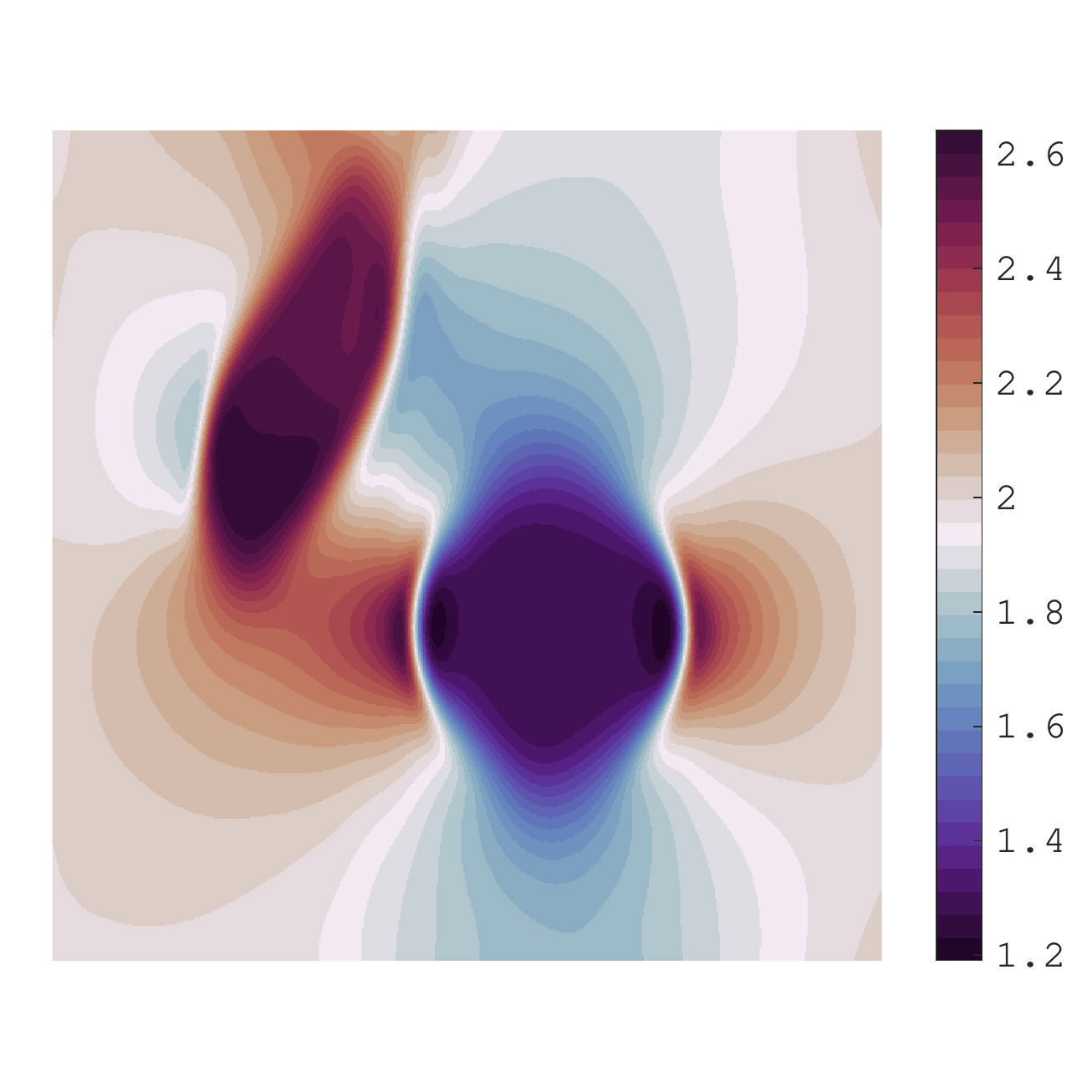}}
\caption{The ground truth $a^{\dagger}$, noisy measurements $a^{\delta}$ with $\delta=10\%$, and reconstruction $\hat{a}$ by our method in \cref{example:disjoint}. TV-regularization}
\label{fig:example3:observation}
\end{figure}
\begin{figure}
\centering
\subfloat[$\gamma^{\dagger}$]
{\includegraphics[width=0.30\linewidth]{./figures/example03gamma\_exact.pdf}}
\subfloat[$\hat{\gamma}$]
{\includegraphics[width=0.30\linewidth]{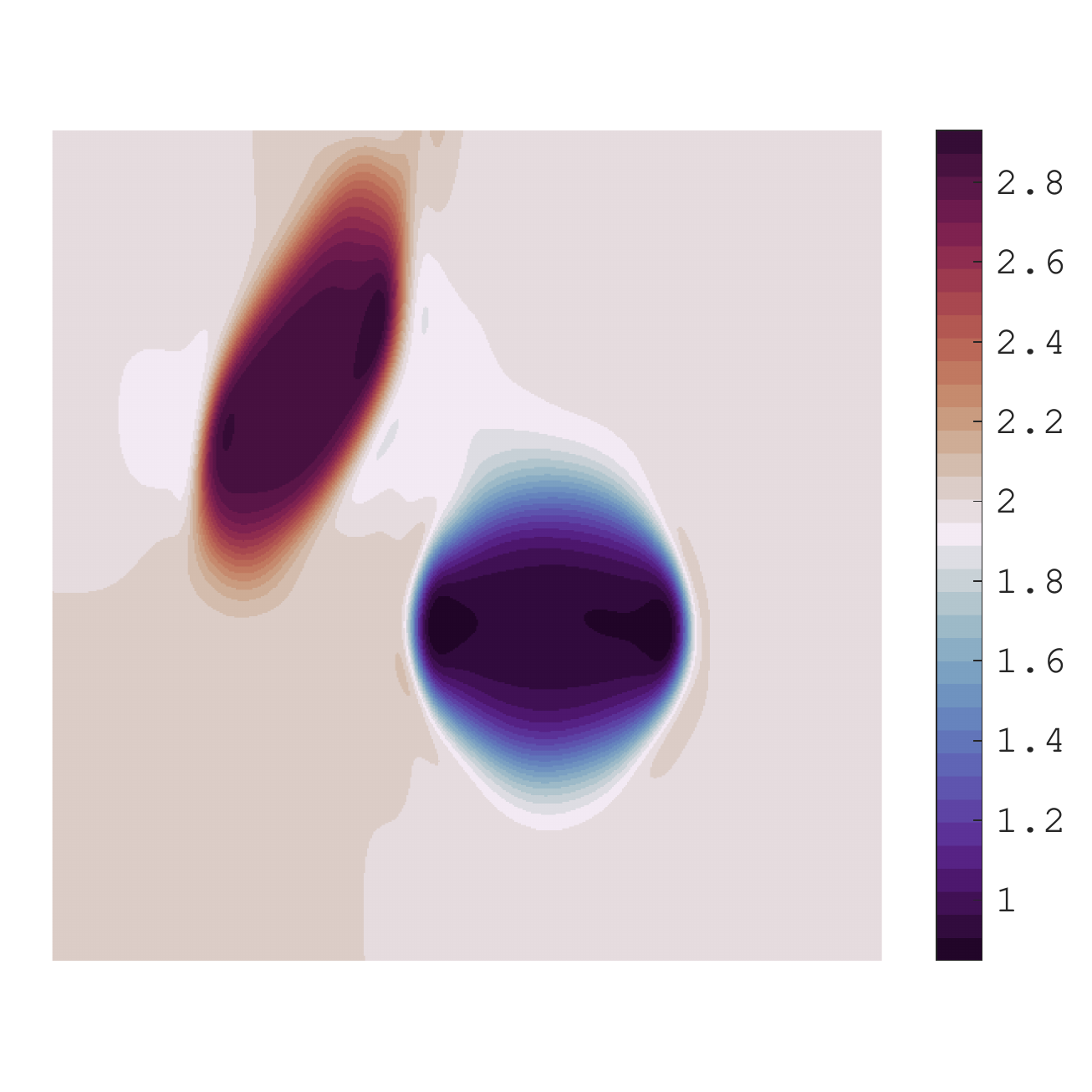}} 
\subfloat[$|\hat{\gamma}-\gamma^{\dagger}|$]
{\includegraphics[width=0.30\linewidth]{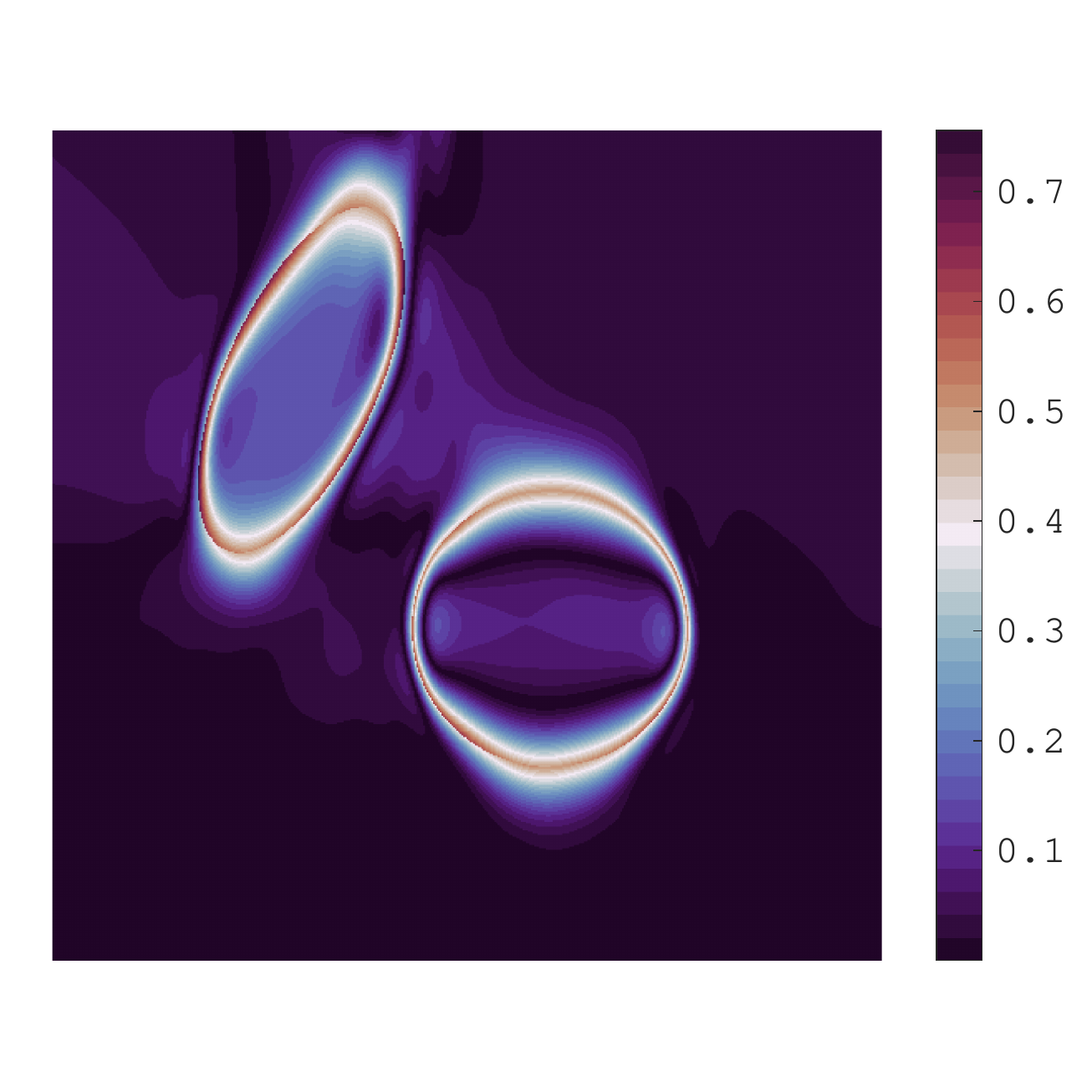}} 
\\
\subfloat[$\gamma^{\dagger}$]
{\includegraphics[width=0.30\linewidth]{./figures/example03gamma\_exact.pdf}}
\subfloat[$\hat{\gamma}$]
{\includegraphics[width=0.30\linewidth]{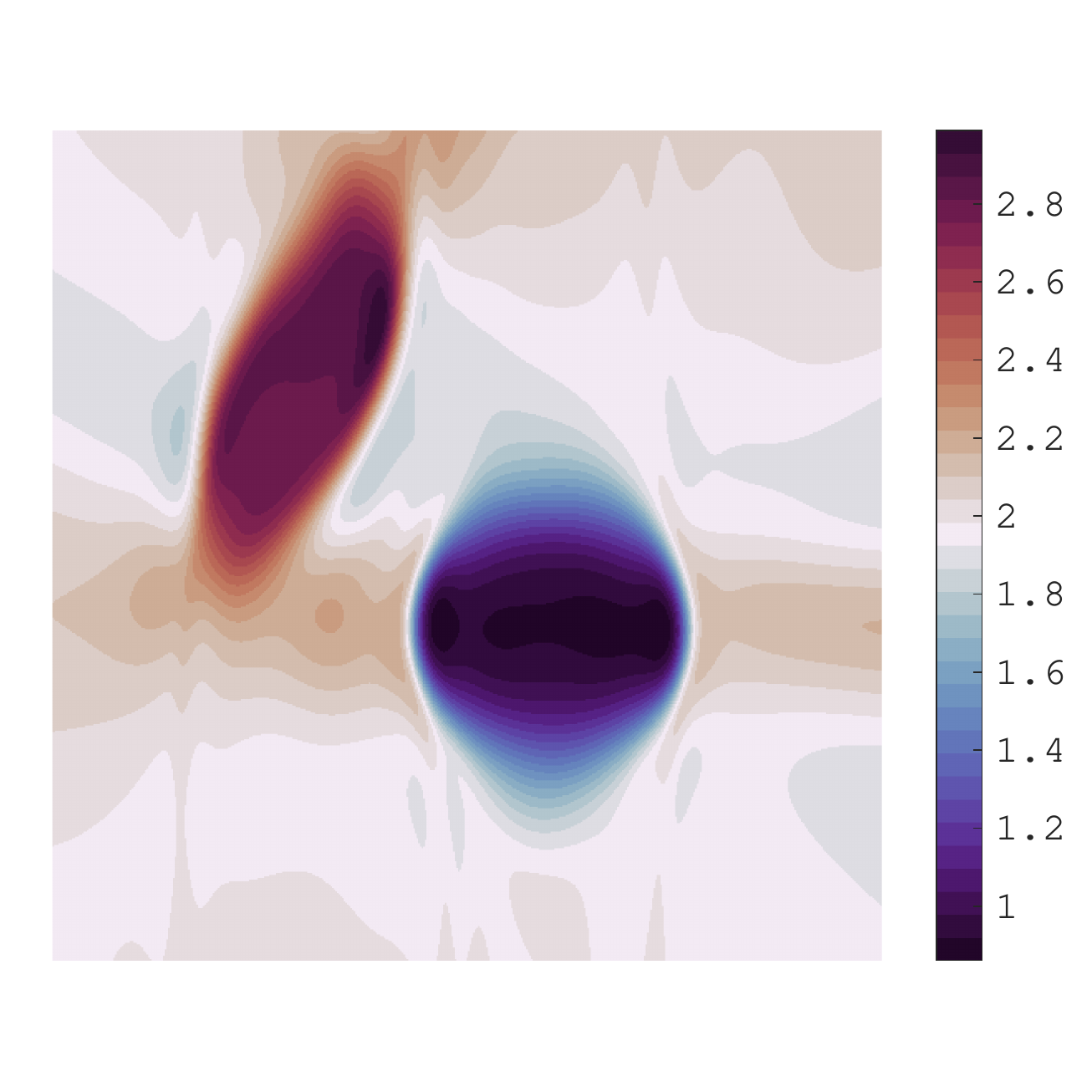}} 
\subfloat[$|\hat{\gamma}-\gamma^{\dagger}|$]
{\includegraphics[width=0.30\linewidth]{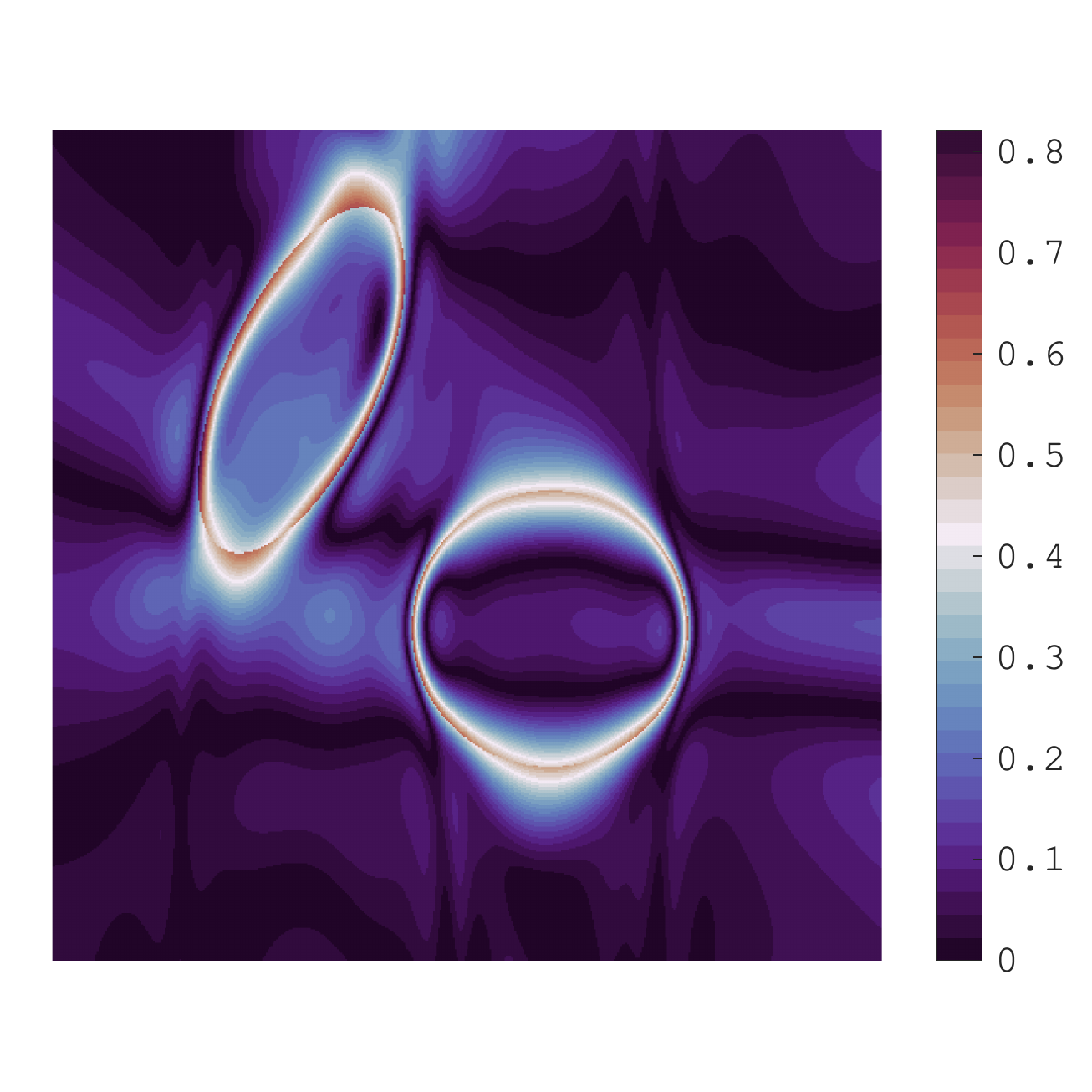}} 
\caption{The ground truth conductivity $\gamma^{\dagger}$, reconstruction $\hat{\gamma}$ by our method, and the point-wise absolute error $|\hat{\gamma}-\gamma^{\dagger}|$ in  \cref{example:disjoint}. $\delta=10\%$. TV-regularization (top) $L^{2}$-regularization (bottom)}
\label{fig:example3:gamma}
\end{figure}
\begin{figure}
\centering
\subfloat[$u^{\dagger}$]
{\includegraphics[width=0.30\linewidth]{./figures/example03solution\_exact.pdf}}
\subfloat[$\hat{u}$]
{\includegraphics[width=0.30\linewidth]{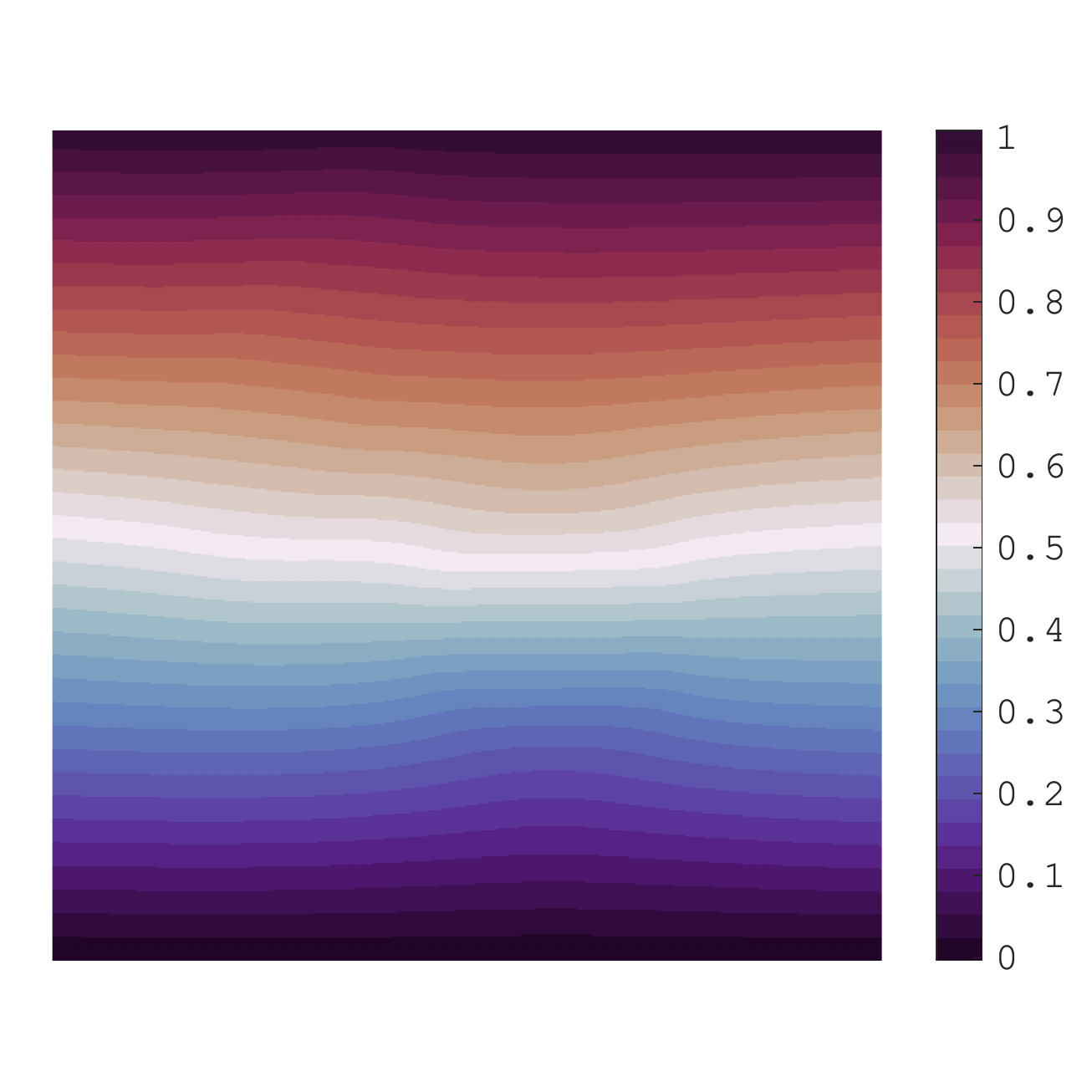}} 
\subfloat[$|\hat{u}-u^{\dagger}|$]
{\includegraphics[width=0.30\linewidth]{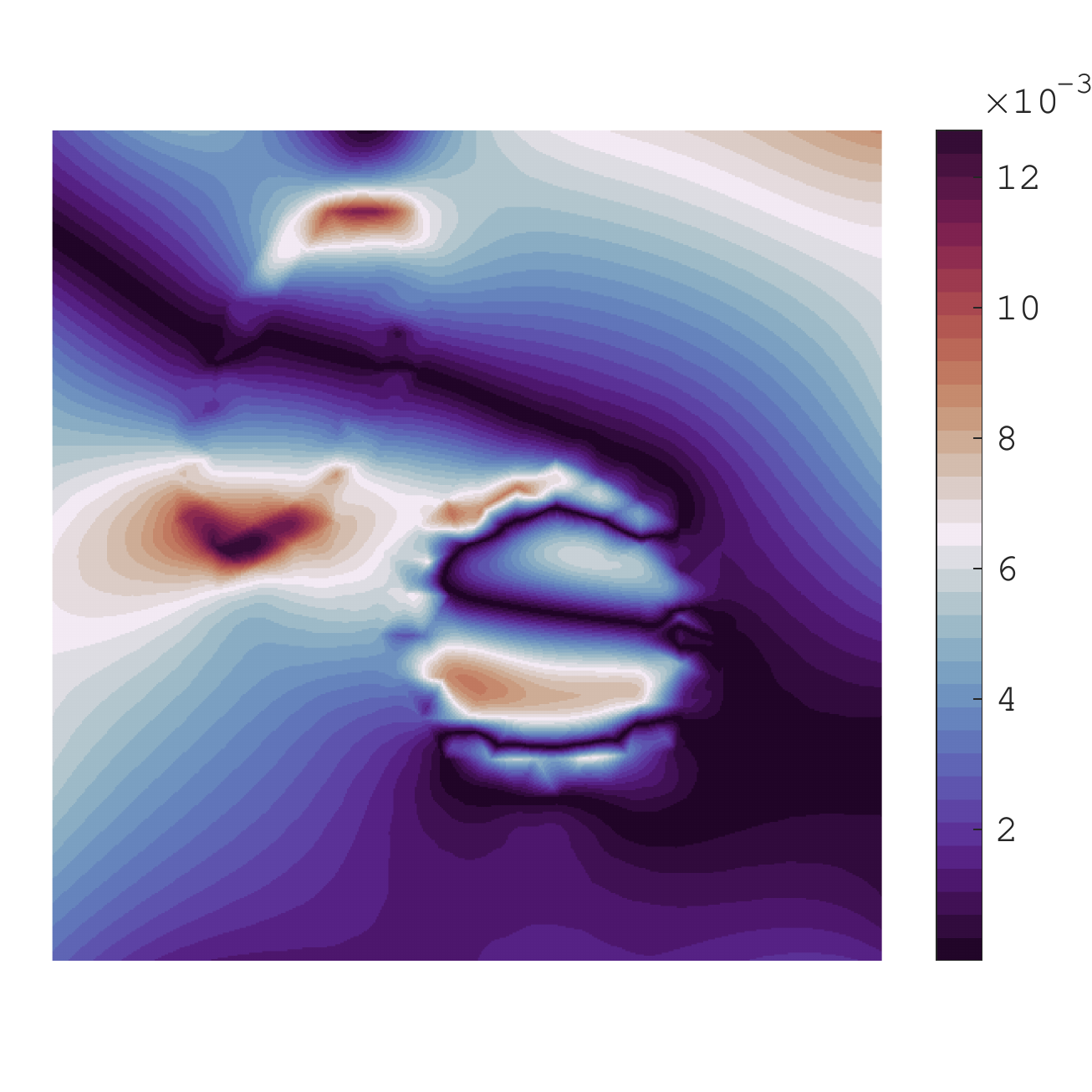}} 
\caption{The ground truth voltage $u^{\dagger}$, reconstruction $\hat{u}$ by our method, and the point-wise absolute error $|\hat{u}-u^{\dagger}|$ in \cref{example:disjoint}. $\delta=10\%$ TV-regularization}
\label{fig:example3:u}
\end{figure}

\section{Conclusions}\label{conc}
\par In this paper, we propose a PINNs-based method for solving CDII. We first use two neural networks to represent the conductivity and voltage and then construct the loss function over the measurement data. Our method directly reconstructs the conductivity and voltage by minimizing the loss function, avoiding iterative updates. We present an error estimate and give a convergence rate. At the same time, our error analysis can provide a way to choose the structure of the neural networks. The stability of CDII will be taken into account in future work. The method has shown robustness to noise in numerical experiments. Finally, the coupled-PINNs method can be directly extended to general inverse problems, although we have only considered CDII in this paper. Simultaneously, the error analysis framework and techniques are readily applicable.

\section{Proofs}\label{append}
\subsection{The proof of the error decomposition}\label{sec:appendix:errdec}
\par Inspired by the proof of Lemma 4 in \cite{Johannes2020Nonparametric}, we prove \cref{lemma:err:dec} in this section.
\par For an estimator $(\tilde{\gamma}_{n},\tilde{u}_{n})\in\calF_{\gamma}\times\calF_{u}$ depending upon a data sample $S$, we introduce its expected empirical risk
\begin{equation*}
R_{n}(\tilde{\gamma}_{n},\tilde{u}_{n}):=\mathbb{E}_{S}\Big[\frac{1}{n}\sum_{i=1}^{n}(\tilde{\gamma}_{n}|\nabla\tilde{u}_{n}|(X_{i})-\gamma^{*}|\nabla u^{*}|(X_{i}))^{2}\Big]+\mathbb{E}_{S}\Big[G_{n}(\tilde{\gamma}_{n},\tilde{u}_{n})\Big],
\end{equation*}
and recall its optimization error
\begin{equation*}
\Delta_{n}(\tilde{\gamma}_{n},\tilde{u}_{n})=\mathbb{E}_{S}\Big[L_{n}(\tilde{\gamma}_{n},\tilde{u}_{n})-L_{n}(\hat{\gamma}_{n},\hat{u}_{n})\Big].
\end{equation*}
\begin{lemma}\label{lemma:errdec:1}
Let $\calF_{\gamma}$ and $\calF_{u}$ be two function classes with $N_{\gamma}=N(\varepsilon,\calF_{\gamma},\|\cdot;W^{1,\infty}(U)\|)$ and $N_{u}=N(\varepsilon,\calF_{u},\|\cdot;W^{2,\infty}(U)\|)$. Suppose \cref{ass:gammau:bound,ass:fg:bound} are fulfilled. Then it holds for each estimator $(\tilde{\gamma}_{n},\tilde{u}_{n})$ taking values in $\calF_{\gamma}\times\calF_{u}$ that
\begin{align*}
&R(\tilde{\gamma}_{n},\tilde{u}_{n})\lesssim R_{n}(\tilde{\gamma}_{n},\tilde{u}_{n})+(B_{\gamma}^{2}B_{u}^{2}+B_{f}^{2})\frac{\log(N_{\gamma}N_{u})}{n}+B_{\gamma}B_{u}(B_{\gamma}+B_{u})\varepsilon,
\end{align*}
with $\delta>0$ and $n$ large enough.
\end{lemma}
\begin{proof}
The proof is divided to four parts as follows. 
\par\noindent\emph{Step 1.} For any $\varepsilon>0$, let $\calC_{\gamma}=\{\gamma_{k}:k=1,\ldots,N_{\gamma}\}$ be the minimal $\varepsilon$-cover of $\calF_{\gamma}$ with respect to $\|\cdot\|_{W^{1,\infty}(U)}$, and let $\calC_{u}=\{u_{\ell}:\ell=1,\ldots,N_{u}\}$ be the minimal $\varepsilon$-cover of $\calF_{u}$ with respect to $\|\cdot\|_{W^{2,\infty}(U)}$. Then for an estimator $(\tilde{\gamma}_{n},\tilde{u}_{n})\in\calF_{\gamma}\times\calF_{u}$, there exists $(\gamma_{k^{*}},u_{\ell^{*}})\in\calC_{\gamma}\times\calC_{u}$, such that
\begin{equation*}
\|\gamma_{k^{*}}-\tilde{\gamma}_{n}\|_{W^{1,\infty}(U)}\leq\varepsilon\quad\text{and}\quad\|u_{\ell^{*}}-\tilde{u}_{n}\|_{W^{2,\infty}(U)}\leq\varepsilon.
\end{equation*} 
Without loss of generality, we can assume that $\|\gamma_{k}\|_{W^{1,\infty}(U)}\leq B_{\gamma}$ and $\|u_{\ell}\|_{W^{2,\infty}(U)}\leq B_{u}$. Generate a ghost sample $S^{\prime}=\{(X_{i}^{\prime},Y_{i}^{\prime})\}_{i=1}^{n}\cup\{\bar{X}_{i}^{\prime}\}_{i=1}^{n}$ independent of $S$, where $(X_{i}^{\prime},Y_{i}^{\prime})$ and and $\bar{X}_{i}^{\prime}$ are $n$ independent copies of $(X,Y)$ and $\bar{X}_{i}^{\prime}$, respectively. Then we have
\begin{equation}\label{eq:RsubRn}
\begin{aligned}
|R(\tilde{\gamma}_{n},\tilde{u}_{n})-R_{n}(\tilde{\gamma}_{n},\tilde{u}_{n})|
&=\Big|\mathbb{E}_{S^{\prime}}\Big[\frac{1}{n}\sum_{i=1}^{n}(\tilde{\gamma}_{n}|\nabla\tilde{u}_{n}|(Z_{i}^{\prime})-\gamma^{*}|\nabla u^{*}|(Z_{i}^{\prime}))^{2}\Big]\\
&\quad-\mathbb{E}_{S}\Big[\frac{1}{n}\sum_{i=1}^{n}(\tilde{\gamma}_{n}|\nabla\tilde{u}_{n}|(Z_{i})-\gamma^{*}|\nabla u^{*}|(Z_{i}))^{2}\Big] \\
&\quad+\mathbb{E}_{S^{\prime}}\Big[\frac{1}{n}\sum_{i=1}^{n}g(\tilde{\gamma}_{n},\tilde{u}_{n},Z_{i}^{\prime})\Big]-\mathbb{E}_{S}\Big[\frac{1}{n}\sum_{i=1}^{n}g(\tilde{\gamma}_{n},\tilde{u}_{n},Z_{i})\Big]\Big| \\
&\leq\mathbb{E}_{S}\mathbb{E}_{S^{\prime}}\Big[\Big|\frac{1}{n}\sum_{i=1}^{n}\varphi_{k^{*},\ell^{*}}(Z_{i},Z_{i}^{\prime})\Big|\Big]+28B_{\gamma}B_{u}(B_{\gamma}+B_{u})\varepsilon,
\end{aligned}
\end{equation} 
where 
\begin{align*}
g(\gamma,u,X,\bar{X})&:=(\nabla\cdot(\gamma\nabla u)(X))^{2}+(Tu(\bar{X})-f(\bar{X}_{i}))^{2} \\
h(\gamma_{k},u_{\ell},X,\bar{X})&:=(\gamma_{k}|\nabla u_{\ell}|(X)-\gamma^{*}|\nabla u^{*}|(X))^{2}+g(\gamma_{k},u_{\ell},X,\bar{X}), \\
\varphi_{k,\ell}(X,\bar{X},X^{\prime},\bar{X}^{\prime})&:=h(\gamma_{k},u_{\ell},X,\bar{X})-h(\gamma_{k},u_{\ell},X^{\prime},\bar{X}^{\prime}),
\end{align*}
for $k=1,\ldots,N_{\gamma}$ and $\ell=1,\ldots,N_{u}$.
\par\noindent\emph{Step 2.} Set $r_{k,\ell}=\max\{A,\mathbb{E}_{(X,\bar{X})}^{1/2}[h(\gamma_{k},u_{\ell},X,\bar{X})]\}$ with $A>0$ for $k=1,\ldots,N_{\gamma}$ and $\ell=1,\ldots,N_{u}$, then 
\begin{equation}\label{eq:rstar2}
\begin{aligned}
r_{k^{*},\ell^{*}}^{2}&=\max\big\{A^{2},\mathbb{E}_{(X,\bar{X})}[h(\gamma_{k^{*}},u_{\ell^{*}},X,\bar{X})|S]\big\} \\
&\leq A^{2}+\mathbb{E}_{(X,\bar{X})}[h(\tilde{\gamma}_{n},\tilde{u}_{n},X,\bar{X})|S]+14B_{\gamma}B_{u}(B_{\gamma}+B_{u})\varepsilon,
\end{aligned}
\end{equation}
where the last inequality is follows from  $\|\tilde{\gamma}_{n}-\gamma_{k^{*}}\|_{W^{1,\infty}(U)}\leq\varepsilon$ and $\|\tilde{u}_{n}-u_{\ell^{*}}\|_{W^{2,\infty}(U)}\leq\varepsilon$. Define a random variable
\begin{equation}
T=\max_{k,\ell}\Big|\frac{1}{nr_{k,\ell}}\sum_{i=1}^{n}\varphi_{k,\ell}(X_{i},\bar{X}_{i},X_{i}^{\prime},\bar{X}_{i}^{\prime})\Big|,
\end{equation}
then by (\ref{eq:RsubRn}) we obtain by using Cauchy-Schwarz and AM-GM inequality 
\begin{equation}\label{eq:R:r2:T2}
\begin{aligned}
|R(\tilde{\gamma}_{n},\tilde{u}_{n})-R_{n}(\tilde{\gamma}_{n},\tilde{u}_{n})|&\leq\mathbb{E}_{S,S^{\prime}}[r_{k^{*},\ell^{*}}T]+28B_{\gamma}B_{u}(B_{\gamma}+B_{u})\varepsilon \\
&\leq\frac{1}{2}\mathbb{E}_{S}[r_{k^{*},\ell^{*}}^{2}]+\frac{1}{2}\mathbb{E}_{S,S^{\prime}}[T^{2}]+28B_{\gamma}B_{u}(B_{\gamma}+B_{u})\varepsilon.
\end{aligned}
\end{equation}
\par\noindent\emph{Step 3.} Now we turn to estimate $\mathbb{E}_{S}[r_{k^{*},\ell^{*}}^{2}]$ and $\mathbb{E}_{S,S^{\prime}}[T^{2}]$. By (\ref{eq:rstar2}), $\mathbb{E}_{S}[r_{k^{*},\ell^{*}}^{2}]$ can be evaluated by
\begin{equation}\label{eq:Er2}
\begin{aligned}
\mathbb{E}_{S}[r_{k^{*},\ell^{*}}^{2}]&\leq A^{2}+\mathbb{E}_{S}\mathbb{E}_{(X,\bar{X})}[h(\tilde{\gamma}_{n},\tilde{u}_{n},X,\bar{X})|S]+14B_{\gamma}B_{u}(B_{\gamma}+B_{u})\varepsilon \\
&\leq A^{2}+R(\tilde{\gamma}_{n},\tilde{u}_{n})+14B_{\gamma}B_{u}(B_{\gamma}+B_{u})\varepsilon.
\end{aligned}
\end{equation}
In order to estimate $\mathbb{E}_{S,S^{\prime}}[T^{2}]$, we bound the tail probability $\mathbb{P}\{T>t\}$ by Bernstein inequality. Observe that $\mathbb{E}[\varphi_{k,\ell}(X,\bar{X},X^{\prime},\bar{X}^{\prime})]=0$, $|\varphi_{k,\ell}(X,\bar{X},X^{\prime},\bar{X}^{\prime})|\leq 12B_{\gamma}^{2}B_{u}^{2}+2(B_{f}^{2})$, and
\begin{align*}
\sigma_{\varphi}^{2}&=\operatorname{Var}(\varphi_{k,\ell}(X,\bar{X},X^{\prime},\bar{X}^{\prime})=2\operatorname{Var}(h(\gamma_{k},u_{\ell},X,\bar{X}))\leq 2\mathbb{E}_{(X,\bar{X})}[h^{2}(\gamma_{k},u_{\ell},X,\bar{X})] \\
&\leq(12B_{\gamma}^{2}B_{u}^{2}+2B_{f}^{2})\mathbb{E}_{(X,\bar{X})}[h(\gamma_{k},u_{\ell},X,\bar{X})]\leq (12B_{\gamma}^{2}B_{u}^{2}+2B_{f}^{2})r_{k,\ell}^{2},
\end{align*}
where the second inequality is owing to $h(\gamma_{k},u_{\ell},X,\bar{X})\geq 0$. Then using Bernstein inequality gives
\begin{equation}\label{eq:Bernstein}
\begin{aligned}
&\mathbb{P}\Big\{\Big|\frac{\sum_{i=1}^{n}\varphi_{k,\ell}(X,\bar{X},X^{\prime},\bar{X}^{\prime})}{n}\Big|\geq t\Big\} \\
&\leq 2\exp\Big(-\frac{1}{4(6B_{\gamma}^{2}B_{u}^{2}+B_{f}^{2})}\cdot\frac{nt^{2}}{t/3+r_{k,\ell}^{2}}\Big).
\end{aligned}
\end{equation}
For ease of notations, we denote $B=4(6B_{\gamma}^{2}B_{u}^{2}+B_{f}^{2})$. Since $A\leq\min_{k,\ell}r_{k,\ell}$, we obtain by (\ref{eq:Bernstein})
\begin{align*}
\mathbb{P}\{T\geq t\}
&\leq\sum_{k=1}^{N_{\gamma}}\sum_{\ell=1}^{N_{u}}\mathbb{P}\Big\{\Big|\frac{\sum_{i=1}^{n}\varphi_{k,\ell}(Z_{i},Z_{i}^{\prime})}{nr_{k,\ell}}\Big|\geq t\Big\} \\
&\leq N_{\gamma}N_{u}\max_{k,\ell}\mathbb{P}\Big\{\Big|\frac{\sum_{i=1}^{n}\varphi_{k,\ell}(Z_{i},Z_{i}^{\prime})}{n}\Big|\geq tr_{k,\ell}\Big\} \\
&\leq 2N_{\gamma}N_{u}\exp\Big(-\frac{1}{B}\cdot\frac{nt^{2}}{t/3A+1}\Big)\leq 2N_{\gamma}N_{u}\exp\Big(-\frac{3Ant}{2B}\Big),
\end{align*}
where the last inequality holds for $t\geq 3A$. Hence for $b\geq 3A$, it follows that
\begin{align*}
\mathbb{E}_{S,S^{\prime}}[T^{2}]&=\int_{0}^{\infty}\mathbb{P}\{T^{2}\geq u\}du=\int_{0}^{\infty}\mathbb{P}\{T^{2}\geq\sqrt{u}\}du \\
&\leq b^{2}+2N_{\gamma}N_{u}\int_{b^{2}}^{\infty}\exp\Big(-\frac{3An\sqrt{u}}{2B}\Big)du \\
&\leq b^{2}+4N_{\gamma}N_{u}\Big(\frac{2B}{3An}\Big)^{2}\Big(1+\frac{3Anb}{2B}\Big)\exp\Big(-\frac{3Anb}{2B}\Big).
\end{align*}
Assume $5\leq\log N_{\gamma}N_{u}\leq n$, setting $b=\frac{2B}{3An}\log(N_{\gamma}N_{u})$ gives
\begin{equation}\label{eq:ET2}
\mathbb{E}_{S,S^{\prime}}[T^{2}]\leq\Big(\frac{2B}{3An}\Big)^{2}\big(\log^{2}(N_{\gamma}N_{u})+\log(N_{\gamma}N_{u})+4\big)\leq2\Big(\frac{2B}{3A}\Big)^{2}\Big(\frac{\log(N_{\gamma}N_{u})}{n}\Big)^{2}.
\end{equation}
\par\noindent\emph{Step 4.} Combining (\ref{eq:R:r2:T2})(\ref{eq:Er2})(\ref{eq:ET2}) yields
\begin{equation*}
\begin{aligned}
&|R(\tilde{\gamma}_{n},\tilde{u}_{n})-R_{n}(\tilde{\gamma}_{n},\tilde{u}_{n})| \\
&\leq\frac{1}{2}R(\tilde{\gamma}_{n},\tilde{u}_{n})+\frac{1}{2}A^{2}+\Big(\frac{2B}{3A}\Big)^{2}\Big(\frac{\log(N_{\gamma}N_{u})}{n}\Big)^{2}+35B_{\gamma}B_{u}(B_{\gamma}+B_{u})\varepsilon.
\end{aligned}
\end{equation*}
Setting $A=\sqrt{\frac{2B\log(N_{\gamma}N_{u})}{9n}}$ gives the result.
\end{proof}
In order to estimate the expected empirical risk, we first introduce the following lemma.
\begin{lemma}\label{lemma:sumgaussian}
Let $\eta_{j}\sim\subG(0,\delta^{2})$ for $j=1,\ldots,N$, then
\begin{equation*}
\mathbb{E}\Big[\max_{1\leq j\leq N}\eta_{j}^{2}\Big]\leq 4\varepsilon^{2}(1+\log N).
\end{equation*}
\begin{proof}
For any $0<t<(2\delta^{2})^{-1}$, we can derive
\begin{align*}
\exp\Big(t\mathbb{E}\Big[\max_{1\leq j\leq N}\eta_{j}^{2}\Big]\Big)&\leq\mathbb{E}\Big[\max_{1\leq j\leq N}\exp(t\eta_{j}^{2})\Big]\leq N\mathbb{E}[\exp(t\eta_{1}^{2})] \\
&\leq\frac{N}{\sqrt{2\pi\varepsilon^{2}}}\int_{\mathbb{R}}\exp(tx^{2})\exp\Big(-\frac{x^{2}}{2\delta^{2}}\Big)dx \\
&=\frac{N}{\sqrt{1-2\delta^{2}t}}.
\end{align*}
Therefore
\begin{equation*}
\mathbb{E}\Big[\max_{1\leq j\leq N}\eta_{j}^{2}\Big]\leq\frac{1}{t}\log\frac{N}{\sqrt{1-2\delta^{2}t}}.
\end{equation*}
Setting $t=(4\delta^{2})^{-1}$ completes the proof.
\end{proof}
\end{lemma}
\begin{lemma}\label{lemma:errdec:2}
Suppose \cref{ass:gammau:bound,ass:fg:bound} hold. 
\begin{equation*}
R_{n}(\tilde{\gamma}_{n},\tilde{u}_{n})\lesssim\inf_{(\gamma,u)\in\calF_{\gamma}\times\calF_{u}}R(\gamma,u)+\Delta_{n}(\tilde{\gamma}_{n},\tilde{u}_{n})+\frac{\delta^{2}\log(N_{\gamma}N_{u})}{n}+(\delta+B_{\gamma}B_{u})(B_{\gamma}+B_{u})\varepsilon.
\end{equation*}
\end{lemma}
\begin{proof}
\par We divide the proof into three steps.
\par\noindent\emph{Step 1.} For any fixed $(\gamma,u)\in\calF_{\gamma}\times\calF_{u}$, it is easy to show that
\begin{equation}\label{eq:deltan}
\mathbb{E}_{S}[L_{n}(\tilde{\gamma}_{n},\tilde{u}_{n})]\leq\mathbb{E}_{S}[L_{n}(\gamma,u)]+\Delta_{n}(\tilde{\gamma}_{n},\tilde{u}_{n}).
\end{equation}
By definition, we have
\begin{equation}\label{eq:Eln:tilde}
\begin{aligned}
&\mathbb{E}_{S}[L_{n}(\tilde{\gamma}_{n},\tilde{u}_{n})] =\mathbb{E}_{S}\Big[\frac{1}{n}\sum_{i=1}^{n}(Y_{i}-\tilde{\gamma}_{n}|\nabla\tilde{u}_{n}|(X_{i}))^{2}\Big]+\mathbb{E}_{S}\Big[\frac{1}{n}\sum_{i=1}^{n}g(\tilde{\gamma}_{n},\tilde{u}_{n},X_{i},\bar{X}_{i})\Big] \\
&=\mathbb{E}_{S}\Big[\frac{1}{n}\sum_{i=1}^{n}(\gamma^{\dagger}|\nabla u^{\dagger}|(X_{i})-\tilde{\gamma}_{n}|\nabla\tilde{u}_{n}|(X_{i})+\xi_{i})^{2}\Big]+\mathbb{E}_{S}\Big[\frac{1}{n}\sum_{i=1}^{n}g(\tilde{\gamma}_{n},\tilde{u}_{n},X_{i},\bar{X}_{i})\Big] \\
&=R_{n}(\tilde{\gamma}_{n},\tilde{u}_{n})+\delta^{2}-\mathbb{E}_{S}\Big[\frac{2}{n}\sum_{i=1}^{n}\xi_{i}\tilde{\gamma}_{n}|\nabla\tilde{u}_{n}|(X_{i})\Big],
\end{aligned}
\end{equation}
and
\begin{equation}\label{eq:Eln}
\begin{aligned}
&\mathbb{E}_{S}[L_{n}(\gamma,u)]=\mathbb{E}_{S}\Big[\frac{1}{n}\sum_{i=1}^{n}(Y_{i}-\gamma|\nabla u|(X_{i}))^{2}\Big]+\mathbb{E}_{S}\Big[\frac{1}{n}\sum_{i=1}^{n}g(\gamma,u,X_{i},\bar{X}_{i})\Big] \\
&=\mathbb{E}_{S}\Big[\frac{1}{n}\sum_{i=1}^{n}(\gamma^{\dagger}|\nabla u^{\dagger}|(X_{i})-\gamma|\nabla u|(X_{i})+\xi_{i})^{2}\Big]+\mathbb{E}_{S}\Big[\frac{1}{n}\sum_{i=1}^{n}g(\gamma,u,X_{i},\bar{X}_{i})\Big] \\
&=\|\gamma^{\dagger}|\nabla u^{\dagger}|-\gamma|\nabla u|\|_{L^{2}(U)}+\delta^{2}+G(\gamma,u).
\end{aligned}
\end{equation}
Combining (\ref{eq:deltan})(\ref{eq:Eln:tilde})(\ref{eq:Eln}) gives
\begin{equation}\label{eq:proof:Rn:step1}
\begin{aligned}
&R_{n}(\tilde{\gamma}_{n},\tilde{u}_{n}) \\
&\leq\|\gamma^{\dagger}|\nabla u^{\dagger}|-\gamma|\nabla u|\|_{L^{2}(U)}^{2}+G(\gamma,u)+\Big|\mathbb{E}_{S}\Big[\frac{2}{n}\sum_{i=1}^{n}\xi_{i}\tilde{\gamma}_{n}|\nabla\tilde{u}_{n}|(X_{i})\Big]\Big|+\Delta_{n}(\tilde{\gamma}_{n},\tilde{u}_{n}).
\end{aligned}
\end{equation}
\par\noindent\emph{Step 2.} Now we show
\begin{equation}\label{eq:proof:Rn:step2}
\begin{aligned}
&\Big|\mathbb{E}_{S}\Big[\frac{2}{n}\sum_{i=1}^{n}\xi_{i}\tilde{\gamma}_{n}|\nabla\tilde{u}_{n}|(X_{i})\Big]\Big| \\
&\leq\frac{1}{2}R_{n}(\tilde{\gamma}_{n},\tilde{u}_{n})+\frac{8\delta^{2}(1+\log(N_{\gamma}N_{u}))}{n}+2(\delta+B_{\gamma}B_{u})(B_{\gamma}+B_{u})\varepsilon.
\end{aligned}
\end{equation}
In fact, for estimators $(\tilde{\gamma}_{n},\tilde{u}_{n})\in\calF_{\gamma}\times\calF_{u}$, there exists $k^{\prime}$ and $\ell^{\prime}$, such that $\|\gamma_{k^{\prime}}-\tilde{\gamma}_{n}\|_{W^{1,\infty}(U)}\leq\varepsilon$ and $\|u_{\ell^{\prime}}-\tilde{u}_{n}\|_{W^{2,\infty}(U)}\leq\varepsilon$. Then we have
\begin{align}
&\Big|\mathbb{E}_{S}\Big[\frac{2}{n}\sum_{i=1}^{n}\xi_{i}\tilde{\gamma}_{n}|\nabla\tilde{u}_{n}|(X_{i})\Big]\Big|\nonumber \\
&\leq\Big|\mathbb{E}_{S}\Big[\frac{2}{n}\sum_{i=1}^{n}\xi_{i}(\tilde{\gamma}_{n}|\nabla\tilde{u}_{n}|(X_{i})-\gamma_{k^{\prime}}|\nabla u_{\ell^{\prime}}|(X_{i}))\Big]\Big|\nonumber \\
&\quad+\Big|\mathbb{E}_{S}\Big[\frac{2}{n}\sum_{i=1}^{n}\xi_{i}(\gamma_{k^{\prime}}|\nabla u_{\ell^{\prime}}|(X_{i})-\gamma^{\dagger}|\nabla u^{\dagger}|(X_{i}))\Big]\Big|\nonumber \\
&\leq\frac{2(B_{\gamma}+B_{u})\delta}{n}\mathbb{E}_{S}\Big[\sum_{i=1}^{n}|\xi_{i}|\Big]+\Big|\mathbb{E}_{S}\Big[\frac{2}{n}\sum_{i=1}^{n}\xi_{i}(\gamma_{k^{\prime}}|\nabla u_{\ell^{\prime}}|(X_{i})-\gamma^{\dagger}|\nabla u^{\dagger}|(X_{i}))\Big]\Big|\nonumber \\
&\leq2\sqrt{\frac{2}{\pi}}(B_{\gamma}+B_{u})\delta\varepsilon+\Big|\mathbb{E}_{S}\Big[\frac{2}{n}\sum_{i=1}^{n}\xi_{i}(\gamma_{k^{\prime}}|\nabla u_{\ell^{\prime}}|(X_{i})-\gamma^{\dagger}|\nabla u^{\dagger}|(X_{i}))\Big]\Big|. \label{eq:xi:gammau}
\end{align}
Let $\eta_{k,\ell}$ be random variables defined as
\begin{equation}
\eta_{k,\ell}=\frac{\sum_{i=1}^{n}\xi_{i}(\gamma_{k}|\nabla u_{\ell}|(X_{i})-\gamma^{\dagger}|\nabla u^{\dagger}|(X_{i}))}{(\sum_{i=1}^{n}(\gamma_{k}|\nabla u_{\ell}|(X_{i})-\gamma^{\dagger}|\nabla u^{\dagger}|(X_{i}))^{2})^{1/2}},
\end{equation}
then each $\eta_{k,\ell}$ follows sub-Gaussian distribution $\subG(\delta^{2})$ conditionally on $S=\{(X_{i},Y_{i})\}_{i=1}^{n}\cup\{\bar{X}_{i}\}_{i=1}^{n}$.
By using Cauchy-Schwarz and AM-GM inequality, we find
\begin{align}
&\Big|\mathbb{E}_{S}\Big[\frac{2}{n}\sum_{i=1}^{n}\xi_{i}(\gamma_{k^{\prime}}|\nabla u_{\ell^{\prime}}|(X_{i})-\gamma^{\dagger}|\nabla u^{\dagger}|(X_{i}))\Big]\Big|\nonumber \\
&\leq\frac{2}{n}\Big|\mathbb{E}_{S}\Big[\Big(\sum_{i=1}^{n}(\gamma_{k^{\prime}}|\nabla u_{\ell^{\prime}}|(X_{i})-\gamma^{\dagger}|\nabla u^{\dagger}|(X_{i}))^{2}\Big)^{1/2}\eta_{k^{\prime},\ell^{\prime}}\Big]\Big| \nonumber \\
&\leq\frac{2}{n}\mathbb{E}_{S}\Big[\sum_{i=1}^{n}(\gamma_{k^{\prime}}|\nabla u_{\ell^{\prime}}|(X_{i})-\gamma^{\dagger}|\nabla u^{\dagger}|(X_{i}))^{2}\Big]^{1/2}\mathbb{E}_{S}\Big[\eta_{k^{\prime},\ell^{\prime}}^{2}\Big]^{1/2} \nonumber \\
&\leq\frac{1}{2}\mathbb{E}_{S}\Big[\frac{1}{n}\sum_{i=1}^{n}(\gamma_{k^{\prime}}|\nabla u_{\ell^{\prime}}|(X_{i})-\gamma^{\dagger}|\nabla u^{\dagger}|(X_{i}))^{2}\Big]+\frac{2}{n}\mathbb{E}_{S}\Big[\eta_{k^{\prime},\ell^{\prime}}^{2}\Big] \nonumber\\
&\leq\frac{1}{2}\mathbb{E}_{S}\Big[\frac{1}{n}\sum_{i=1}^{n}(\tilde{\gamma}_{n}|\nabla \tilde{u}_{n}|(X_{i})-\gamma^{\dagger}|\nabla u^{\dagger}|(X_{i}))^{2}\Big]+2B_{\gamma}B_{u}(B_{\gamma}+B_{u})\delta+\frac{2}{n}\mathbb{E}_{S}\Big[\max_{k,\ell}\eta_{k,\ell}^{2}\Big] \nonumber\\
&\leq\frac{1}{2}R_{n}(\tilde{\gamma}_{n},\tilde{u}_{n})+\frac{8\varepsilon^{2}(1+\log N_{\gamma}N_{u})}{n}+2B_{\gamma}B_{u}(B_{\gamma}+B_{u})\delta, \label{eq:xisub:Rn}
\end{align}
where the last inequality is owing to $G_{n}(\tilde{\gamma}_{n},\tilde{u}_{n})\geq 0$ 
and \cref{lemma:sumgaussian}. Combining (\ref{eq:xi:gammau}) and (\ref{eq:xisub:Rn}), we obtain (\ref{eq:proof:Rn:step2}).
\par\noindent\emph{Step 3.} Using (\ref{eq:proof:Rn:step1}) and (\ref{eq:proof:Rn:step2}), we have
\begin{equation}
\begin{aligned}
R_{n}(\tilde{\gamma}_{n},\tilde{u}_{n})&\leq 2(\|\gamma|\nabla u|-\gamma^{\dagger}|\nabla u^{\dagger}|\|_{L^{2}(U)}^{2}+G(\gamma,u))+\frac{16\varepsilon^{2}(1+\log(N_{\gamma}N_{u}))}{n} \\
&\quad+4\sqrt{\frac{2}{\pi}}(B_{\gamma}+B_{u})\varepsilon\delta+4B_{\gamma}B_{u}(B_{\gamma}+B_{u})\varepsilon+\Delta_{n}(\tilde{\gamma}_{n},\tilde{u}_{n}),
\end{aligned}
\end{equation}
which completes the proof.
\end{proof}
\begin{proof}[Proof of \cref{lemma:err:dec}]
Combining \cref{lemma:errdec:1,lemma:errdec:2} yields the conclusion.
\end{proof}

\subsection{The proof of the approximation error bound}\label{sec:appendix:proof:approx}
\begin{lemma}[Lipschitz continuity of forward operator]\label{lemma:forward:lip}
Suppose $f\in H^{3/2}(\partial U)$. Given a function $\gamma$, we denote $u(\gamma)$ the weak solution of 
\begin{equation*}
\nabla\cdot(\gamma\nabla u)=0,~\text{in}~U,\quad Tu=f,~\text{on}~\partial U.
\end{equation*}
Then the mapping $\gamma\mapsto u(\gamma)$ is Lipschitz continuous from $\calK\cap C^{1}(\bU)$ to $H^{2}(U)$, i.e., for any $\gamma,\gamma+\delta\gamma\in\calK\cap C^{1}(\bU)$, there holds
\begin{equation*}
\|u(\gamma+\delta\gamma)-u(\gamma)\|_{H^{2}(U)}\leq C_{\lip}B_{f}\|\delta\gamma\|_{C^{1}(\bU)},
\end{equation*}
where $C_{\lip}$ is a constant depending on $\gamma_{0}$, $\gamma_{1}$ and $U$.
\end{lemma}
\begin{proof}
Let $\delta u=u(\gamma+\delta\gamma)-u(\gamma)$. Then $\delta u$ is the unique solution of  
\begin{equation*}
-\nabla\cdot(\gamma\nabla\delta u)=\nabla\cdot(\delta\gamma\nabla u(\gamma+\delta\gamma)),~\text{in}~U,\quad T\delta u=0,~\text{on}~\partial U.
\end{equation*}
It is easy to verify that 
\begin{align*}
\|\delta u\|_{H^{2}(U)}&\leq c_{1}\|\nabla\cdot(\delta\gamma\nabla u(\gamma+\delta\gamma))\|_{L^{2}(U)} \\
&\leq c_{1}\|u(\gamma+\delta\gamma))\|_{H^{2}(U)}\|\delta\gamma\|_{C^{1}(\bU)} \\
&\leq c_{1}c_{2}\|f\|_{H^{3/2}(\partial U)}\|\delta\gamma\|_{C^{1}(\bU)},
\end{align*}
where $c_{1}$, $c_{2}$ are constants depending on $\gamma_{0}$, $\gamma_{1}$ and $U$. 
\end{proof}

\begin{proof}[Proof of \cref{lemma:app:1}]
For each fixed $(\gamma,u)\in\calF_{\gamma}\times\calF_{u}$, recall the definition of the excess risk
\begin{equation}\label{eq:excess}
R(\gamma,u)=\|\gamma|\nabla u|-\gamma^{\dagger}|\nabla u^{\dagger}|\|_{L^{2}(U)}^{2}+G(\gamma,u).
\end{equation}
We first consider the first term in \cref{eq:excess}. By triangular inequality and AM-GM inequality, we have 
\begin{align}
\|\gamma|\nabla u|-\gamma^{\dagger}|\nabla u^{\dagger}|\|_{L^{2}(U)}^{2}
&\leq 2\|\gamma|\nabla u|-\gamma^{\dagger}|\nabla u|\|_{L^{2}(U)}^{2}+2\|\gamma^{\dagger}|\nabla u|-\gamma^{\dagger}|\nabla u^{\dagger}|\|_{L^{2}(U)}^{2} \nonumber \\
&\leq2B_{u}^{2}\|\gamma-\gamma^{\dagger}\|_{L^{2}(U)}^{2}+2B_{\gamma}^{2}\|u-u^{\dagger}\|_{H^{1}(U)}^{2}.\label{eq:proof:app:1}
\end{align}
We next investigate the second term in \cref{eq:excess}. Denote by $u_{\gamma}$ the solution of 
\begin{equation*}
\nabla\cdot(\gamma\nabla u)=0,~\text{in}~U,\quad Tu=f,~\text{on}~\partial U.
\end{equation*}
Then it follows that
\begin{align*}
G(\gamma,u)&=\|\nabla\cdot(\gamma\nabla(u-u_{\gamma}))\|_{L^{2}(U)}^{2}+\|Tu-Tu_{\gamma}\|_{L^{2}(\partial U)}^{2} \\
&\leq B_{\gamma}^{2}\|u-u_{\gamma}\|_{H^{2}(U)}^{2}+C_{\tr}^{2}\|u-u_{\gamma}\|_{H^{1}(U)}^{2}\leq(B_{\gamma}^{2}+C_{\tr}^{2})\|u-u_{\gamma}\|_{H^{2}(U)}^{2},
\end{align*}
where the first inequality is due to the trace theorem, and the constant $C_{\tr}$ only depends on $U$. By \cref{lemma:forward:lip}, it holds
\begin{equation*}
\|u_{\gamma}-u^{\dagger}\|_{H^{2}(U)}\leq C_{\lip}B_{f}\|\gamma-\gamma^{\dagger}\|_{C^{1}(\bU)},
\end{equation*}
and consequently,
\begin{equation}\label{eq:proof:app:2}
G(\gamma,u)\leq2(B_{\gamma}^{2}+C_{\tr}^{2})C_{\lip}^{2}B_{f}^{2}\|\gamma-\gamma^{\dagger}\|_{C^{1}(\bU)}^{2}+2(B_{\gamma}^{2}+C_{\tr}^{2})\|u-u^{\dagger}\|_{H^{2}(U)}^{2}.
\end{equation}
Combining \cref{eq:proof:app:1,eq:proof:app:2} completes the proof.
\end{proof}
\begin{lemma}\label{lemma:app:nn}
Let $s\in\mathbb{N}_{+}$ and $\mu>0$. Suppose $\varrho$ is an exponential PU-admissible activation function.
\begin{enumerate}[(i)]
\item For each $\gamma^{\dagger}\in C^{s+1}(\bU)$, there exists a $\varrho$-network $\gamma\in\calN_{\varrho}(\calD_{\gamma},\calS_{\gamma},\calB_{\gamma})$ such that 
\begin{equation*}
\|\gamma-\gamma^{\dagger}\|_{C^{1}(\bU)}\leq\|\gamma^{\dagger}\|_{C^{s+1}(\bU)}\cdot\calS_{\gamma}^{-\frac{s-\mu}{d}},\quad\text{with}~\calD_{\gamma}=C\log(d+s)~\text{and}~\calB_{\gamma}=C\calS_{\gamma}^{\frac{2s}{d}+7}.
\end{equation*}
\item For any $u^{\dagger}\in H^{s+2}(U)$, there exists a $\varrho$-network $u\in\calN_{\varrho}(\calD_{u},\calS_{u},\calB_{u})$ such that 
\begin{equation*}
\|u-u^{*}\|_{H^{2}(U)}\leq\|u^{\dagger}\|_{H^{s+2}(U)}\cdot\calS_{u}^{-\frac{s+1-\mu}{d}},\quad\text{with}~\calD_{u}=C\log(d+s+1)~\text{and}~\calB_{u}=C\calS_{u}^{\frac{2s+2}{d}+7}.
\end{equation*}
\end{enumerate}
Here $C$ is a constant depending on $d$, $s$, $\mu$ and $U$.
\end{lemma}
\begin{proof}
A direct conclusion of Proposition 4.8 in \cite{Guhring2021Approximation}.
\end{proof}

\begin{proof}[Proof of \cref{lemma:approx}]
Combining \cref{lemma:app:1,lemma:app:nn} yields the desired result.
\end{proof}

\subsection{The proof of the statistical error bound}\label{sec:appendix:proof:staerr}
\begin{lemma}\label{lemma:bound:nn}
Let $m\in\bbN$ and $U\subseteq[0,1]^{d}$ be a domain.
For each neural network $\phi\in\calN_{\varrho}(\calW,\calS,\calB)$, we have $\phi\in C^{m}(\bU)$ and 
\begin{align*}
\|\phi\|_{C^{0}(\bU)}&=\calO(N_{\max}\calB),\quad\|\phi\|_{C^{1}(\bU)}=\calO(N_{\max}^{\calD}\calB^{\calD}), \\
\|\phi\|_{C^{2}(\bU)}&=\calO(N_{\max}^{\calD(2\calD+1)}\calB^{\calD(\calD+1)}), \\
\|\phi\|_{C^{3}(\bU)}&=\calO(N_{\max}^{\calD(2\calD^{2}+2\calD+1)}\calB^{\calD(3\calD^{2}+3\calD+1)}), 
\end{align*}
where $N_{\max}=\max\{N_{\ell}:\ell=0,\ldots,\calD\}$.
\end{lemma}
\begin{proof}
Recall the definition of the neural network $\phi$:
\begin{align*}
\phi_{0}(x)&=x, \\
\phi_{\ell}(x)&=\varrho(T_{\calD}(\phi_{\ell-1}(x))),\quad \ell=1,\ldots,\calD-1, \\
\phi(x)&=T_{\calD}(\phi_{\calD-1}(x)). 
\end{align*}
Let $\calB_{\varrho,i}$ ($i=0,1,2,3$) be the bounds of the activation function $\varrho$ and its derivatives, i.e., $|\varrho(x)|\leq\calB_{\varrho,0}$, $|\varrho^{\prime}(x)|\leq\calB_{\varrho,1}$, $|\varrho^{\prime\prime}(x)|\leq\calB_{\varrho,2}$ and $|\varrho^{\prime\prime\prime}(x)|\leq\calB_{\varrho,3}$, for each $x\in\bU$. Then it holds 
\begin{equation*}
|\phi(x)|\leq N_{\calD}\calB\calB_{\varrho,0}.
\end{equation*}
Let $\phi_{\ell}(x)=(\phi_{\ell}^{1}(x),\ldots,\phi_{\ell}^{N_{\ell+1}}(x))$, then we deduce
\begin{equation*}
\partial_{x_{q}}\phi_{\ell}^{m}(x)=\sum_{i=1}^{N_{\ell}}A_{\ell}^{m,i}\varrho^{\prime}(\phi_{\ell-1}^{i}(x))\partial_{x_{q}}\phi_{\ell-1}^{i}(x),\quad\text{for}~1\leq q\leq d,
\end{equation*}
and consequently, 
\begin{equation*}
|\partial_{x_{q}}\phi_{\ell}^{m}(x)|\leq N_{\ell}\calB\calB_{\varrho,1}\max_{1\leq i\leq N_{\ell}}|\partial_{x_{q}}\phi_{\ell-1}^{i}(x)|,\quad\text{for}~~1\leq q\leq d.
\end{equation*}
Denote $B_{\ell,1}=\max_{m}\sup_{x\in\bU}|\partial_{x_{q}}\phi_{\ell}^{m}(x)|$, we have
\begin{equation}\label{eq:proof:nnbound:1}
B_{\ell,1}\leq\Big(\prod_{i=1}^{\ell}N_{i}\Big)\calB^{\ell}\calB_{\varrho,1}^{\ell}\leq N_{\max}^{\ell}\calB^{\ell}\calB_{\varrho,1}^{\ell}\quad\text{and}\quad
\sup_{x\in\bU}|\partial_{x_{p}}\phi(x)|\leq N_{\max}^{\calD}\calB^{\calD}\calB_{\varrho,1}^{\calD}.
\end{equation}
We next consider the bound of $\partial_{x_{q}}\partial_{x_{p}}\phi_{\ell}^{m}(x)$. It is easy to show that
\begin{align*}
\partial_{x_{q}}\partial_{x_{p}}\phi_{\ell}^{m}(x)
&=\partial_{x_{q}}\Big(\sum_{i=1}^{N_{\ell}}A_{\ell}^{m,i}\varrho^{\prime}(\phi_{\ell-1}^{i}(x))\partial_{x_{p}}\phi_{\ell-1}^{i}(x)\Big) \\
&=\sum_{i=1}^{N_{\ell}}A_{\ell}^{m,i}\partial_{x_{q}}\Big\{\varrho^{\prime}(\phi_{\ell-1}^{i}(x))\Big\}\partial_{x_{p}}\phi_{\ell-1}^{i}(x)+\sum_{i=1}^{N_{\ell}}A_{\ell}^{m,i}\varrho^{\prime}(\phi_{\ell-1}^{i}(x))\partial_{x_{q}}\partial_{x_{p}}\phi_{\ell-1}^{i}(x) \\
&=\sum_{i=1}^{N_{\ell}}A_{\ell}^{m,i}\varrho^{\prime\prime}(\phi_{\ell-1}^{i}(x))\partial_{x_{q}}\phi_{\ell-1}^{i}(x)\partial_{x_{p}}\phi_{\ell-1}^{i}(x)+\sum_{i=1}^{N_{\ell}}A_{\ell}^{m,i}\varrho^{\prime}(\phi_{\ell-1}^{i}(x))\partial_{x_{q}}\partial_{x_{p}}\phi_{\ell-1}^{i}(x).
\end{align*}
Thus for $1\leq p,q\leq d$, we have
\begin{equation*}
|\partial_{x_{q}}\partial_{x_{p}}\phi_{\ell}^{m}(x)|\leq N_{\ell}\calB\calB_{\varrho,2}B_{\ell-1,1}^{2}+N_{\ell}\calB\calB_{\varrho,1}\max_{1\leq i\leq N_{\ell}}|\partial_{x_{q}}\partial_{x_{p}}\phi_{\ell-1}^{i}(x)|.
\end{equation*}
Then using \cref{eq:proof:nnbound:1} gives
\begin{equation}\label{eq:proof:nnbound:2}
B_{\ell,2}\lesssim N_{\max}^{\ell}\Big(\prod_{i=1}^{\ell}B_{i,1}^{2}\Big)\calB^{\ell}(\calB_{\varrho,1}\calB_{\varrho,2})^{\ell}\lesssim N_{\max}^{\ell(2\ell+1)}\calB^{\ell(\ell+1)}\calB_{\varrho,1}^{\ell(\ell+1)}\calB_{\varrho,2}^{\ell},
\end{equation}
and
\begin{equation*}
|\partial_{x_{q}}\partial_{x_{p}}\phi(x)|\lesssim N_{\max}^{\calD(2\calD+1)}\calB^{\calD(\calD+1)}\calB_{\varrho,1}^{\calD(\calD+1)}\calB_{\varrho,2}^{\calD}.
\end{equation*}
Similarly, we now bound $\partial_{x_{r}}\partial_{x_{q}}\partial_{x_{p}}\phi_{\ell}^{m}(x)$. It follows that
\begin{align*}
&\partial_{x_{r}}\partial_{x_{q}}\partial_{x_{p}}\phi_{\ell}^{m}(x) \\
&=\sum_{i=1}^{N_{\ell}}A_{\ell}^{m,i}\partial_{x_{r}}\Big\{\varrho^{\prime\prime}(\phi_{\ell-1}^{i}(x))\partial_{x_{q}}\phi_{\ell-1}^{i}(x)\partial_{x_{p}}\phi_{\ell-1}^{i}(x)\Big\} \\
&\quad+\sum_{i=1}^{N_{\ell}}A_{\ell}^{m,i}\partial_{x_{r}}\Big\{\varrho^{\prime}(\phi_{\ell-1}^{i}(x))\partial_{x_{q}}\partial_{x_{p}}\phi_{\ell-1}^{i}(x)\Big\} \\
&=\sum_{i=1}^{N_{\ell}}A_{\ell}^{m,i}\varrho^{\prime\prime\prime}(\phi_{\ell-1}^{i}(x))\partial_{x_{r}}\phi_{\ell-1}^{i}(x)\partial_{x_{q}}\phi_{\ell-1}^{i}(x)\partial_{x_{p}}\phi_{\ell-1}^{i}(x) \\
&\quad+\sum_{i=1}^{N_{\ell}}A_{\ell}^{m,i}\varrho^{\prime\prime}(\phi_{\ell-1}^{i}(x))\Big\{\partial_{x_{r}}\partial_{x_{q}}\phi_{\ell-1}^{i}(x)\partial_{x_{p}}\phi_{\ell-1}^{i}(x)+\partial_{x_{q}}\phi_{\ell-1}^{i}(x)\partial_{x_{r}}\partial_{x_{p}}\phi_{\ell-1}^{i}(x)\Big\} \\
&\quad+\sum_{i=1}^{N_{\ell}}A_{\ell}^{m,i}\varrho^{\prime\prime}(\phi_{\ell-1}^{i}(x))\partial_{x_{r}}\phi_{\ell-1}^{i}(x)\partial_{x_{q}}\partial_{x_{p}}\phi_{\ell-1}^{i}(x) \\
&\quad+\sum_{i=1}^{N_{\ell}}A_{\ell}^{m,i}\varrho^{\prime}(\phi_{\ell-1}^{i}(x))\partial_{x_{r}}\partial_{x_{q}}\partial_{x_{p}}\phi_{\ell-1}^{i}(x).
\end{align*}
Hence
\begin{align*}
&|\partial_{x_{r}}\partial_{x_{q}}\partial_{x_{p}}\phi_{\ell}^{m}(x)| \\
&\leq N_{\ell}\calB\calB_{\varrho,3}B_{\ell-1,1}^{3}+3N_{\ell}\calB\calB_{\varrho,2}B_{\ell-1,1}B_{\ell-1,2}+N_{\ell}\calB\calB_{\varrho,1}\max_{1\leq i\leq N_{\ell}}|\partial_{x_{r}}\partial_{x_{q}}\partial_{x_{p}}\phi_{\ell-1}^{i}(x)|.
\end{align*}
Then we have for $1\leq p,q,r\leq d$
\begin{align*}
\sup_{x\in\bU}|\partial_{x_{r}}\partial_{x_{q}}\partial_{x_{p}}\phi(x)|
&\lesssim N_{\max}^{\calD}\Big(\prod_{\ell=1}^{\calD}B_{\ell,1}^{3}+\prod_{\ell=1}^{\calD}B_{\ell,1}B_{\ell,2}\Big)\calB^{\calD}(\calB_{\varrho,1}\calB_{\varrho,2}\calB_{\varrho,3})^{\calD} \\
&\lesssim N_{\max}^{\calD(2\calD^{2}+2\calD+1)}\calB^{\calD(3\calD^{2}+3\calD+1)}\calB_{\varrho,1}^{\calD(3\calD^{2}+3\calD+1)}\calB_{\varrho,2}^{\calD(\calD+1)}\calB_{\varrho,2}^{\calD}.
\end{align*}
This completes the proof.
\end{proof}

\begin{lemma}[Extension]\label{lemma:extension}
Let $m\in\bbN$ and $U\subset\subset(0,1)^{d}$ be a domain with $C^{\infty}$-boundary. Then for each $\psi\in C^{m}(\bU)$, there exists a bounded and compact support extension $E\psi\in C_{0}^{m}([0,1]^{d})$, such that
\begin{enumerate}[(i)]
\item $E\psi(x)=\psi(x)$ in $\bU$,
\item $\|E\psi\|_{C^{m}([0,1]^{d})}\leq C\|\psi\|_{C^{m}(\bU)}$,
\end{enumerate}
where $C$ is a constant depending only on $U$.
\end{lemma}
\begin{proof}
For each function $\psi\in C^{m}(\bU)$, by Theorem 5.24 \cite{adams2003sobolev}, there exists $E_{1}\psi\in C^{m}([0,1]^{d})$. Let $\zeta\in C_{0}^{m}([0,1]^{d})$ be a cut-off function such that
\begin{enumerate}[(i)]
\item $0\leq\zeta(x)\leq 1$ for each $x\in[0,1]^{d}$,
\item $\zeta(x)=1$ for each $x\in\bU$, and
\item $|\partial^{\alpha}\zeta|\leq C/(\dist(U,\partial(0,1)^{d}))^{|\alpha|}$.
\end{enumerate}
Then $E\psi:=\zeta E_{1}\psi\in C_{0}^{m}([0,1]^{d})$ is the desired function of $\psi$. The proof is completed.
\end{proof}

\begin{lemma}\label{lemma:metric:entropy}
Let $U\subset\subset(0,1)^{d}$ be a domain with $C^{\infty}$-boundary. Let $k,s\in\bbN$ with $k<s$, and $\calC_{B}^{s}$ be the norm-ball of radius $B$ in $C^{s}(\bU)$. Then 
\begin{equation*}
H(\varepsilon,\calC_{B}^{s},C^{k}(\bU)\leq C\cdot\Big(\frac{B}{\varepsilon}\Big)^{\frac{d}{s-k}},
\end{equation*}
where $C$ is a constant depending on $s$, $k$ and $U$.
\end{lemma}
\begin{proof}
Notice that
\begin{equation*}
\Big\{\partial^{\alpha}f:f\in\calC_{B}^{s},~|\alpha|=k\Big\}\subseteq\calC_{B}^{s-k}.
\end{equation*}
Let $\calC_{B,\varepsilon}^{s-k}$ be a $\|\cdot\|_{C(\bU)}$ $\delta$-cover of $\calC_{B}^{s-k}$ with $|\calC_{B,\varepsilon}^{s-k}|=N(\delta,\calC_{B}^{s-k},C(\bU))$.
Then for each $f\in\calC_{B}^{s}$ and $\alpha=(\alpha_{1},\ldots,\alpha_{d})$ with $|\alpha|=k$, there exist $\pi^{\alpha}(f)\in\calC_{B,\varepsilon}^{s-k}$ such that
\begin{equation}\label{emtropy:proof:1}
\|\pi^{\alpha}(f)-\partial^{\alpha}f\|_{C(\bU)}\leq\delta.
\end{equation}
Without loss of generality, we assume $\alpha_{j}\geq 1$, and define
\begin{equation*}
F(x_{1},\ldots,x_{d})=\int_{x_{j,0}}^{x_{j}}\pi^{\alpha}(f)(x_{1},\ldots,x_{j-1},z,x_{j+1},\ldots,x_{d})dz,
\end{equation*}
where $x_{j,0}=\min\{x_{j}:(x_{1},\ldots,x_{j},\ldots,x_{d})\in\bU\}$.
It is clear that $\partial_{j}F=\pi^{\alpha}(f)$. Denote $\bar{\alpha}=(\alpha_{1},\ldots,\alpha_{j}-1,\ldots,\alpha_{d})$. By \cref{lemma:extension}, 
there exist an extension $E(F-\partial^{\bar{\alpha}}f)\in C_{0}^{1}([0,1]^{d})$, such that $E(F-\partial^{\bar{\alpha}}f)(x)=(F-\partial^{\bar{\alpha}}f)(x)$ for each $x\in\bU$, and consequently
\begin{equation}\label{emtropy:proof:2}
\|F-\partial^{\bar{\alpha}}f\|_{C(\bU)}\leq\|E(F-\partial^{\bar{\alpha}}f)\|_{C([0,1]^{d})}.
\end{equation}
Using Poincar{\'e}'s inequality, we have
\begin{equation}\label{emtropy:proof:3}
\|E(F-\partial^{\bar{\alpha}}f)\|_{C([0,1]^{d})}\leq C\|\partial_{j}E(F-\partial^{\bar{\alpha}}f)\|_{C([0,1]^{d})}=C\|\widetilde{E}(\pi^{\alpha}(f)-\partial^{\alpha}f)\|_{C([0,1]^{d})},
\end{equation}
where $\widetilde{E}$ is an extension operator, satisfying $\widetilde{E}(\pi^{\alpha}(f)-\partial^{\alpha}f)\in C_{0}([0,1]^{d})$ and
\begin{equation}\label{emtropy:proof:4}
\|\widetilde{E}(\pi^{\alpha}(f)-\partial^{\alpha}f)\|_{C([0,1]^{d})}\leq C\|\pi^{\alpha}(f)-\partial^{\alpha}f\|_{C(\bU)},
\end{equation}
where $C$ is a constant depending on $U$. Here the existence of $\widetilde{E}$ can be guaranteed by \cref{lemma:extension}. Combining \cref{emtropy:proof:1,emtropy:proof:2,emtropy:proof:3,emtropy:proof:4} yields
\begin{equation*}
\max\Big\{\|F-\partial^{\bar{\alpha}}f\|_{C(\bU)},\|\partial_{j}(F-\partial^{\bar{\alpha}}f)\|_{C(\bU)}\Big\}\leq C\|\pi^{\alpha}(f)-\partial^{\alpha}f\|_{C(\bU)}\leq C\delta,
\end{equation*}
where $C$ is a constant depending on $U$.
\par Repeating the same procedure as above, we can construct a function $\pi(f)\in C^{k}(\bU)$, such that
\begin{equation*}
\max_{0\leq|\alpha|\leq k}\|\partial^{\alpha}(\pi(f)-f)\|_{C(\bU)}\leq C\delta,
\end{equation*}
which means $\calC_{B,\varepsilon}^{s}:=\{\pi(f):f\in\calC_{B}^{s}\}$ is a $\|\cdot\|_{C^{k}(\bU)}$ $C\delta$-cover of $\calC_{B}^{s}$ satisfying $|\calC_{B,\varepsilon}^{s}|=|\calC_{B,\varepsilon}^{s-k}|$. Then
\begin{equation*}
N(C\delta,\calC_{B}^{s},C^{k}(\bU))\leq N(\delta,\calC_{B}^{s-k},C(\bU)).
\end{equation*}
Setting $\varepsilon=C\delta$ and by Theorem 4.3.36 in \cite{Gine2015Mathematical}, we complete the proof.
\end{proof}

\begin{proof}[Proof of \cref{lemma:stat}]
It is sufficient to estimate the metric entropy of the following function classes:
\begin{enumerate}[(i)]
\item $\calF_{\gamma}=\calN_{\varrho}(\calD_{\gamma},\calS,\calB_{\gamma})$ with $\calD_{\gamma}=C\log(d+s)$ and $\calB_{\gamma}=C\calS^{\frac{2s}{d}+7}$, and
\item $\calF_{u}=\calN_{\varrho}(\calD_{u},\calS,\calB_{u})$ with $\calD_{u}=C\log(d+s+1)$ and $\calB_{u}=C\calS^{\frac{2s+2}{d}+7}$.
\end{enumerate}
By \cref{lemma:bound:nn}, we have
\begin{equation}\label{eq:staerr:proof:3}
\calF_{\gamma}\subseteq \calC_{B_{\gamma}}^{2},\quad\text{with}~B_{\gamma}=\calO(N_{\gamma,\max}^{\calD_{\gamma}(2\calD_{\gamma}+1)}\calB_{\gamma}^{\calD_{\gamma}(\calD_{\gamma}+1)}),
\end{equation}
and
\begin{equation}\label{eq:staerr:proof:4}
\calF_{u}\subseteq \calC_{B_{u}}^{3},\quad\text{with}~B_{u}=\calO(N_{u,\max}^{\calD_{u}(2\calD_{u}^{2}+2\calD_{u}+1)}\calB_{u}^{\calD_{u}(3\calD_{u}^{2}+3\calD_{u}+1)}).
\end{equation}
Consequently, we have
\begin{equation}\label{eq:staerr:proof:1}
H(\varepsilon,\calF_{\gamma},\|\cdot\|_{W^{1,\infty}(U)})\leq
H(\varepsilon,\calC_{B_{\gamma}}^{2},\|\cdot\|_{W^{1,\infty}(U)}),
\end{equation}
and
\begin{equation}\label{eq:staerr:proof:2}
H(\varepsilon,\calF_{u},\|\cdot\|_{W^{2,\infty}(U)})\leq
H(\varepsilon,\calC_{B_{u}}^{3},\|\cdot\|_{W^{2,\infty}(U)}).
\end{equation}
Applying \cref{lemma:metric:entropy,eq:staerr:proof:3,eq:staerr:proof:4} yields
\begin{align*}
H(\varepsilon,\calF_{\gamma},\|\cdot\|_{W^{1,\infty}(U)})&\leq
C\cdot\Big(\frac{N_{\gamma,\max}^{\calD_{\gamma}(2\calD_{\gamma}+1)}\calB_{\gamma}^{\calD_{\gamma}(\calD_{\gamma}+1)}}{\varepsilon}\Big)^{d},\quad\text{and} \\
H(\varepsilon,\calF_{u},\|\cdot\|_{W^{2,\infty}(U)})&\leq
C\cdot\Big(\frac{N_{u,\max}^{\calD_{u}(2\calD_{u}^{2}+2\calD_{u}+1)}\calB_{u}^{\calD_{u}(3\calD_{u}^{2}+3\calD_{u}+1)}}{\varepsilon}\Big)^{d}.
\end{align*}
Substituting $N_{\gamma,\max}\leq\calS$ and $N_{u,\max}\leq\calS$ gives
\begin{equation}\label{eq:staerr:proof:5}
\max\big\{H_{\gamma}^{\varepsilon},H_{u}^{\varepsilon}\big\}\lesssim\calO(\calS^{6(4d+s+1)\log^{3}(d+s+1)}\varepsilon^{-d}).
\end{equation}
Setting 
\begin{equation*}
\varepsilon=\Big(\frac{B_{\gamma}^{2}B_{u}^{2}+B_{f}^{2}+\delta^{2}}{(\delta+B_{\gamma}B_{u})(B_{\gamma}+B_{u})}\frac{\calS^{6(4d+s+1)\log^{3}(d+s+1)}}{n}\Big)^{\frac{1}{d+1}}
\end{equation*}
and letting $m=s+1$ yield the desired result.
\end{proof}

\bibliography{CDII}
\bibliographystyle{plainnat}
\end{document}

%% file: settings.tex
\usepackage[utf8]{inputenc}
\usepackage[english]{babel}
\usepackage[T1]{fontenc}
\usepackage{lmodern}
\usepackage[margin=1.0in]{geometry}
\usepackage{amssymb,latexsym}
\usepackage{amsmath,amsthm,amsfonts}
\usepackage[foot]{amsaddr}
\usepackage{mathtools}
\usepackage{enumerate}
\usepackage{hyperref}
\usepackage[capitalize,noabbrev,nameinlink]{cleveref}
\usepackage[numbers,square]{natbib}

\usepackage{graphicx}
\usepackage{subfig}
\usepackage{booktabs}
\usepackage{algorithm}
\usepackage{algpseudocode}

\usepackage{footnote}

\usepackage{xcolor}
\definecolor{ClassicBlue}{HTML}{096191}
\definecolor{VivaMagenta}{HTML}{BB2649}
\colorlet{inlinkcolor}{ClassicBlue!85!green}
\colorlet{exlinkcolor}{VivaMagenta}
\hypersetup{colorlinks=True,
            linkcolor=inlinkcolor,
            citecolor=inlinkcolor,
            urlcolor=exlinkcolor,
            linktoc=page,
            breaklinks=true,
            plainpages=false}

\newtheorem{theorem}{Theorem}[section]

\newtheorem{lemma}[theorem]{Lemma}

\theoremstyle{definition}
\newtheorem{definition}[theorem]{Definition}
\newtheorem{example}[theorem]{Example}
\newtheorem{assumption}[theorem]{Assumption}
\theoremstyle{remark}
\newtheorem{remark}[theorem]{Remark}

\crefname{equation}{}{}
\crefname{assumption}{Assumption}{Assumptions}

\numberwithin{equation}{section}

\usepackage{bm}

\def\calB{{\mathcal{B}}}
\def\calC{{\mathcal{C}}}
\def\calD{{\mathcal{D}}}

\def\calF{{\mathcal{F}}}
\def\calG{{\mathcal{G}}}

\def\calJ{{\mathcal{J}}}
\def\calK{{\mathcal{K}}}

\def\calN{{\mathcal{N}}}
\def\calO{{\mathcal{O}}}

\def\calS{{\mathcal{S}}}

\def\calW{{\mathcal{W}}}

\def\bbE{{\mathbb{E}}}

\def\bbN{{\mathbb{N}}}

\def\bbR{{\mathbb{R}}}

\usepackage{mathrsfs}

\usepackage[mathscr]{euscript}

\DeclareMathOperator*{\argmin}{arg\,min}
\DeclareMathOperator{\dist}{dist} 

\DeclareMathOperator{\lip}{Lip}